\documentclass[reqno,a4paper,10pt]{book}
\usepackage{amssymb,amstext,amsthm,amsfonts,mathrsfs,amsbsy,amsmath,array,bbm,multirow}

\usepackage{slashed,fancyhdr,hojadeestilotesisdoctoral,bbding}

\usepackage{a4wide} % para hacer mas ancho la pagina
\usepackage[english,greek]{babel}
\usepackage[usenames,dvipsnames,svgnames,table]{xcolor}
\usepackage[pdftex]{graphicx}
\usepackage{xparse}
\usepackage{collref} % serves to collect multiple \bibitem references (in the same sequence) into a single \bibitem block the problem with 'collref' is that it blocks the references automatically, and it makes them look less important against ones that were mentioned earlier in the text and thus have are a unique ref
\usepackage{tikz}

\usepackage[a4paper,bindingoffset=0.2in,left=0.6in,right=0.6in,top=1in,bottom=1in]{geometry}

\usepackage{concmath}
\usepackage[T1]{fontenc}

\theoremstyle{plain}
\newtheorem{thm}{Theorem}[section]

\newtheorem{prop}[thm]{Proposition}
\newtheorem{lemma}[thm]{Lemma}
\newtheorem{cor}[thm]{Corollary}

\theoremstyle{definition}
\newtheorem{definition}[thm]{Definition}

\theoremstyle{remark}

\newtheorem{ep}[thm]{Example}

\newcommand{\g}{\ensuremath{\frak{g}}}

\newcommand{\Sh}{\mathrm{Sh}}

\newcommand{\maps}{\colon}
\newcommand{\tensor}{\otimes}
\renewcommand{\deg}[1]{\left \lvert #1 \right \rvert}

\newcommand{\pr}{\mathrm{pr}}
\newcommand{\epi}{\twoheadrightarrow}
\renewcommand{\S}{\bar{S}}
\newcommand{\rDelta}{\bar{\Delta}}
\newcommand{\rdDelta}[1]{\bar{\Delta}^{(#1)}}

\newcommand{\Xham}{\mathfrak{X}_{\mathrm{Ham}}}

\newcommand{\ham}[1]{\Omega^{#1}_{\mathrm{Ham}}\left(M\right)}

\newcommand{\xto}[1]{\xrightarrow{#1}}

\NewDocumentCommand\mycite{mgggg}{\IfNoValueTF{#5}{\IfNoValueTF{#3}{\IfNoValueTF{#2}{\singlecite{#1}}{\singlecitedetail{#1}{#2}}}{\multicite{#1}{#2}{#3}}}{\multimulticite{#1}{#2}{#3}{#4}{#5}}} 
\usepackage[authoryear,round,square,colon,sort&compress,comma]{natbib}
\bibpunct{[}{]}{,}{n}{}{,}

\usepackage[pdftex,unicode,implicit]{hyperref}
\hypersetup{%
  pdftitle    = {Strongly Homotopy Lie Algebras from Multisymplectic Geometry},
  pdfkeywords = {differential geometry, algebraic topology, multisymplectic Geometry, generalized complex geometry}, 
  pdfsubject  = {math.dg, math.sg, hep-th},
  pdfauthor   = {C.S.Shahbazi},
  pdfcreator  = {pdf\LaTeXe\ with package \flqq hyperref\frqq},
  pdffitwindow= false,  % Open the pdf file without the sidebars open.
  bookmarksopen   = false,
  bookmarksopenlevel = 1,
  bookmarksnumbered = true,
  unicode     = true,
  plainpages  = false,
  colorlinks  = true,  % enable to avoid having coloured squares around hyperlinks
  citecolor   = black,
  urlcolor    = black,
  linkcolor   = black,
}

\usetikzlibrary{matrix}
\begin{document}
\setcounter{tocdepth}{1}
\selectlanguage{english}
%\renewcommand{\contentsname}{Contents/ Contenidos}
%%%%%%%%%%%%%%%%%%%%%%%%%%%%%%%%%%%%%%%%%%%%%%%%%%%%%%%%%%%%%%%%%%%%%%%
%%% PORTADA
%%%%%%%%%%%%%%%%%%%%%%%%%%%%%%%%%%%%%%%%%%%%%%%%%%%%%%%%%%%%%%%%%%%%%%%
\thispagestyle{empty}
\phantom{a}
\vspace{2cm}

\vspace*{-1.3cm}
\newcommand{\HRule}{\rule{\linewidth}{0.3mm}}
\setlength{\parindent}{1cm}
\setlength{\parskip}{1mm}
\noindent
\HRule
\begin{center}
%{\fontfamily{Calligra}\selectfont\textbf{Black Holes in Supergravity with Applications to String Theory}}

\setlength{\baselineskip}{3\baselineskip}
\textbf{Strongly Homotopy Lie Algebras from Multisymplectic Geometry}
\HRule
\vspace{4cm}

%{\sizea {\sizeA B}lack holes in{\sizeA S}upergravity and}\\
%{\sizea {\sizeA S}tring {\sizeA T}heory}}

\includegraphics[scale=0.5]{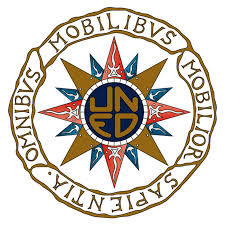}

\vspace{4cm}

%\begin{minipage}{12cm}

%\begin{center}
%\bf{Carlos Shahbazi Alonso} \\ \vspace{0.2cm} \bf{U.N.E.D.} \\
%\end{center}
%\end{minipage}
\begin{center}

\renewcommand{\thefootnote}{\alph{footnote}}

{\bf{C. S.~Shahbazi}\footnotetext{Institut de Physique Th\'eorique, CEA-Saclay. Email: carlos.shabazi [at] cea.fr} \\ \bf{U.N.E.D.}
}
\\
\renewcommand{\thefootnote}{\arabic{footnote}}

\vspace{.5cm}

\vspace{.5cm}

\end{center}

\vspace{2cm}

%\begin{minipage}{12cm}
%\begin{center}
%\emph{A thesis submitted for the master degree in} \\ \emph{Geometry and Topology by the U.N.E.D.} \\ %\emph{June 2013}
%\end{center}
%\end{minipage}

\end{center}
\clearpage

%%%%%%%%%%%%%%%%%%%%%%%%%%%%%%%%%%%%%%%%%%%%%%%%%%%%%%%%%%%%%%%%%%%%%%%
%%% PAGINA VACIA
%%%%%%%%%%%%%%%%%%%%%%%%%%%%%%%%%%%%%%%%%%%%%%%%%%%%%%%%%%%%%%%%%%%%%%%
\thispagestyle{empty}
\phantom{1}

\clearpage

%%%%%%%%%%%%%%%%%%%%%%%%%%%%%%%%%%%%%%%%%%%%%%%%%%%%%%%%%%%%%%%%%%%%%%%
%%% CONTRAPORTADA
%%%%%%%%%%%%%%%%%%%%%%%%%%%%%%%%%%%%%%%%%%%%%%%%%%%%%%%%%%%%%%%%%%%%%%%
\thispagestyle{empty}
\phantom{a}
\begin{center}
{
\renewcommand{\tabcolsep}{2em}
\begin{tabular}{cc}  
\textbf{Universidad Nacional}& \textbf{Consejo Superior de} \\ 
\textbf{de Eduaci\'on a distancia} & \textbf{Investigaciones Cient\'ificas} \\  
%\hline
\includegraphics[scale=0.3]{unedlogo.jpeg} & \includegraphics[scale=0.3]{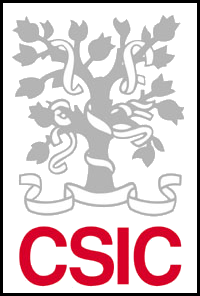} \\
Facultad de Ciencias & Instituto de Ciencias Matem\'aticas\\

\end{tabular}} 
\end{center}
\vspace{2cm}

\vspace{\stretch{5}}
\begin{center}
\setlength{\baselineskip}{2\baselineskip}
\textbf{
{\sizea {\sizeA G}EOMETR\'IA MULTISIMPL\'ECTICA Y $L_{\infty}$-\'ALGEBRAS}}
	
%\vspace{\stretch{4}}
%\setlength{\baselineskip}{0.5\baselineskip}
%{Alberto R. Palomo-Lozano}
%\vspace{\stretch{6}}
\vspace{\stretch{1}}
\setlength{\baselineskip}{0.5\baselineskip}
\vspace{\stretch{6}}

%\begin{minipage}{16cm}
%\begin{center}
%Memoria de Tesis de M\'aster presentada ante la facultad de Ciencias\\
%de la Universidad Nacional de Eduaci\'on a distancia para la obtenci\'on del t\'itulo de M\'aster en Matem\'aticas Avanzadas.
%\end{center}
%\end{minipage}
	
\vspace{\stretch{4}}
Tesis de M\'aster dirigida por:
\\[1ex]
\textbf{Profesor D. Marco Zambon}\\
Profesor contratado Doctor, Universidad Aut\'onoma de Madrid, y miembro del ICMAT
\\[5ex]
%Septiembre, 2013
\end{center}

\newpage
%%%%%%%%%%%%%%%%%%%%%%%%%%%%%%%%%%%%%%%%%%%%%%%%%%%%%%%%%%%%%%%%%%%%%%%
%%% PAGINA VACIA
%%%%%%%%%%%%%%%%%%%%%%%%%%%%%%%%%%%%%%%%%%%%%%%%%%%%%%%%%%%%%%%%%%%%%%%
\thispagestyle{empty}
\phantom{1}

\clearpage

\section*{Acknowledgments}
 
En \'esta t\'esis de m\'aster se recoge parte del trabajo que he realizado como estudiante de m\'aster bajo la direcci\'on del profesor Marco Zamb\'on, a quien estoy profundamente agradecido por todo el tiempo, paciencia y dedicaci\'on que ha invertido en mi. Ha tenido en cuenta desde el primer momento mis circunstancias personales especiales como estudiante de m\'aster y ha hecho todo lo posible por facilitar mi trabajo y aprendizaje. No solo eso, sino que ha me ha ayudado y aconsejado extensamente en multitud de cuestiones matem\'aticas relacionadas con mi \emph{otro} trabajo. Por todo ello, y por introducirme en un \'area de la matem\'atica tan bonita como es la geometr\'ia diferencial, !`gracias Marco!.

Tengo que agradecer tambi\'en a Pablo Bueno, Patrick Meessen y Tom\'as Ort\'in el inter\'es mostrado por el trabajo que en \'estas p\'aginas se recoge: siendo F\'isicos Te\'oricos con una impecable formaci\'on en geometr\'ia diferencial siempre est\'an dispuestos a unirse cualquier debate sobre el tema. Sin duda tambi\'en he aprendido mucho de ellos.
 
Finalmente quiero agradecer a mi familia y amigos, en especial a mi madre que me ha apoyado en todo momento durante mis estudios, la alegr\'ia y compa\~n\'ia que me brindan cada d\'ia.

%%%%%%%%%%%%%%%%%%%%%%%%%%%%%%%%%%%%%%%%%%%%%%%%%%%%%%%%%%%%%%%%%%%%%%%
%%% TABLE OF CONTENTS
%%%%%%%%%%%%%%%%%%%%%%%%%%%%%%%%%%%%%%%%%%%%%%%%%%%%%%%%%%%%%%%%%%%%%%%
%\renewcommand{\thepage}{\Roman{page}}
\addtocontents{toc}{\protect\thispagestyle{empty}}
%\addtocontents{toc}{\vspace{-5ex}}
%\addtocontents{toc}{\contentsline {chapter}{\numberline {}}{}{}}
\tableofcontents
\thispagestyle{empty}
\clearpage

%%%%%%%%%%%%%%%%%%%%%%%%%%%%%%%%%%%%%%%%%%%%%%%%%%%%%%%%%%%%%%%%%%%%%%%
%%%%%%%%%%%%%%%%%%%%%%%%%%%%%%%%%%%%%%%%%%%%%%%%%%%%%%%%%%%%%%%%%%%%%%%
%%%%%%%%%%%%%%%%%%%%%%%%%%%%%%%%%%%%%%%%%%%%%%%%%%%%%%%%%%%%%%%%%%%%%%%
%%%%%%%%%%%%%%%%%%%%%%%%%%%%%%%%%%%%%%%%%%%%%%%%%%%%%%%%%%%%%%%%%%%%%%%

\renewcommand{\thepage}{\arabic{page}}
%\addtocontents{toc}{\contentsline{chapter}{\numberline {\small{PART I}}}{}{}}
\pagestyle{fancy}
\renewcommand{\leftmark}{\MakeUppercase{Chapter \thechapter. Introduction}}
\renewcommand{\headrulewidth}{0.0pt}
\fancyhead[LE]{\thepage}
\fancyhead[RE]{\slshape \leftmark}
\fancyhead[LO]{\slshape \rightmark}
\fancyhead[RO]{\thepage}
\fancyfoot{}

%%%%%%%%%%%%%%%%%%%%%%%%%%%%%%%%%%%%%%%%%%%%%%%%%%%%%%%%%%%%%%%%%%%%%%%
%%% CHAPTER 1: INTRODUCTION
%%%%%%%%%%%%%%%%%%%%%%%%%%%%%%%%%%%%%%%%%%%%%%%%%%%%%%%%%%%%%%%%%%%%%%%
%\renewcommand{\chaptername}{Part I\\ Chapter}

%\renewcommand{\leftmark}{Introduction}
\chapter{Introduction}

Multisymplectic geometry \cite{CILeon,1998math......5040E,JAZ:4974756,Ibort} considers generalizations of symplectic manifolds\footnote{A symplectic manifold is adifferentiable manifold equipped with a closed non-degenerate two-form} which are called $n$-plectic manifolds. A  differentiable manifold is $n$-plectic if it is equipped with a closed non-degenerate, in a particular sense to be explained later, $(n+1)$-form. 

In symplectic geometry, one can equip the space of functions on a symplectic manifold with the structure of a Poisson algebra by means of the symplectic form \cite{Ana,Meinrenken}. It turns out that in multisymplectic geometry, the $n$-plectic structure present on an $n$-plectic manifold gives the structure of Lie-$n$ algebra to a particular complex $L$ constructed out of differential forms on the manifold \cite{2009arXiv0901.4721B,2010arXiv1001.0040R,
2010arXiv1009.2975R,2011arXiv1106.4068R,2012LMaPh.100...29R}. Lie $n$-algebras are particular instances of strongly homotopy Lie algebras \cite{Lada:1994mn}, or $L_{\infty}$-algebras, in which the underlying complex is finite. 

In symplectic geometry it is possible to define a particular kind of Lie group action on the symplectic manifold, such that they preserve the symplectic form and is generated by Hamiltonian vector fields. More precisely, the action of a Lie group $G$, with Lie algebra $\mathfrak{g}$, on a symplectic manifold $\left(\mathcal{M},\omega\right)$ is said to be Hamiltonian if it admits a moment map, that is, a map \cite{Ana,Meinrenken}

\begin{equation}
\mu\maps \mathcal{M}\to \mathfrak{g}^{\ast}\, ,
\end{equation}

\noindent
such that the following conditions are satisfied

\begin{enumerate}

\item For each $v\in\mathfrak{g}$, let

\begin{itemize}

\item $\mu^{v} :\, \mathcal{M}\to \mathbb{R}\, ,$ given by $\mu^{v}(p)\equiv <\mu (p), v >$, where $<\cdot, \cdot >$ is the natural pairing of $\mathfrak{g}$ and $\mathfrak{g}^{\ast}$.

\item $\tilde{v}$ be the vector field generated by the one-parameter subgroup $\left\{e^{t v}\, \,  |\, \, t\in\mathbb{R}\right\}\subseteq G$.

\end{itemize}

Then 

\begin{equation}
d\mu^{v} = -\iota_{\tilde{v}}\omega\, ,
\end{equation}

\noindent
that is, $\mu^{v}$ is a Hamiltonian function for the vector field $\tilde{v}$.

\item $\mu$ is equivariant with respect to the given action and the coadjoint action $\mathrm{Ad}^{\ast}$ of $G$ on $\mathfrak{g}^{\ast}$.

\end{enumerate}

\noindent
In particular, if an action is Hamiltonian then its infinitesimally generated by Hamiltonian vector fields. The notion of Hamiltonian action on a symplectic manifold can be equivalently defined in terms of a comoment map, that is, a Lie algebra homomorphism

\begin{equation}
\label{eq:mmnt}
\mu^{\ast}\maps \mathfrak{g}\to C^{\infty}\left(\mathcal{M}\right)\, ,
\end{equation}

\noindent
such that $d\mu^{\ast}\left(v\right) = -\iota_{v}\omega\, , \quad v\in\mathfrak{\mathcal{M}}$. That is, $\mu^{\ast}\left(v\right)$ is the Hamiltonian vector field of $v$. Remarkably enough, the notion of Hamiltonian action can be also defined for the action of a Lie group on an $n$-plectic manifold by defining the so-called \emph{homotopy moment map} \cite{2013arXiv1304.2051F}, a generalization of the comoment map \eqref{eq:mmnt} for $n$-plectic manifolds. Loosely speaking, it consists on a $L_{\infty}$-morphism 

\begin{equation}
f \maps \mathfrak{g}\to L\, ,
\end{equation}

\noindent
that lifts, in a suitable sense, the map from $\mathfrak{g}$ to the set of Hamiltonian vector fields which is assumed to exist from the outset.

This master thesis is devoted to the study of $L_{\infty}$-morphisms between Lie-$n$ algebras constructed on $n$-plectic manifolds, as well as the study of homotopy moment maps on $n$-plectic manifolds equipped with the Hamiltonian action of a Lie group.

The study of differentiable manifolds equipped with closed non-degenerate forms can be justified from different points of views in mathematics as well as in physics. Standard motivations correspond to the important role that symplectic and multisymplectic manifolds play in classical mechanics, classical field theory and also in the corresponding quantization procedures. However, $n$-plectic manifolds\footnote{Maybe dropping the non-degeneracy condition on the $(n+1)$-form. The corresponding manifold is then called pre-$n$-plectic. Note however that most of the results about $n$-plectic manifolds can be extended in a suitable way to pre-$n$-plectic manifolds.} may be physically relevant on another level: the space-time manifold that describes the universe, at least up to some energy scale, could have the structure of an $n$-plectic manifold.

Such possibility naturally arises in Superstring Theory \cite{Scherk:1974ca,Yoneya:1974jg,
Green:1984sg,Gross:1985fr,Gross:1985rr,Strominger:1996sh,gswst,Greene:1996cy,
Ortin:2004ms,Kiritsis:2007zza,Ibanez:2012zz}, a very promising candidate theory for the quantum description of all the known interactions of nature. Superstring theory implies the existence of several differential forms defined on the space-time manifold, some of them closed, corresponding to field strengths of the Ramond-Ramond and the Neveu-Schwarz Neveu-Schwarz forms of the corresponding Supergravity. Therefore the space-time manifold in Superstring theory is going to be at least a pre-$n$-plectic manifold. The non-degeneracy properties of the forms will depend on the particular solution to be considered.

The space-time that we observe is four-dimensional, yet Superstring theory predicts that the space-time must be ten-dimensional. In order to fix this apparent contradiction, several mechanisms have been proposed in the literature \cite{Kaluza,Klein,Duff:1986hr,Randall:1999ee,Randall:1999vf,Grana:2005jc}. One of them, the Kaluza-Klein reduction \cite{Duff:1986hr}, consists in assuming that the space-time manifold $\mathcal{M}$ is locally the product of a four-dimensional non-compact manifold $\mathcal{M}_{4}$ and a six-dimensional compact manifold $\mathcal{M}_{6}$

\begin{equation}
\mathcal{M} = \mathcal{M}_{4}\times \mathcal{M}_{6}\, ,
\end{equation}

\noindent
\emph{small enough} to not be accessible in current high-energy experiments. Superstring theory constrains the different manifolds $\mathcal{M}_{6}$ that we can consider as compact manifolds \cite{Grana:2005jc}. In particular, for supersymmetric compactifications, the existence of one or several globally defined spinors on the compact manifold implies the existence of globally defined forms, which, depending on the details of the compactification, may be closed and non-degenerate. As an example, we can consider M-theory \cite{Witten:1995ex}, closely related to Superstring Theory, which is a theory that predicts the space-time manifold to be eleven-dimensional. The fluxless compactification of such theory on a seven dimensional compact manifold $\mathcal{M}_{7}$ implies that $\mathcal{M}_{7}$ must be a manifold of $G_{2}$-holonomy \cite{Duff:1986hr,Acharya:2003ii,Behrndt:2003uq,House:2004pm}. Therefore, it has a globally defined, closed and non-degenerate three-form \cite{Joyce} and thus it is a two-plectic manifold. 

It is worth pointing out that the interpretation of the Lie-$n$ algebras associated to the space-time manifold or the compactification manifold is not known, and it would be interesting to find out if it encodes any physical information about the theory itself. Notice that $L_{\infty}$-algebras have appeared already in Superstring Theory and Supergravity. For example, the algebra of states in the Fock space of closed String Field Theory is a strongly homotopy Lie algebra \cite{Lada:1992wc}. For more applications of $L_{\infty}$-algebras to Superstring Theory and Supergravity the interested reader may consult \cite{D'Auria:1982nx,Castellani:1991et,Castellani:1991eu,Castellani:1991ev,Sorokin:1997ps,Sati:2008eg,
Baez:2009xt,Baez:2010ye,Ritter:2013wpa}. It is clear then that multisymplectic geometry and $L_{\infty}$-algebras play an important role in theoretical physics and in particular in Superstring Theory and Supergravity, and thus more effort is needed in order to uncover the role that these mathematical structures play in the theories that describe the fundamental interactions of nature.

{\bf The outline of this work is as follows}. {\bf In chapter \ref{chapter:mathpreliminaries}} we introduce the relevant background material necessary for the rest of the paper, which includes graded algebra theory and homological algebra theory, fibre bundles and basics of Lie groups and symplectic geometry. It is intended for non-experts, perhaps interested physicists, and therefore can be skipped by experts. {\bf In chapter \ref{chapter:linfty}} we introduce $L_{\infty}$-algebras and define $L_{\infty}$-morphisms in an independent way, not related yet to multisymplectic geometry, giving explicit formulae relating $L_{\infty}[1]$-algebras and $L_{\infty}$-algebras. {\bf Chapter \ref{chapter:multisymplectic}} contains the new results present in this thesis. We first introduce $n$-plectic manifolds and connect them to $L_{\infty}$-algebras. Then we introduce, closely following \cite{2013arXiv1304.2051F}, the concept of homotopy moment map. {\bf Sections \ref{sec:multidiff} and \ref{sec:productman}} contains the new material present in this work. In section \ref{sec:multidiff} we obtain specific conditions under which two $n$-plectic manifolds with strictly isomorphic Lie-$n$ algebras are symplectomorphic. Finally, in section \ref{sec:productman}, we study the construction of an homotopy moment map for a product manifold assuming that the factors are $n$-plectic manifolds equipped with the corresponding homotopy moment maps.

\cleardoublepage

%%%%%%%%%%%%%%%%%%%%%%%%%%%%%%%%%%%%%%%%%%%%%%%%%%%%%%%%%%%%%%%%%%%%%%%
%%% CHAPTER 2: Mathematical preliminaries
%%%%%%%%%%%%%%%%%%%%%%%%%%%%%%%%%%%%%%%%%%%%%%%%%%%%%%%%%%%%%%%%%%%%%%%
%\renewcommand{\chaptername}{Part I\\ Chapter}

%\renewcommand{\leftmark}{Mathematical preliminaries}
\chapter{Mathematical preliminaries}
\label{chapter:mathpreliminaries}

The goal of this chapter is to introduce the relevant basic background that will be needed through the rest of the paper. The content is elementary and therefore can be skipped by expert readers, although it may be useful for interested non-mathematicians. Basic references for this chapter are \cite{JMLee,DHbundle,WarnerKoba,KN,Ana,Meinrenken,2013arXiv1304.2051F,Lada:1994mn,tensoranalisis}. We will consider exclusively real vector spaces, real manifolds and real bundles over them, unless otherwise stated.

%%%%%%%%%%%%%%%%%%%%%%%%%%%%%%%%%%%%%%%%%%%%%%%%%%%%%%%%%%%%%%%%%%%%%%%%%%%%%%%%%%%%%%%%%%%%%%%%%%%
%%%%%%%%%%%%%%%%%%%%%%%%%%%%%%%%%%%%%%%%%%%%%%%%%%%%%%%%%%%%%%%%%%%%%%%%%%%%%%%%%%%%%%%%%%%%%%%%%%%
%%%%%%%%%%%%%%%%%%%%%%%%%%%%%%%%%%%%%%%%%%%%%%%%%%%%%%%%%%%%%%%%%%%%%%%%%%%%%%%%%%%%%%%%%%%%%%%%%%%
%%%%%%%%%%%%%%%%%%%%%%%%%%%%%%%%%%%%%%%%%%%%%%%%%%%%%%%%%%%%%%%%%%%%%%%%%%%%%%%%%%%%%%%%%%%%%%%%%%%

\section{Graded and homological algebra theory}
\label{sec:hlahomalgebra}

%%%%%%%%%%%%%%%%%%%%%%%%%%%%%%%%%%%%%%%%%%%%%%%%%%%%%%%%%%%%%%%%%%%%%%%%%%%%%%%%%%%%%%%%%%%%%%%%%%%
%%%%%%%%%%%%%%%%%%%%%%%%%%%%%%%%%%%%%%%%%%%%%%%%%%%%%%%%%%%%%%%%%%%%%%%%%%%%%%%%%%%%%%%%%%%%%%%%%%%
%%%%%%%%%%%%%%%%%%%%%%%%%%%%%%%%%%%%%%%%%%%%%%%%%%%%%%%%%%%%%%%%%%%%%%%%%%%%%%%%%%%%%%%%%%%%%%%%%%%
%%%%%%%%%%%%%%%%%%%%%%%%%%%%%%%%%%%%%%%%%%%%%%%%%%%%%%%%%%%%%%%%%%%%%%%%%%%%%%%%%%%%%%%%%%%%%%%%%%%

The purpose of this section is to introduce the basic elements of graded algebra theory and Homological algebra that we will need through the rest of the letter. Graded algebra theory refers to the study of algebraic structures in a graded vector space. Higher  Homological algebra theory studies homology (or co-homology) in an abstract setting, that is, on abstractly defined complexes (or co-complexes). I have tried to recollect here the basic definitions and constructions that are well known to experts but that might be difficult to find in a clear way on a first approximation to the topic. 

%%%%%%%%%%%%%%%%%%%%%%%%%%%%%%%%%%%%%%%%%%%%%%%%%%%%%%%%%%%%%%%%%%%%%%%%%%%%%%%%%%%%%%%%%%%%%%%%%%%
%%%%%%%%%%%%%%%%%%%%%%%%%%%%%%%%%%%%%%%%%%%%%%%%%%%%%%%%%%%%%%%%%%%%%%%%%%%%%%%%%%%%%%%%%%%%%%%%%%%

\subsection{Higher algebras}
\label{sec:higheralgebra}

%%%%%%%%%%%%%%%%%%%%%%%%%%%%%%%%%%%%%%%%%%%%%%%%%%%%%%%%%%%%%%%%%%%%%%%%%%%%%%%%%%%%%%%%%%%%%%%%%%%
%%%%%%%%%%%%%%%%%%%%%%%%%%%%%%%%%%%%%%%%%%%%%%%%%%%%%%%%%%%%%%%%%%%%%%%%%%%%%%%%%%%%%%%%%%%%%%%%%%%

We begin with a series of definitions increasing step by step the structures involved. We begin by the simple definition of an algebra.

\definition{\label{def:algebra} A real algebra $(V,\cdot)$ is a real vector space $V$ space equipped with a bilinear product $V\times V\to V$, denoted by $\cdot$ or concatenation of elements. }

\noindent
We require the product $\cdot$ of an algebra to be inner. It may have, however, other properties. For example, if $x_{1}\cdot x_{2} = x_{2}\cdot x_{1}$ for all $x_{1}, x_{2} \in V$ the algebra $\left(V,\cdot\right)$ is said to be commutative. If $x_{1}\cdot ( x_{2}\cdot x_{3}) = ( x_{1}\cdot x_{2} )\cdot x_{3}$ for all $x_{1}, x_{2}, x_{3} \in \mathrm{X}$  then the algebra is said to be associative. If the algebra $V$ contains an identity, that is, an element $e\in V$ such that $e\cdot x=x\cdot e = x$ for all $x\in V$ then it is said to be unital. A real algebra such that the product is associative and has an identity is therefore a ring that is also a vector space, and it is called a \emph{unital associative algebra}.

\definition{\label{def:gradedvectorspace} Let $G$ be an abelian group. A $G$-\emph{graded vector space} $V$ over $\mathbb{R}$ is a collection $\left(V_{g}\right)_{g\in G}$ of vector spaces over $\mathbb{R}$.}

\noindent
The homogeneous elements of degree $g\in G$ of a $G$-graded vector space $V$ are the elements of $V_{g}$. In this work we will consider exclusively $G=\mathbb{Z}$-graded vector spaces $V=\bigoplus_{i\in\mathbb{Z}} V_{i}$. The grade of an homogeneous element $x\in V$ is denoted by $|x|$, and of course, since we will consider only $\mathbb{Z}$-graded vector spaces, we will always have $|x|\in\mathbb{Z}$. 

\definition{\label{def:gradedalgebra} A real graded algebra is a real graded vector space $V=\bigoplus_{i\in\mathbb{Z}} V_{i}$ equipped with a bilinear product $V\times V\to V$ which will be denoted by $\cdot$ or concatenation of elements, such that
 
\begin{equation}
V_{i}\cdot V_{j}\subset V_{i+j}\, .
\end{equation}

}

\definition{\label{def:gradedliealgebra} A real graded Lie algebra is a real graded algebra $(V,\cdot)$ equipped with a bilinear product $[\cdot,\cdot]: V\times V\to V$ such that the following axioms are satisfied

\begin{enumerate}

\item $[\cdot,\cdot]$ respects the grading of $V$, that is, $[V_i,V_j]\subseteq V_{i+j}$.

\item If $x_{1}, x_{2} \in V$ are homogeneous elements then

\begin{equation}
\label{eq:antsymbracket}
[x_{1},x_{2}]=-(-1)^{|x_{1}||x_{2}|}\,[x_{2},x_{1}]\, .
\end{equation}

\noindent
That is, $[\cdot,\cdot]$ is antisymmetric in the graded sense.

\item If $x_{1}, x_{2}, x_{3} \in V$ then 

\begin{equation}
\label{eq:gradedjacobi}
(-1)^{|x_{1}||x_{3}|}[x_{1},[x_{2},x_{3}]]+(-1)^{|x_{1}||x_{2}|}[x_{2},[x_{3},x_{1}]]+(-1)^{|x_{2}||x_{2}|} [x_{3},[x_{1},x_{2}]] = 0\, .
\end{equation}

\noindent
That is, $[\cdot,\cdot]$ satisfies the graded Jacobi identity.
\end{enumerate}

}

\noindent
Equation (\ref{eq:antsymbracket}) is just the graded version of the antisymmetric Lie bracket of a not-graded Lie algebra. Analogously, Eq. (\ref{eq:gradedjacobi}) is just the graded version of the Jacobi identity for not-graded Lie algebras. When $V$ is concentrated in degree zero, a real graded Lie algebra is just an ordinary real Lie algebra. Let us consider now a simple example taken from physics.

\begin{ep}\label{ep:susyalgebra} {\bf Supersymmetry algebra}. The Super-Poincar\'e algebra $\mathfrak{sp}$ is, in physical terms, an extension of the Poincar\'e algebra by fermionic generators that obey anti-commutation relations. We denote by $M_{ab}$ the generators of the Lorentz group, by $P_{a}$ the generators of translations and by $Q_{\alpha}$ the fermionic generators, which are $\mathrm{Spin}\left(1,3\right)$ spinors of definite chirality. They obey the following (anti)-commutation relations

\begin{eqnarray}
\label{eq:poincare1}
\left[ M_{ab}, M_{cd}\right] &=& -M_{eb} \Gamma_{\mathrm{v}}\left(M_{cd}\right)^{e}_{a}-M_{ae} \Gamma_{\mathrm{v}}\left(M_{cd}\right)^{e}_{b} \\
\label{eq:poincare2}\left[ P_{a}, M_{cd}\right] &=&-P_{e} \Gamma_{\mathrm{v}}\left(M_{cd}\right)^{e}_{a}\\
\label{eq:fermionpoincare}
\left[ Q^{\alpha}, M_{ab}\right] &=& \Gamma_{\mathrm{s}}\left(M_{ab}\right)^{\alpha}_{\beta} Q^{\beta}\\
\label{eq:fermion}
\left\{ Q^{\alpha}, Q^{\beta}\right\} &=& i \left(\gamma^{a} C^{-1}\right)^{\alpha\beta} P_{a}\, ,
\end{eqnarray}

\noindent
where $\Gamma_{\mathrm{v}}$ denotes the vectorial representation, $\Gamma_{\mathrm{s}}$ denotes the spinorial representation and $C$ is the charge conjugation matrix. $\left\{\cdot,\cdot\right\}$ is defined as follows $\left\{ Q^{\alpha},Q^{\beta}\right\} = Q^{\alpha} Q^{\beta} + Q^{\beta}Q^{\alpha}$. The Super-Poincar\'e algebra gets beautifully described as a particular instance of $\mathbb{Z}_{2} = \mathbb{Z}/2\mathbb{Z}$\footnote{The definition of $\mathbb{Z}_{2}$-graded space can be obtained from the definition of $\mathbb{Z}$-graded vector space by simply substituting $\mathbb{Z}$ by $\mathbb{Z}_{2}$.} graded Lie algebra $\mathfrak{sp} = \mathfrak{sp}_{0}\oplus \mathfrak{sp}_{1}$ with bracket $[\cdot ,\cdot ]_{\mathbb{Z}_{2}}$. The elements of $\mathfrak{sp}_{0}$ are called even and correspond to the \emph{bosonic} generators $M_{ab}$ and $P_{a}$ of the algebra. The elements of $\mathfrak{sp}_{1}$ are called odd and correspond to the \emph{fermionic} generators of the algebra $Q_{\alpha}$. In particular, we have

\begin{eqnarray}
\left[x_{1},x_{2}\right]_{\mathbb{Z}_{2}} &=&  -\left[x_{2},x_{1}\right]_{\mathbb{Z}_{2}}\, ,\qquad\forall\,\, x_{1}, x_{2} \in \mathfrak{sp}_{0}\label{eq:commu12}\, ,\\
\left[x_{1},x_{2}\right]_{\mathbb{Z}_{2}} &=& -\left[x_{2},x_{1}\right]_{\mathbb{Z}_{2}}\, ,
\qquad\forall\,\, x_{1} \in \mathfrak{sp}_{0}\, , \,\forall\,\, x_{2}\in \mathfrak{sp}_{1}\, ,\label{eq:anticommu12}\\
\left[x_{1},x_{2}\right]_{\mathbb{Z}_{2}} &=& \left[x_{2},x_{1}\right]_{\mathbb{Z}_{2}}\, ,\qquad\forall\,\, x_{1}, x_{2}\in \mathfrak{sp}_{1}\, .\label{eq:ac12}
\end{eqnarray}

\noindent
Therefore, equation (\ref{eq:commu12}) corresponds to the commutators (\ref{eq:poincare1}) and (\ref{eq:poincare2}) in the Super-Poincar\'e algebra, equation (\ref{eq:anticommu12}) corresponds to the commutator (\ref{eq:fermionpoincare}) in the Super-Poincar\'e algebra, and equation (\ref{eq:ac12}) corresponds to the anti-commutator (\ref{eq:fermion}) in the Super Poincar\'e algebra. 

\end{ep}

\definition{\label{def:gradedderivation} Let $(V,\cdot)$ be a graded algebra. A homogeneous linear map $d: V\to V$  of grade $|d|$ on $V$ is called a homogeneous derivation if $d(x_{1}\cdot x_{2})=d(x_{1})x_{2}+\epsilon^{|x_{1}||d|}x_{1}\cdot d(x_{2})$, where $\epsilon = \pm 1$ and $x_{1}, x_{2}\in V$ are homogeneous elements. A graded derivation is sum of homogeneous derivations with the same $\epsilon$. In the context of graded algebra, the choice $\epsilon = -1$ is the most natural one, since it takes into account the graded structure of the algebra.

}

%\definition{\label{def:difgradedaalgebra} A real differential graded algebra is a real graded algebra $(V,\cdot)$ equipped with a map $d:V\to V$ which is either degree 1 (cochain complex convention) or degree -1 (chain complex convention) that satisfies two conditions

%\begin{enumerate}

%\item $d \circ d = 0$. Therefore $d$ gives $V$ the structure of a chain ($|d|=-1$) or cochain complex ($|d|=1$).

%\item $d(x_{1} \cdot x_{2})=(dx_{1}) \cdot x_{2} + (-1)^{|x_{1}|} x_{1} \cdot (dx_{2})$. The maps $d$ respects the graded Leibniz rule with $\epsilon = -1$ and it is therefore a derivation. This together with the first condition make $d$ a differential.

%\end{enumerate}

%}

\definition{\label{def:difgradedaliealgebra} A real differential graded Lie algebra is a real graded Lie algebra $(V,\cdot)$ equipped with a degree $\pm 1$ (depending on chain or cochain complex convention) derivation $d:V\to V$ that satisfies 

\begin{enumerate}

\item $d \circ d = 0$. Therefore $d$ gives $V$ the structure of a chain ($|d|=-1$) or cochain complex ($|d|=1$).

\item $d [x_{1}, x_{2}] = [dx_{1}, x_{2}] + (-1)^{|x_{1}|}[x_{1}, dx_{2}]$, where $x_{1}$ and $x_{2}$ are homogeneous elements of $V$. 

\end{enumerate}

}

\noindent
Given two homogeneous elements $x_{1}, x_{2} \in V$ of an arbitrary graded algebra $\left(V,\cdot\right)$, in principle $x_{1}\cdot x_{2}$ and $x_{2}\cdot x_{1}$ are different elements of $\left(V,\cdot\right)$ not related in any particular way. However, if for every pair of homogeneous elements $x_{1}, x_{2} \in V$ the following holds 

\begin{equation}
x_{1}\cdot x_{2} = (-1)^{|x_{1}||x_{2}|} x_{2}\cdot x_{1}\, , 
\end{equation}

\noindent
then $\left(V,\cdot\right)$ is said to be a graded commutative algebra. Analogously, if 

\begin{equation}
x_{1}\cdot x_{2} = -(-1)^{|x_{1}||x_{2}|} x_{2}\cdot x_{1}\, , 
\end{equation}

\noindent
holds, then the algebra $\left(V,\cdot\right)$ is said to be a graded anti-commutative algebra. This procedure can be generalized by defining the \emph{Koszul sign}. Let $x_{1},\dots , x_{n}$ be elements of a symmetric graded algebra $\left(V,\cdot\right)$ and $\sigma \in \Sigma_n$ a permutation. The Koszul sign $\epsilon(\sigma)=\epsilon(\sigma ; x_{1},\hdots,x_{n})$ is defined by the equality

\[
x_{1}  \cdots  x_{n} = \epsilon(\sigma ;
x_{1},\hdots,x_{n}) x_{\sigma(1)} \cdots  x_{\sigma(n)}\, ,
\]

\noindent
which holds in the free graded commutative algebra generated by $V$, with product denoted by concatenation of elements. The Koszul sign can be equivalently defined using an antisymmetric graded algebra $\left(V,\cdot\right)$ as follows

\[
x_{1}  \cdots  x_{n} = (-1)^{\sigma} \epsilon(\sigma ;
x_{1},\hdots,x_{n}) x_{\sigma(1)} \cdots  x_{\sigma(n)}\, ,
\]

\noindent
Given $\sigma \in \Sigma_n$, $(-1)^{\sigma}$ denotes the usual sign of a permutation. Please notice  that $\epsilon(\sigma)$ does not include the sign $(-1)^{\sigma}$. For example, given $x_{1}, x_{2}, x_{2} \in V$, where $\left(V,\cdot\right)$ is free graded commutative algebra, we have

\begin{equation}
x_{1}\cdot x_{2}\cdot x_{3} = \epsilon(3,2,1;x_{1},x_{2}.x_{3}) x_{3} \cdot x_{1}\cdot x_{2}\, ,\qquad \epsilon(3,2,1;x_{1},x_{2}.x_{3}) = (-1)^{|x_{3}||x_{2}|+|x_{3}||x_{1}|}\, . 
\end{equation}

\noindent
We say that $\sigma \in \Sigma_{p+q}$ is a $\mathbf{(p,q)}$-\emph{unshuffle}
if and only if $\sigma$ is a permutation of a set of $(p+q)$ elements such that $\sigma(1) < \cdots < \sigma(p)$ and  $\sigma(p+1) < \cdots <  \sigma(p+q)$.  The set of
$(p,q)$-unshuffles is denoted by $\mathrm{Sh}(p,q)$. For example, $\mathrm{Sh}(2,1)$
is the set of cycles $\{ (1), (23), (123) \}$.

If $V$ is a graded vector space, then ${\bf s}V$ denotes the suspension of $V$, and ${\bf s}^{-1}V$ denotes the desuspension of $V$, defined respectively by

\begin{equation}
\left({\bf s}V\right)_{i} = V_{i-1}\, ,\qquad \left({\bf s}^{-1}V\right)_{i} = V_{i+1}\, .
\end{equation}

\noindent
Another very used notation for the (de)suspension of a graded vector space $V$ is 

\begin{equation}
{\bf s}^{\pm k}V = V[\mp k]\, .
\end{equation}

\noindent
Sometimes we will write $s^{\pm 1}x$ where $x$ is an homogeneous element of a given graded vector space $V$. The meaning of such expression can be understood as follows. Let us assume that $x\in V_{i}\, ,\,\, i\in \mathbb{Z}$. Then we have

\begin{equation}
s^{\pm 1}x \in s^{\pm 1}\left( V_{i}\right) = \left( s^{\pm 1}V\right)_{i\mp 1}\, ,
\end{equation}

\noindent
since $\left( s^{\pm 1}V\right)_{i\mp 1} = V_{i}$. 

\definition{A morphism $f$ from a $\mathbb{Z}$-graded vector space $V$ to a $\mathbb{Z}$-graded vector space $W$ is a collection of linear maps $\left(f_{i}:V_{i}\to W_{i}\right)_{i\in\mathbb{Z}}$.} 

\noindent
Hence, it is implicitly assumed in the definition that a morphism preserves the grading, that is, that $|f(x)| = |x| \in\mathbb{Z}$. It is said then that $f$ is homogeneous of degree zero. An isomorphism of graded vector spaces is a morphism $f$ whose components $f_{i}$ are isomorphisms of vector spaces. It is possible, of course, to consider maps that do not preserve the grading. The degree of a map $f$ is denoted by $|f|$. If a map $f: V\to W$ of graded vector spaces has degree $|f|=k\, ,\,\,k\in\mathbb{Z}$, then we have 

\begin{equation}
f\left(V_{i}\right)\subseteq V_{i+k}\, ,\,\, \forall\,\, i\in\mathbb{Z}\, . 
\end{equation}

\noindent
The set of all morphisms from a graded vector space $V$ to a graded vector space $W$ is denoted by $\mathrm{Hom}\left(V,W\right)$. Notice that $\mathrm{Hom}\left(V,W\right)$ is itself a graded vector space with homogeneous component $\mathrm{Hom}\left(V,W\right)_{i}\,\,\, i\in\mathbb{Z}$ given by the set of components $f_{i}\maps V_{i}\to W_{i}\, ,\,\, f\in\mathrm{Hom}\left(V,W\right)$.  

\definition{\label{def:tensorproduct} Let $F(V\times W)$ the free vector space over $\mathbb{R}$ whose generators are the points of $V\times W$, where $\times$ stands for the Cartesian product \footnote{The Cartesian product $V \times W$ is the set of pairs $(v, w)\, , v \in V\, , w \in W$, which is itself a vector space, although here it is considered merely as a set.} The tensor product is defined as a certain quotient vector space of $F(V \times W)$. Consider the subspace $R$ of $F(V \times W)$ generated by the following elements\footnote{To simplify the notation, we denote by $(v,w)$ the element $1\cdot (u,w) \in F(V \times W)$.}

\begin{eqnarray}
(v_1, w) + (v_2, w) - (v_1 + v_2, w)\, ,\nonumber\\ 
(v, w_1) + (v, w_2) - (v, w_1 + w_2)\, ,\nonumber\\ 
c \cdot (v, w) - (cv, w)\, , \nonumber\\ 
c \cdot (v, w) - (v, cw)\, ,
\end{eqnarray}

\noindent
where $v, v_1, v_2 \in V\, , \,\,\, w, w_1, w_2 \in W\, ,$ and $c\in\mathbb{R}$. The tensor product is then defined as the vector space

\begin{equation}
V \otimes W \equiv F(V \times W) / R\, .
\end{equation}

}

\noindent
The tensor product of two vectors $v$ and $w$ is denoted by the equivalence class $v\otimes w\in ((v,w) + R)$, where $v\otimes w\in V \otimes W$. The principal effect of taking the quotient by R in the free vector space is that the following equations hold in $V \otimes W$

\begin{eqnarray}
(v_1 + v_2) \otimes w &=& v_1 \otimes w + v_2 \otimes w\, , \nonumber\\ 
v \otimes (w_1 + w_2) &=& v \otimes w_1 + v \otimes w_2\, ,\nonumber\\ 
cv \otimes w &=& v \otimes cw = c(v \otimes w)\, .    
\end{eqnarray}     
    
\noindent    
The tensor product of two graded vector spaces $V$ and $W$ is defined to be another graded vector space $V\otimes W$ with grading

\begin{equation}
\left(V\otimes W\right)_{i} = \bigoplus_{i=j+k} V_{j}\otimes W_{k}\, , \qquad i=j+k\, .
\end{equation}

\noindent
The tensor product can be also applied to morphisms $f, g\in\mathrm{Hom}\left(V,W\right)$, and it is given by

\begin{equation}
\left(f\otimes g\right)\left(v\otimes w\right) = (-1)^{|g||f|}f(v)\otimes g(w)\, .
\end{equation}

\definition{\label{def:gradedtensoralgebra} Let $V$ be a graded vector space. The \emph{tensor algebra} $\mathcal{T}\left(V\right)$ is the graded vector space given by the collection of vector spaces

\begin{equation}
\mathcal{T}\left( V\right)_{m} = \bigoplus_{k\geq 0}\bigoplus_{j_{1}+\dots + j_{k} = m} V_{j_{1}}\otimes\dots\otimes V_{j_{k}}\, ,\qquad m\in\mathbb{Z}\, .
\end{equation}

\noindent
For $k=0$ the corresponding summand is set to be equal to $\mathbb{R}$.
 
}

\noindent
Every component $\mathcal{T}\left(V\right)_{m}$ can be decomposed with respect to the tensor product degree $\otimes$ as follows

\begin{equation}
\mathcal{T}^{k}\left( V\right)_{m}\equiv \mathcal{T}\left( V\right)_{m}\cap V^{\otimes k}\, ,\qquad k\geq 1\, .
\end{equation}

\noindent
For $k=0$ we have $\mathcal{T}^{0}\left(V\right) = \mathbb{R}$. The degree of an element $x_{1}\otimes\dots\otimes x_{k}\in V_{1}\otimes\dots\otimes V_{k}$ is defined as $|x_{1}\otimes\dots\otimes x_{k}| = \sum^{k}_{i=1}|x_{i}|$.

%\ep{\label{ep:gradedtensoralgebra} Let $V = V_{0}\oplus V_{1}$ a $\mathbb{Z}/\mathbb{Z}_{2} = \left\{[0],[1]\right\}$ graded vector space.}

\noindent
The vector space $\mathcal{T}^{k}\left( V\right)_{m}$ carries two natural actions, \emph{even} and \emph{odd}, of the group $\Sigma_{k}$ of permutations of a set of $k$ elements. The \emph{even} representation intuitively corresponds to the elements of $\mathcal{T}^{k}\left( V\right)_{m}$ symmetric in the graded sense. It is defined by

\begin{equation}
\sigma\cdot \left( x_{1}\otimes\dots\otimes x_{k} \right)\equiv (-1)^{|x_{i}||x_{(i+1)}|} x_{1}\otimes\dots\otimes x_{(i+1)}\otimes x_{i}\otimes\dots\otimes x_{k}\, ,
\end{equation} 

\noindent
where $\sigma$ is the transposition of the $i$'th and the $(i+1)$'th element. Similarly, the odd representation intuitively corresponds to the antisymmetric elements of $\mathcal{T}^{k}\left( V\right)_{m}$ in the graded sense. It is defined by

\begin{equation}
\sigma\cdot\left(x_{1}\otimes\dots\otimes x_{k}\right)\equiv -(-1)^{|x_{i}||x_{(i+1)}|} x_{1}\otimes\dots\otimes x_{(i+1)}\otimes x_{i}\otimes\dots\otimes x_{k}\, ,
\end{equation} 

\noindent
where $\sigma$ is the transposition of the $i$'th and the $(i+1)$'th element. Notice that the even as well as the odd representations can be defined for an arbitrary element $\sigma\in\Sigma_{k}$ as follows

\begin{equation}
\sigma\cdot \left( x_{1}\otimes\dots\otimes x_{k} \right)\equiv \epsilon\left(\sigma ; x_{1},\cdots , x_{k}\right) x_{\sigma(1)}\otimes\dots\otimes x_{\sigma(k)}\, ,
\end{equation} 

\noindent
for the even representation and

\begin{equation}
\sigma\cdot \left( x_{1}\otimes\dots\otimes x_{k} \right)\equiv (-1)^{\sigma}\epsilon\left(\sigma ; x_{1},\cdots , x_{k}\right)  x_{\sigma(1)}\otimes\cdots\otimes x_{\sigma(k)}\, ,
\end{equation} 

\noindent
for the odd representation.

\definition{\label{def:symmtensor} Given a graded vector space $V$, the \emph{graded symmetric algebra} $\mathcal{S}(V)$ of $V$ is the graded vector space whose elements are the invariants or the even representation of $\Sigma$ on $\mathcal{T}(V)$ with the inherited grading.}

\definition{\label{def:antisymmtensor} Given a graded vector space $V$, the \emph{graded antisymmetric algebra} $\Lambda(V)$ of $V$ is the graded vector space whose elements are the invariants or the odd representation of $\Sigma$ on $\mathcal{T}(V)$ with the inherited grading.}

\noindent
We will denote by $\mathcal{S}^{k}\left( V\right)$ and $\Lambda^{k}\left( V\right)$ the homogeneous elements of degree $k$ of $\mathcal{S}\left( V\right)$ and $\Lambda\left( V\right)$ respectively.

We can extend the action of ${\bf s}^{\pm 1}$ to the tensor product of an arbitrary number of graded vector spaces $V_{1},\dots ,V_{k}$ as follows

\begin{eqnarray}
\label{eq:sk}
{\bf s}^{\pm k} : &&V_{1}\otimes\dots\otimes V_{k}\to {\bf s}^{\pm 1}V_{1}\otimes\dots\otimes {\bf s}^{\pm 1}V_{k}\nonumber\\
&&x_{1}\otimes\dots\otimes x_{k}\mapsto (-1)^{\sum^{k}_{i=1} (k-i)|x_{i}|}  {\bf s}^{\pm 1}x_{1}\otimes\dots\otimes {\bf s}^{\pm 1}x_{k}\, .
\end{eqnarray}

\noindent
Notice that the sign $(-1)^{\sum^{k}_{i=1} (k-i)|x_{i}|}$ can be understood by considering $s^{\pm 1}$ as an odd element which requires the introduction of the sign $(-1)^{|x_{i}|}$ every time it jumps over $x_{i}$ when acting on $x_{1}\otimes\cdots\otimes x_{k}$. $s^{\pm k}$ is an isomorphism of vector spaces which is not an isomorphism of graded vector spaces, since it does not preserve the grading. Defining now $\mathrm{dec}_{k}:V^{\otimes^{k}}\to V^{\otimes^{k}}$ by

\begin{equation}
\mathrm{dec}_{k}\left(x_{1}\otimes\dots\otimes x_{k}\right)\equiv (-1)^{\sum^{k}_{i=1} (k-i)|x_{i}|} x_{1}\otimes\dots\otimes x_{k}\, ,
\end{equation}

\noindent
we obtain the so-called \emph{d\'ecalage-isomorphism} between ${\bf s}^{\pm k}\mathcal{S}^{k}\left( V\right)$ and $\Lambda^{k}\left( {\bf s}^{\pm 1} V\right)$. The d\'ecalage-isomorphism preserves the grading and therefore ${\bf s}^{\pm k}\mathcal{S}^{k}\left( V\right)$ and $\Lambda^{k}\left( {\bf s}^{\pm 1} V\right)$ are isomorphic not only as vector spaces but also as graded vector spaces. 

%\begin{ep}
%\label{ep:S2Lambda2} As an example, let us check now the d\'ecalage-isomorphism in the case of a trivially graded vector space $V=V_{-1}$ and $k=2$. Therefore we want to see how ${\bf s}^{- 2}\mathcal{S}^{2}\left( V\right)\simeq \Lambda^{2}\left( {\bf s}^{- 1} V\right)$. $\Lambda^{2}\left( {\bf s}^{- 1} V\right)$ can be understood as the space of antisymmetric bilinear  from $V\times V$ to $\mathbb{R}$. Analogously, $\mathcal{S}^{2}\left( V\right)$ can be understood as the space of symmetric bilinear applications from $V\times V$ to $\mathbb{R}$. Given $l\in \Lambda^{2}\left( {\bf s}^{- 1} V\right)$ and $x_{1}, x_{2} \in V = V_{-1}$ we have

%\begin{equation}
%l\left( s^{-1} x_{1}, s^{-1} x_{2}\right) = (-1)^{|x_{1}|} s^{-2} l\left( x_{1}, x_{2}\right) = s^{-2} m\left( x_{1}, x_{2}%\right)\, .
%\end{equation}

%\noindent
%We can check now that $m\left(\cdot, \cdot\right)\maps V\times V \to V$ is indeed a symmetric bilinear map

%\begin{equation}
%m\left( x_{1}, x_{2}\right) = (-1)^{|x_{1}|} l\left( x_{1}, x_{2}\right) = -(-1)^{|x_{1}|+|x_{1}||x_{2}|} l\left( x_{2}, x_{1}\right) = (-1)^{|x_{1}||x_{2}|} m\left( x_{2}, x_{1}\right)\, ,
%\end{equation}

%\noindent
%and thus there is a one-to-one correspondence between elements of $\Lambda^{2}\left( {\bf s}^{- 1} V\right)$ and elements of $%%\mathcal{S}^{2}\left( V\right)$.

%\end{ep}

%\noindent
Since it will be useful later, we will rewrite now the Koszul sign $\epsilon\left(\sigma ; s^{-1} x_{1},\hdots, s^{-1} x_{n}\right)$ in terms of $\epsilon\left(\sigma ;  x_{1},\hdots, x_{n}\right)$. In order to do so, we notice that if $\left(x_{1},\hdots,x_{k}\right)\in \Lambda^{k} V$ then

\begin{equation}
\left(x_{1},\hdots,x_{k}\right) = (-1)^{\sigma} \epsilon\left(\sigma ;  x_{1},\hdots, x_{k}\right) \left(x_{\sigma (1)},\hdots,x_{\sigma (k)}\right)\, ,
\end{equation}

\noindent
and therefore

\begin{equation}
s^{-k}\left(\left(x_{1},\hdots,x_{k}\right)\right) = (-1)^{\sigma} \epsilon\left(\sigma ;  x_{1},\hdots, x_{k}\right) s^{-k}\left(\left(x_{\sigma (1)},\hdots,x_{\sigma (k)}\right)\right)\, ,
\end{equation}

\noindent
which implies, using equation (\ref{eq:sk})

\begin{eqnarray}
(-1)^{\sum_{i=1}^{k}(k-i)|x_{\sigma (i)}|} (-1)^{\sigma}\epsilon\left(\sigma ; x_{1},\hdots, x_{k}\right) \epsilon\left(\sigma ; s^{-1} x_{1},\hdots, s^{-1} x_{k}\right) \left(s^{-1}x_{1},\hdots, s^{-1}x_{k}\right) \nonumber\\ = (-1)^{\sum_{i=1}^{k}(k-i)|x_{\sigma (i)}|} \left(s^{-1}x_{1},\hdots, s^{-1}x_{k}\right)\, .
\end{eqnarray}

\noindent
Hence we conclude

\begin{equation}
\label{eq:epsilondecalage}
\epsilon\left(\sigma ;  s^{-1} x_{1},\hdots, s^{-1} x_{k}\right) = (-1)^{\sum_{i=1}^{k}(k-i)\left(|x_{i}|+|x_{\sigma (i)}|\right)} (-1)^{\sigma}\epsilon\left(\sigma ; x_{1},\hdots, x_{k}\right) \, .
\end{equation}

\noindent
We finish this section by defining two specific kind of algebras that we will find later. When we introduce the \emph{Cartan calculus} in section \ref{sec:cartancalculus}, we will encounter a particular instance of \emph{Gerstenhaber algebra}, which is defined as follows

\definition{\label{def:gealgebra} A Gerstenhaber algebra is an associative and commutative graded algebra $(V,\cdot)$ equipped with a degree -1 bilinear map $[\cdot,\cdot ]: V\to V$ such that the following conditions hold

\begin{enumerate}

\item $[x_{1},x_{2}] = -(−1)^{(|x_{1}|-1)(|x_{2}|-1)} [x_{2},x_{1}]\, ,$ for any two homogeneous elements $x_{1}, x_{2} \in V$. That is, the bilinear map is antisymmetric in the graded sense in $s^{-1}V$. 

\item $[x,\cdot]$ is a derivation on $V$ of degree $|x|-1$ for every $x\in V$.

\item The bilinear map obeys 

\begin{equation}
\left[x_{1},\left[x_{2},x_{3}\right]\right] = \left[\left[x_{1},x_{2}\right],x_{3}\right] + (−1)^{(|x_{1}|-1)(|x_{2}|-1)}\left[x_{2},\left[x_{3},x_{1}\right]\right]\, ,
\end{equation}

\noindent
which becomes the graded Jacobi identity on $s^{-1}V$.

\end{enumerate}

}

\noindent
From the definition In the context of symplectic geometry in section \ref{sec:symplectic} we will find a particular example of a Poisson algebra, whose definition is given by

\definition{\label{def:poissonalgebra} A Lie algebra $\left( X,\cdot, [\cdot , \cdot]\right)$ is called a Poisson algebra if $X$ has a commutative, associative algebra structure such that

\begin{equation}
\label{eq:poissonalgebra}
\left[ x_{1} x_{2}, x_{3}\right] = x_{1}\left[x_{2},x_{3}\right] + \left[ x_{1},x_{3}\right] x_{2}\, ,
\end{equation}

\noindent
for all $x_{1}, x_{2}, x_{3} \in X$.
}

%%%%%%%%%%%%%%%%%%%%%%%%%%%%%%%%%%%%%%%%%%%%%%%%%%%%%%%%%%%%%%%%%%%%%%%%%%%%%%%%%%%%%%%%%%%%%%%%%%%
%%%%%%%%%%%%%%%%%%%%%%%%%%%%%%%%%%%%%%%%%%%%%%%%%%%%%%%%%%%%%%%%%%%%%%%%%%%%%%%%%%%%%%%%%%%%%%%%%%%

\subsection{Homological algebra}
\label{sec:homologicalalgebra}

%%%%%%%%%%%%%%%%%%%%%%%%%%%%%%%%%%%%%%%%%%%%%%%%%%%%%%%%%%%%%%%%%%%%%%%%%%%%%%%%%%%%%%%%%%%%%%%%%%%
%%%%%%%%%%%%%%%%%%%%%%%%%%%%%%%%%%%%%%%%%%%%%%%%%%%%%%%%%%%%%%%%%%%%%%%%%%%%%%%%%%%%%%%%%%%%%%%%%%%

The chain complex is the basic structure of homological algebra. It is defined as follows

\definition{\label{def:chaincomplex} A chain complex $C$  is a sequence $\left(C_{\bullet}, d_{\bullet}\right)$ of abelian groups and group homomorphisms

\begin{equation}
C_{\bullet}: \cdots \longrightarrow C_{n+1} \stackrel{d_{n+1}}{\longrightarrow} C_n \stackrel{d_n}{\longrightarrow} C_{n-1} \stackrel{d_{n-1}}{\longrightarrow} \cdots\, ,
\end{equation}

\noindent
such that 

\begin{equation}
d_n \circ d_{n+1}=0\, .
\end{equation}

\noindent
The abelian groups $C_{i}\, ,\,\, i\in\mathbb{Z}$ are the so-called $i$-chains and the homomorphisms $d_{i}$ are called the so-called boundary maps.

}

Analogously, one can define a co-complex by a simple relabelling of the $i$-chains and the homomorphisms $d_{i}$ as follows

\definition{\label{def:chaincocomplex} A cochain $C$  is a sequence $\left(C^{\bullet}, d^{\bullet}\right)$ of abelian groups and group homomorphisms

\begin{equation}
C^{\bullet}: \cdots \longrightarrow C^{n-1} \stackrel{d^{n-1}}{\longrightarrow} C^n \stackrel{d^n}{\longrightarrow} C^{n+1} \stackrel{d^{n+1}}{\longrightarrow} \cdots\, ,
\end{equation}

\noindent
such that 

\begin{equation}
\label{eq:ddcochain}
d^n \circ d^{n-1}=0\, .
\end{equation}

\noindent
The abelian groups $C^{i}\, ,\,\, i\in\mathbb{Z}$ are the so-called $i$-co-chains and the homomorphisms $d^{i}$ are called the so-called coboundary maps.

}

In this letter we will use the cochain notation. From equation (\ref{eq:ddcochain}) we immediately deduce that

\begin{equation}
B^{i} = \mathrm{Im}\, d^{i-1} \subseteq Z^{i} = \mathrm{Ker}\, d^{i}\, , \qquad i\in\mathbb{Z}\, ,
\end{equation}

\noindent
where $\mathrm{Ker}$ and $\mathrm{Im}$ respectively denote the kernel and image of the given map. Since subgroups of abelian groups are normal, we can define a new group by taking the corresponding quotient. We define this way the $i$th-cohomology group as follows

\begin{equation}
\label{eq:cohomologydef}
H^{i}\left(C\right) = \frac{Z^{i}}{B^{i}}\, , \qquad i\in\mathbb{Z}\, .
\end{equation}

\noindent
A cochain is called an exact sequence if all its cohomology groups are zero. The cochain groups $C^{i}$ may be endowed with extra structure; for example, they may be vector spaces or modules over a fixed ring $\mathbb{K}$. In that case, the coboundary operators must preserve the extra structure if it exists; for instance, they must be linear maps or homomorphisms of $\mathbb{K}$-modules.

Let $C = \left(C^{\bullet}, d_C^{\bullet}\right)$ and $D = \left(D^\bullet, d_D^{\bullet}\right)$ be cochain complexes. A morphism $F: C^\bullet\to D^\bullet$ between $C$ and $D$ is a family of homomorphisms of abelian groups $F^{i}:C^{i} \to D^{i}\, ,\, i\in\mathbb{Z}$ that commute with the co-boundary operators, that is

\begin{equation}
F^{i+1} \circ  d_{C}^{i} = d_{D}^{i} \circ F^{i}\, , \qquad i\in\mathbb{Z}\, .
\end{equation}

\noindent
A morphism of chain complexes induces a morphism $F_{H}$ of their homology groups, consisting of the homomorphisms $F_{H}^{i}\, ,\,\, i\in \mathbb{Z}$ defined by

\begin{equation}
F_{H}^{i}\left([x]\right) = \left[ F^{i}\left(x\right)\right]\, , \qquad \forall [x]\in H^{i}\left( C\right)\, ,\,\, i\in \mathbb{Z} \, .
\end{equation}

\noindent
A morphism $F$ such that $F_{H}$ is an isomorphism, that is, that induces an isomorphism in cohomology, is called a \emph{quasi-isomorphism}. The prominent example of cochain that will appear in this letter is the $L_{\infty}$-algebra, which we will introduce in chapter \ref{chapter:linfty}. For $L_{\infty}$-algebras we will need to introduce a less restrictive concept of morphism, which is adapted to the structure present in $L_{\infty}$-algebras.

Suppose that there exists a map $f: X\to Y$ between two objects $X$ and $Y$. Then, Homological algebra studies the relation, induced by the map $f$, between chain complexes (or co-complexes) associated with $X$ and $Y$ and their homology (or cohomology).

\begin{ep}
\label{ep:deRhamcohomology} Given a manifold $\mathcal{M}$ we can construct the complex 

\begin{equation}
C^{i} \equiv \Omega^{i} \left(\mathcal{M}\right)\, ,
\end{equation}

\noindent
where $\Omega^{i} \left(\mathcal{M}\right)$ denotes the set of $i$-forms on $\mathcal{M}$\footnote{See section \ref{sec:differentialgeometry} for more details.}. The co-boundary operator is the exterior derivative $d^{i}: \Omega^{i} \left(\mathcal{M}\right)\to \Omega^{i+1} \left(\mathcal{M}\right)$. The corresponding cohomology is the \emph{de-Rham cohomology}

\begin{equation}
H^{i}\left(\mathcal{M}\right)=\frac{\mathrm{Ker}\, d^{i}}{\mathrm{Im}\, d^{i-1}}\, .
\end{equation}

\noindent
The de-Rahm cohomology groups give important information about the manifold where it is defined. For instance, the zero cohomology group $H^{0}\left(\mathcal{M}\right)$ of any differentiable manifold $\mathcal{M}$ is given by

\begin{equation}
H^{0}\left(\mathcal{M}\right)\simeq \mathbb{R}^{n}\, ,
\end{equation}

\noindent
where $n$ is the number of connected components of $\mathcal{M}$. This  can be easily seen from the fact that any function $f\in C^{\infty}\left( \mathcal{M}\right)$ such that $df=0$ is constant on each of the connected component of $\mathcal{M}$. Therefore, the dimension of the $H^{0}\left(\mathcal{M}\right)$ gives the number of connected componentes of $\mathcal{M}$.
\end{ep}

\begin{ep}
\label{ep:Liealgebracohomology} {\bf Lie Algebra Cohomology}. Let $\mathfrak{g}$ be a Lie algebra. We define the following cochain complex $C = \left( C^{\bullet}, \delta^{\bullet}\right)$ 

\begin{equation}
C^{i} \equiv \Lambda^{i} \mathfrak{g}^{\ast}\, ,
\end{equation}

\noindent
with coboundary operator $\delta^{k}:\, C^{k}\to  C^{k+1}$ given by

\begin{eqnarray}
\delta^{k} c\left(x_{1},\cdots ,x_{k}\right) = \sum_{1\leq i<j\leq k} (-1)^{i+j} c\left([x_{i}, x_{j}],x_{1},\dots ,\hat{x}_{i},\dots ,\hat{x}_{j},\dots ,x_{k}\right)\, ,
\end{eqnarray}

\noindent
for all $x_{1},\cdots, x_{k}\in\,\mathfrak{g}$. Notice that we interpret an element $c\in C^{i}$ as an alternating $i$-linear operator $c:\, \mathfrak{g}^{\times k}\to \mathbb{R}$. $\delta$ is in fact a coboundary operator; it can be checked that $\delta^{2} = 0$. We can introduce now the \emph{Lie algebra cohomology groups}, or \emph{Chevalley cohomology groups}, of $\mathfrak{g}$ using (\ref{eq:cohomologydef})

\begin{equation}
H^{i}\left(\mathfrak{g}, \mathbb{R}\right)\equiv \frac{\mathrm{Ker}\,\delta^{i}}{\mathrm{Im}\,\delta^{i-1}}\, .
\end{equation}

\noindent
which in fact contain the same information as the de-Rham cohomology groups of $G$, in the following sense

\thm{\label{thm:Liealgebraderham} If $\mathfrak{g}$ is the Lie algebra of compact connected Lie group $G$, then

\begin{equation}
H^{i}\left(\mathfrak{g}, \mathbb{R}\right) = H^{i}_{\mathrm{de\, Rham}}\left(G\right)\, .
\end{equation}
}

\end{ep}

%%%%%%%%%%%%%%%%%%%%%%%%%%%%%%%%%%%%%%%%%%%%%%%%%%%%%%%%%%%%%%%%%%%%%%%%%%%%%%%%%%%%%%%%%%%%%%%%%%%
%%%%%%%%%%%%%%%%%%%%%%%%%%%%%%%%%%%%%%%%%%%%%%%%%%%%%%%%%%%%%%%%%%%%%%%%%%%%%%%%%%%%%%%%%%%%%%%%%%%

\subsection{Coalgebras} 
\label{sec:coalgebras}

%%%%%%%%%%%%%%%%%%%%%%%%%%%%%%%%%%%%%%%%%%%%%%%%%%%%%%%%%%%%%%%%%%%%%%%%%%%%%%%%%%%%%%%%%%%%%%%%%%%
%%%%%%%%%%%%%%%%%%%%%%%%%%%%%%%%%%%%%%%%%%%%%%%%%%%%%%%%%%%%%%%%%%%%%%%%%%%%%%%%%%%%%%%%%%%%%%%%%%%

Coalgebras are structures that are dual, in the category-theoretic sense of reversing arrows, to  algebras. An algebra $\left( V,\cdot \right)$ is, as explained in section \ref{sec:higheralgebra}, a vector space equipped with a product $\cdot$, which defines the following application

\begin{eqnarray}
\cdot \maps V\otimes V &\to & V\nonumber\\
(x_{1}, x_{2}) &\mapsto & x_{1}\cdot x_{2}\, .
\end{eqnarray}

\noindent
Therefore, we should expect a coalgebra $C$ to be equipped with some map in the opposite direction

\begin{equation}
C\to C\otimes C\, .
\end{equation}

\noindent
The precise definition goes as follows. 

\definition{\label{def:coalgebra}
A graded coalgebra $\left( C, \Delta\right)$ is a graded vector space $C$ equipped with a
linear map $\Delta \maps C \to C \otimes C$, the so-called \emph{comultiplication}, such that
 
\begin{equation}
\Delta\left(C_{i}\right) \subset \bigoplus_{j+k=i} C_{j} \otimes C_{k}\, .
\end{equation} 

}

\noindent
A coalgebra $\left( C,\Delta\right)$ is said to be \emph{coassociative} if the following condition holds

\begin{equation}
\label{eq:coassociativity}
\left( \Delta \otimes \mathrm{id} \right) \Delta = \left(\mathrm{id} \otimes \Delta \right) \Delta\, .
\end{equation}

\noindent
Equation (\ref{eq:coassociativity}) is just the coalgebra version of the associativity condition for an algebra. Similarly, we can define a \emph{cocommutative} coalgebra by imposing the dual condition of commutativity in an algebra. Indeed, if $\left( C ,\Delta\right)$ is a coalgebra, let us denote by $T : C \otimes C \to C$ the twist map $T\left( x \otimes y\right) = (-1)^{|x| |y|} y \otimes x$. It is said that $\left( C,\Delta\right)$ is cocommutative if and only if $T \circ \Delta = \Delta$. As well as a the concept of indentity can be defined for algebras, we can introduce a similar structure for coalgebras, called the \emph{identity}. 

\definition{\label{def:counit}
A counit for $\left(C,\Delta\right)$  is a linear map $\epsilon : C \to \mathbb{R}$ such that $(\epsilon\otimes \mathrm{id}) \Delta = (\mathrm{id} \otimes \epsilon) \Delta=\mathrm{id}$.}

\noindent
In order to elucidate the structure present in a coalgebra, the following commutative diagrams may be illustrative 

\begin{center}
\begin{tikzpicture}
\label{diag:coalgebra1}
  \matrix (m) [matrix of math nodes,row sep=8em,column sep=9em,minimum width=2em]
  {
C & C\otimes C \\
C\otimes C & C\otimes C\otimes C \\};
  \path[-stealth]
    (m-1-1) edge node [left] {$\Delta$} (m-2-1)
    (m-1-1) edge node [above] {$\Delta$} (m-1-2)
    (m-2-1) edge node [above] {$\Delta\otimes \mathrm{id}$} (m-2-2)
    (m-1-2) edge node [right] {$\mathrm{id}\otimes \Delta$} (m-2-2);
\end{tikzpicture}

\begin{tikzpicture}
\label{diag:coalgebra2}
  \matrix (m) [matrix of math nodes,row sep=8em,column sep=9em,minimum width=2em]
  {
C & C\otimes C \\
C\otimes C & \mathbb{R}\otimes C\simeq C\simeq C\otimes \mathbb{R} \\};
  \path[-stealth]
    (m-1-1) edge node [left] {$\Delta$} (m-2-1)
    (m-1-1) edge node [above] {$\Delta$} (m-1-2)
    (m-2-1) edge node [above] {$\epsilon\otimes \mathrm{id}$} (m-2-2)
    (m-1-1) edge node [above] {$\mathrm{id}$} (m-2-2)
    (m-1-2) edge node [right] {$\mathrm{id}\otimes \epsilon$} (m-2-2);
\end{tikzpicture}
\end{center}

\noindent
Diagram (\ref{diag:coalgebra1}) is the dual, in the category-theoretic sense of reversing arrows, of the analogous diagram that express associativity of algebra multiplication. Diagram (\ref{diag:coalgebra2}) is the dual, in the same sense as before, of the analogous diagram expressing the existence of an identity in a unital algebra. 

\definition{\label{def:coaugmented} A coalgebra $\left( C, \Delta\right)$ with counit $\epsilon$ is  \emph{coaugmented} if and only if it can be equipped with an injective linear map

\begin{eqnarray}
\mathbb{R} &\hookrightarrow & C\\
1 &\mapsto & 1_{C}
\end{eqnarray}

\noindent
such that $\epsilon\left( 1_{C}\right) = \mathrm{1}$ and $\Delta\left( \mathrm{1}_{C}\right)=\mathrm{1}_{C}\otimes \mathrm{1}_{C}$. We can write in that case $\bar{C} = \ker \epsilon$ so that we have $C \simeq \mathbb{R} \oplus \bar{C}$ as vector spaces.}

\noindent
Intuitively, if a coalgebra $C$ admits a coaugmentation, then it contains a copy of $\mathbb{R}$.  For a given coaugmented coalgebra, we define the \emph{reduced comultiplication} $\bar{\Delta} \maps \bar{C} \to\bar{C} \otimes \bar{C}$ as follows

\begin{equation}
\bar{\Delta} c = \Delta c - c \otimes \mathrm{1}_{C} - \mathrm{1}_{C}\otimes c.
\end{equation}

\noindent
The equation above makes $\bar{C}$ into a coalgebra with no counit. We denote now the \textbf{reduced diagonal} by $\Delta^{(n)}$. It can be recursively defined as follows

\begin{eqnarray}
\label{eq:red_diag}
\bar{\Delta}^{(0)}&=&\mathrm{id}\nonumber\\
\bar{\Delta}^{(1)}&=&\bar{\Delta}\\
\bar{\Delta}^{(n)} &=& \left( \bar{\Delta} \otimes \mathrm{id}^{\otimes(n-1)} \right) \circ
\bar{\Delta}^{(n-1)} \maps \bar{C} \to \bar{C}^{\otimes (n+1)}\nonumber\, .
\end{eqnarray} 

\noindent
It can be shown by induction that we can rewrite $\rdDelta{n}$ as

\begin{equation} 
\label{eq:red_diag2}
\bar{\Delta}^{n} = \left( \bar{\Delta}^{(n-1)} \otimes \mathrm{id} \right) \circ \bar{\Delta}\, .
\end{equation}

\noindent
A coaugmented coalgebra $\left(C,\Delta,\epsilon,\mathrm{1}_{C}\right)$ has always a canonical filtration\footnote{A filtration is an indexed set $S_{i}$ of sub-objects of a given algebraic structure $S$, with the index $i$ taking values in a totally ordered set $I$, subject to the condition that if $i \leq j$ in $I$ then $S_{i} \subseteq S_{j}$.} which can be defined recursively as follows

\begin{eqnarray}
\label{eq:filter_def}
F_{0} C &=& \mathbb{R} \cdot \mathrm{1}_{C} \\
F_{k} C &=& \left\{ x \in \bar{C} ~ \vert ~ \rDelta x \in F_{k-1}C \otimes F_{k-1} C \right\}\, ,
\end{eqnarray}

\noindent
and it is \emph{connected} if and only if 

\begin{equation}
C = \bigcup F_{k} C\, .
\end{equation}

\noindent
If $\left(C,\Delta,\epsilon,\mathrm{1}_{C}\right)$ is connected, it can be proven that the coaugmentation $\mathrm{1}_{C}$ is unique \footnote{ See proposition 3.1 in section B3 of reference \cite{Quillen:1969}}.

\begin{ep}
\label{ep:gradedS} Graded symmetric algebra. Let $V$ be a graded vector space. Then the the graded symmetric algebra is given by

\begin{equation}
S(V) = \mathbb{R} \oplus \sum_{k=1}^{\infty} S^{k} V   =\mathbb{R} \oplus \S(V)
\end{equation}

\noindent
is a coaugmented cocommutative coalgebra in a natural way. The comultiplication $\Delta$ is defined as the unique morphism of algebras such that $\Delta(v) = v \otimes \mathrm{1} + \mathrm{1} \otimes v$ holds for all $v\in V$, assuming that $\Delta (1) = 1\otimes 1$. The counit is defined as the projection $S\left(V\right) \to \mathbb{R}$, and the coaugmentation is given by the inclusion $\mathbb{R}
\hookrightarrow S\left(V\right)$. The reduced comultiplication $\bar{\Delta}$ on $\S(V)$ can be given explicitly by
 
\begin{eqnarray}
 \label{eq:redcomult}
 \bar{\Delta}(v_{\mathrm{1}} \odot v_{2} \odot \cdots \odot v_{n}) &=&
  \sum_{1 \leq p \leq n-1} \sum_{\sigma \in \Sh\left(p,n-p\right)} \epsilon\left(\sigma\right)
  \left(v_{\sigma\left(1\right)} \odot v_{\sigma\left(2\right)} \odot \cdots \odot v_{\sigma\left(p\right)} \right)\nonumber\\
  & \otimes &\left ( v_{\sigma\left(p+1\right)} \odot v_{\sigma\left(p+2\right)} \odot \cdots \odot v_{\sigma\left(n\right)}
  \right)\, ,
\end{eqnarray}

\noindent
where $\odot$ denotes the symmetrized tensor product. That is, if $v_{1}, \cdots, v_{p}\in V$ are $p\in\mathbb{N}$ homogeneous elements then we have

\begin{equation}
v_{1}\odot \cdots \odot v_{p} = \frac{1}{p!} \sum_{\sigma\in \Sigma_{p}} v_{\sigma (1)}\otimes\cdots\otimes v_{\sigma (p)} \, .
\end{equation}
\end{ep}

\noindent
The following lemma, which we extract from \cite{2013arXiv1301.4864F},  will be needed in order to properly rewrite $L_{\infty}$-morphisms.

\begin{lemma} 
\label{lemma:red_diag_lemma}
If $v_{\mathrm{1}} \odot v_{\mathrm{2}} \odot \cdots \odot v_{n} \in \bar{S}\left( V\right)$, and $ \mathrm{1} \leq p \leq n-\mathrm{1}$ then

\begin{eqnarray}
\bar{\Delta}^{p}\left( v_{\mathrm{1}} \odot \cdots \odot v_{n}\right) &=& \sum^{k_{1} + k_{2} +
  \cdots + k_{p+1}=n}_{k_{1},k_{\mathrm{2}},\hdots,k_{p+1} \geq 1} \sum_{\sigma
  \in \Sh(k_{\mathrm{1}},k_{2},\hdots,k_{p+1})} \epsilon\left(\sigma\right) v_{\sigma\left( 1\right)}
\odot \cdots \odot v_{\sigma\left(k_{\mathrm{1}}\right)}\nonumber \\
 &\otimes & v_{\sigma(k_{1} + 1)} \odot \cdots \odot
v_{\sigma\left(k_{1}+k_{2}\right)} \otimes v_{\sigma\left(k_{1} +k_{2} + 1\right)} \odot \cdots \odot
v_{\sigma(k_{1}+k_{2} +k_{3})} \otimes \cdots \nonumber\\
&\otimes & v_{\sigma(m-k_{p+1} + 1)} \odot \cdots \odot v_{\sigma(n)}\, ,
\end{eqnarray}

\noindent
and in particular we have

\begin{equation}
\bar{\Delta}^{(n-1)}\left(v_{\mathrm{1}} \odot \cdots \odot v_{n}\right) = \sum_{\sigma \in
  \mathcal{S}_{n}} \epsilon(\sigma) v_{\sigma\left(\mathrm{1}\right)} \otimes v_{\sigma\left(\mathrm{2}\right)}
  \otimes \cdots \otimes v_{\sigma\left( n\right)}\, . 
\end{equation}

\end{lemma}

\proof{See lemma A.1 in \cite{2013arXiv1301.4864F}.}

\noindent
Lemma \ref{lemma:red_diag_lemma} implies that $\ker \bar{\Delta}^{(k)} =\bar{S}^{\bullet \leq  k}\left( V\right)$ for  $k \geq 0$ and also that

\begin{equation} 
\label{eq:prim_cogen}
\bar{S}\left( V\right) = \bigcup_{n} \ker \bar{\Delta}^{(n)}\, . 
\end{equation}

\noindent
The filtration $F_{n}S\left( V\right)$ corresponds, for $n \geq 1$, to the filtration on $\bar{S}\left( V \right)$ defined by $\ker \bar{\Delta}^{(n)}$. This proves that $S(V)$ is a connected coalgebra.

In order to define $L_{\infty}$-algebras and $L_{\infty}$-algebra morphisms, the concepts of coalgebra differential and coalgebra morphism are going to prove essential. 

\definition{\label{def:coalgebraderivation}
A \emph{codifferential} of degree one on a coalgebra $\left(C,\Delta\right)$ is a linear map $Q \maps C^{i} \to C^{i+1}$ satisfying

\begin{equation}
\quad Q \circ Q =0\, ,
\end{equation}

\noindent
and the coLeibniz identity

\begin{equation}
\Delta Q = \left (Q \otimes \mathrm{id} \right) \Delta +  (\mathrm{id} \otimes Q ) \Delta\, .
\end{equation}
}

\noindent
In the case of a codifferential $Q$ defined on connected coalgebra $(C,\Delta,\epsilon,\mathrm{1}_{C})$, we require an additional condition, namely

\begin{equation}
Q(1_{C})=0\, .
\end{equation}

\noindent
A codifferential defined on a connected coalgebra  $(C,\Delta,\epsilon,\mathrm{1}_{C})$ is uniquely given by its restriction to $\bar{C}$, which satisfies the coLeibniz identity with respect to $\bar{\Delta}$. 

Let us consider the particular case $C=\bar{\mathcal{T}}\left(V\right)$. Given a codifferential $Q$ on $\bar{\mathcal{T}}(V)$, consider the restrictions

\begin{equation}
\label{eq:restrictQ}
Q_{m}=Q \vert_ {\bar{\mathcal{T}}^{m}\left( V\right)} : \bar{\mathcal{T}}^{m}\left( V\right) \to \bar{\mathcal{T}}(V), \quad 1 \leq m < \infty\, ,
\end{equation}

\noindent
so that $Q = \sum_{k}^{\infty} Q_{k}$, and the projections

\begin{equation} 
\label{eq:QprojT}
Q^{k}_{m} = \mathrm{pr}_{\bar{\mathcal{T}}^{k}\left( V\right)} \circ Q_{m} \maps \bar{\mathcal{T}}^{m}(V) \to
\bar{\mathcal{T}}^{k}(V)\, .
\end{equation}

Then, the following proposition holds.

\prop{\label{prop:codifferentialrestriction} A coderivation $Q$ of $\bar{\mathcal{T}}\left(V\right)$, respectively $\bar{S}\left(V\right)$, is uniquely determined by the collection $Q^{1}_{i}\, ,\,\, i\in\mathbb{N}^{+}$ by the formula

\begin{eqnarray}
\label{eq:coder_eqT}
Q_{m}\left( x_{\mathrm{1}} \otimes \cdots \otimes x_{m}\right) = 
Q^{1}_{m}( x_{\mathrm{1}} \otimes \cdots \otimes  x_{m}) +\nonumber \\
\sum^{m-1}_{i=1} \sum_{\sigma \in \Sh(i,m-i)}
\epsilon(\sigma) Q^{1}_{i}\left( x_{\sigma(1)} \otimes \cdots \otimes
 x_{\sigma(i)} \right)\otimes  x_{\sigma\left(i+\mathrm{1}\right)} \otimes \cdots \otimes  x_{\sigma(m)}\, ,
\end{eqnarray}

}

\proof{See lemma 2.4 in \cite{Lada:1994mn} or appendix A in \cite{2013arXiv1301.4864F}.}

\definition{\label{def:coalgebradmorphism2}A morphism between connected coalgebras $(C_{1},\Delta_{1},\epsilon_{1},1_{C_{1}})$ and $(C_{2},\Delta_{2},\epsilon_{2},1_{C_{2}})$ is a degree zero linear map $F \maps C_{1} \to C_{2}$ satisfying

\begin{equation} 
\Delta_{\mathrm{2}} \circ F = \left(F \otimes F\right) \circ \Delta\, ,
\end{equation}

\begin{equation}
\label{eq:presvcounit}
\epsilon_{\mathrm{1}} = \epsilon_{\mathrm{2}} \circ F\, .
\end{equation}

\noindent
We have that, since $\mathbb{R}$ is a field, $F$ automatically preserves the coaugmentations. Therefore, it can be uniquely determined by its restriction to $\bar{C}$. Hence,  morphisms between such coalgebras correspond to degree zero linear maps
$F : \bar{C}_{1} \to \bar{C}_{2}$ that satisfy $\bar{\Delta}_{2} \circ F = \left( F \otimes F\right) \circ \bar{\Delta}_{\mathrm{1}}$. The condition (\ref{eq:presvcounit}) is just the dual of the condition 

\begin{equation}
\phi\left(1_{X_{1}}\right) = 1_{X_{2}}\, ,
\end{equation}

\noindent
where $X_{i}\, ,\, i =1,2$
}

\prop{\label{prop:morphismrestriction} Let $\left(C,\Delta,\epsilon,1_{C}\right)$ be a connected coalgebra and let $f:\, \bar{C}\to V$ be a degree zero linear map from $\bar{C}=\mathrm{ker}\, \epsilon$ to a graded vector space $V$. Then, there exists a unique morphism of connected coalgebras $F:\, C\to S\left(V\right)$ such that $\mathrm{pr}_{V}\circ F|_{\bar{C}} = f$, where $\mathrm{pr}_{V}:\, S\left(V\right)\to V$ is the corresponding proyection.}

\proof{See references \cite{Quillen:1969} and \cite{rational}.}

\noindent
Since it will be important in chapter \ref{chapter:linfty} for the study of $L_{\infty}$ algebras, let us consider the particular case when $\bar{C} = \bar{S}\left( V\right)$. We define then the degree zero linear map

\begin{equation}
F^{1} : \bar{S}\left( V_{1}\right)\to V_{2}\, ,
\end{equation}

\noindent
where $V_{1}$ and $V_{2}$ are graded vector spaces. Additionally, we define the restrictions $F^{1}_{k}$ as follows

\begin{equation}
F^{1}_{k} = F^{1}|_{\bar{S}^{k}\left( V_{1}\right)}\, ,
\end{equation}

\noindent
and hence

\begin{equation}
F^{1} = \sum^{\infty}_{k} F^{1}_{k}\, .
\end{equation}

\noindent
Then, the following corollary of proposition \ref{prop:morphismrestriction} holds

\begin{cor}
\label{cor:morphismrestriction}
Let $V_{1}$ and $V_{2}$ be graded vector spaces and let $F^{1} \maps \bar{S}\left( V_{1}\right) \to  V_{2}$  be a degree zero linear map. There exists then a unique morphism of coalgebras

\begin{equation}
F : S\left( V_{\mathrm{1}}\right) \to S\left( V_{\mathrm{2}}\right)\, ,
\end{equation}

\noindent
such that it satisfies $\mathrm{pr}_{V_{2}} \circ F \vert_{\bar{S}\left( V_{2}\right)} = F^{1}$, where $\mathrm{pr}_{V_{2}}$ is the corresponding projection to $V_{2}$.
\end{cor}

\noindent
The construction of the coalgebra morphism $F$ for this particular case can be found in proposition A.2 of \cite{2013arXiv1301.4864F}. 

%%%%%%%%%%%%%%%%%%%%%%%%%%%%%%%%%%%%%%%%%%%%%%%%%%%%%%%%%%%%%%%%%%%%%%%%%%%%%%%%%%%%%%%%%%%%%%%%%%%
%%%%%%%%%%%%%%%%%%%%%%%%%%%%%%%%%%%%%%%%%%%%%%%%%%%%%%%%%%%%%%%%%%%%%%%%%%%%%%%%%%%%%%%%%%%%%%%%%%%
%%%%%%%%%%%%%%%%%%%%%%%%%%%%%%%%%%%%%%%%%%%%%%%%%%%%%%%%%%%%%%%%%%%%%%%%%%%%%%%%%%%%%%%%%%%%%%%%%%%
%%%%%%%%%%%%%%%%%%%%%%%%%%%%%%%%%%%%%%%%%%%%%%%%%%%%%%%%%%%%%%%%%%%%%%%%%%%%%%%%%%%%%%%%%%%%%%%%%%%

\section{Differential geometry basics}
\label{sec:differentialgeometry}

%%%%%%%%%%%%%%%%%%%%%%%%%%%%%%%%%%%%%%%%%%%%%%%%%%%%%%%%%%%%%%%%%%%%%%%%%%%%%%%%%%%%%%%%%%%%%%%%%%%
%%%%%%%%%%%%%%%%%%%%%%%%%%%%%%%%%%%%%%%%%%%%%%%%%%%%%%%%%%%%%%%%%%%%%%%%%%%%%%%%%%%%%%%%%%%%%%%%%%%
%%%%%%%%%%%%%%%%%%%%%%%%%%%%%%%%%%%%%%%%%%%%%%%%%%%%%%%%%%%%%%%%%%%%%%%%%%%%%%%%%%%%%%%%%%%%%%%%%%%
%%%%%%%%%%%%%%%%%%%%%%%%%%%%%%%%%%%%%%%%%%%%%%%%%%%%%%%%%%%%%%%%%%%%%%%%%%%%%%%%%%%%%%%%%%%%%%%%%%%

Let $\mathcal{M}$ be a topological space. $\mathcal{M}$ is said to be Hausdorff or $\mathrm{T}_2$ if for every pair points $p\, , q\,\in \mathcal{M}$ there exist neighbourhoods $\mathcal{U}(p)\, ,\mathcal{U}(q)$ of $p$ and $q$, such that $\mathcal{U}(p)\cap \mathcal{U}(q) = \emptyset$. $\mathcal{M}$ is a second-countable space if its topology has a countable base, that is, if there exists a countable collection $\left\{\mathcal{U}_{i}\right\}^{\infty}_{i=1}$ of open sets such that any open set in $\mathcal{M}$ can be written as a subfamily of the collection. An \emph{open chart} on $\mathcal{M}$ is a pair $(\mathcal{U},\phi)$, where $\mathcal{U}$ is an open subset of $\mathcal{M}$ and $\phi$ is a homeomorphism of $\mathcal{U}$ onto an open subset of $\mathbb{R}^{n}$.

\definition{\label{def:diffstructure} Let $\mathcal{M}$ be a Hausdorff, second-countable, topological space. A \emph{differentiable structure} on $\mathcal{M}$ is a collection of open charts $(\mathcal{U}_{\alpha},\phi_{\alpha})_{\alpha\in I}$ on $\mathcal{M}$ such that the following conditions hold

\begin{enumerate}

\item \label{c1M} $\mathcal{M} = \bigcup_{\alpha\in I} \mathcal{U}_{\alpha}$

\item \label{c2M} $\phi_{\alpha}\left(\mathcal{U}_{\alpha}\right)$ is an open set of $\mathbb{R}^{n}$ for all $\alpha\in I$ and for each pair $\alpha\, , \beta \in I$ $\phi_{\beta}\circ\phi^{-1}_{\alpha}$ is a differentiable\footnote{By differentiable we will always mean, unless otherwise stated, infinitely differentiable or $C^{\infty}$.} mapping of $\phi_{\alpha}\left(\mathcal{U}_{\alpha}\cap \mathcal{U}_{\beta} \right)$ onto $\phi_{\beta}\left(\mathcal{U}_{\alpha}\cap \mathcal{U}_{\beta} \right)$.

\item \label{c3M} The collection $(\mathcal{U}_{\alpha},\phi_{\alpha})_{\alpha\in I}$ is a maximal family of open charts for which conditions \ref{c1M}. and \ref{c2M}. hold. $(\mathcal{U}_{\alpha},\phi_{\alpha})_{\alpha\in I}$ is then called the \emph{atlas} of $\mathcal{M}$.

\end{enumerate}

}

\definition{\label{def:manifold} A \emph{differentiable manifold} of dimension $n$ is a Hausdorff, second-countable, topological space equipped with a differentiable structure of dimension $n$. }

\noindent
\newline If $\mathcal{M}$ is a manifold, a \emph{local chart} or \emph{local coordinate system} on $\mathcal{M}$ is a pair $(\mathcal{U}_{\alpha},\phi_{\alpha})$ where $\alpha\in I$. For every $p\in\mathcal{U}_{\alpha}\, , \alpha\in I$, $\mathcal{U}_{\alpha}$ is called a \emph{coordinate neighbourhood} of $p$ and the numbers $\phi_{\alpha}(p)=({\bf x}^{1}(p),\dots,{\bf x}^{n}(p))$ are the local coordinates of $p$. Condition \ref{c3M} is not essential in the definition of a manifold, since if     only \ref{c1M} and \ref{c2M} are satisfied, the family $(\mathcal{U}_{\alpha},\phi_{\alpha})_{\alpha\in I}$ can be extended in a unique way to a family of charts such that \ref{c1M}, \ref{c2M} and \ref{c3M} are fulfilled.

\noindent
Since a manifold $\mathcal{M}$ is locally homeomorphic to $\mathbb{R}^{n}$, it shares the same local topological properties. In particular, manifolds are locally compact and locally connected. That means, respectively, that every point $p\in\mathcal{M}$ has a compact neighbourhood and a connected neighbourhood. Using that the topology of $\mathcal{M}$ has a countable basis and it is locally compact, it can be shown that $\mathcal{M}$ is paracompact, that is, every open cover of $\mathcal{M}$ admits a locally finite refinement. Paracompactness is a sufficient condition for partitions of unit to exist, and therefore $\mathcal{M}$ is also metrizable\footnote{By metrizable, we mean that $\mathcal{M}$ admits a metric $g$.}. Schematically we can write

\begin{equation}
\mathcal{M}:~ \mathrm{Hausdorff~and ~second-countable ~space}\xrightarrow{\mathrm{locally~ homeomorphic~ to}~\mathbb{R}^{n}} \mathrm{Paracompact~and~metrizable}\, .
\end{equation} 

\noindent
In order to give some intuition or justification to the conditions included in the definition of a manifold, let us take an example from Physics, in particular from General Relativity. In the context of General Relativity, the space-time is usually described as an $n$-dimensional differentiable manifold $\mathcal{M}$. In that context, the Hausdorff condition is natural since it is experimentally observed. On the other hand, the gravitational interaction is described by a Lorentzian metric ${\bf g}$ on $\mathcal{M}$. Since $\mathcal{M}$ is paracompact and therefore metrizable, we know that we can always equip the space-time manifold with such a metric. The second-countable condition is a reasonable assumption for topological spaces locally homeomorphic to $\mathbb{R}^{n}$, since otherwise the space would not be adapted to be locally \emph{like} $\mathbb{R}^{n}$.

A function $f:\mathcal{M}\to\mathbb{R}$ is differentiable at $p\in\mathcal{U}_{\alpha}\subset\mathcal{M}$ if $f\circ\phi^{-1}_{\alpha}:\phi_{\alpha}\left(\mathcal{U}_{\alpha}\right)\subset\mathbb{R}^{n}\to\mathbb{R}$ is differentiable at $\phi_{\alpha}(p)\in\mathbb{R}^{n}$. $f$ is called differentiable if it is differentiable at every point $p\in\mathcal{M}$. We denote by $C^{\infty}(\mathcal{M})$ the set of differentiable functions from $\mathcal{M}$ into $\mathbb{R}$ and by $C^{\infty}(\mathcal{M},p)$ the set of functions from $\mathcal{M}$ into $\mathbb{R}$  differentiable at $p\in \mathcal{M}$.

There are three basic and equivalent ways to define the tangent space $T_{p}\mathcal{M}$ of  a differentiable manifold $\mathcal{M}$ at a point $p\in\mathcal{M}$, namely

\begin{enumerate}

\item Let $\mathfrak{T}_{p}$ be the set of all pairs $(\phi_{\alpha},u)$, where $p\in\mathcal{U}_{\alpha}$ and $u\in\mathbb{R}^{n}$. We define an equivalence relation $(\phi_{\alpha},u)\sim (\phi_{\beta},v)$ by the condition

\begin{equation}
d_{\phi_{\alpha}(p)}\left(\phi_{\beta}\circ\phi^{-1}_{\alpha}\right)(u)=v\, .
\end{equation}

\noindent
The equivalence class of $(\phi_{\alpha},u)$ will be denoted by $\left[\phi_{\alpha},u\right]$. The set $T_{p}\mathcal{M}\equiv\mathfrak{T}_{p}/\sim$ is then the tangent space at the point $p\in\mathcal{M}$. If $e_{i}\, , i=1,\dots,n$ is the canonical basis of $\mathbb{R}^{n}$ we define the partial derivatives respect to ${\bf x}^{i}$ by:

\begin{equation}
\frac{\partial}{\partial {\bf x}^{i}} = \left[\phi_{\alpha},e_{i}\right]\, ,\qquad i=1,\dots,n\, .
\end{equation}

\item A curve in $\mathcal{M}$ is a map $c:\left[0,1\right]\to\mathcal{M}$. A curve $c$ is differentiable at $t_{0}\in (0,1)\, , c(t_{0})\in\mathcal{U}_{\alpha}\subset\mathcal{M}$, if  $\phi_{\alpha}\circ c: [0,1]\to \mathbb{R} $ is differentiable at $t_{0}$. Let be $\mathfrak{T}_{p}$ the set of all the curves in $\mathcal{M}$ passing through $p\in\mathcal{U}_{\alpha}$. We define the following equivalence relation $\sim$: two curves $c_{1}(t)$ and $c_{2}(t)$ on $\mathcal{M}$ passing through $p$ are related by $\sim$ if

\begin{equation}
\partial_{t}\left(\phi_{\alpha}\circ c_{1} \right)(t_{0})=\partial_{t}\left(\phi_{\alpha}\circ c_{2} \right)(t_{0})\, ,
\end{equation}

\noindent
where $t_{0}\in\mathbb{R}$ is a fixed real number. We denote by $[c]$ the class of equivalence of $c$. Then the tangent space is $T_{p}\mathcal{M} \equiv \mathfrak{T}_{p}/\sim$.
 
\item A derivation at $p\in\mathcal{M}$ is a linear application $D:C^{\infty}(\mathcal{M})\to\mathbb{R}$ such that

\begin{equation}
D\left(fg\right)=f\left(p\right)D\left(g\right)+g D\left(f\right)\, .
\end{equation}

\noindent
for every $f,g\in C^{\infty}\left(\mathcal{M},p\right)$. We denote the space of derivations at $p\in\mathcal{M}$ by $\mathrm{Der}\left(\mathcal{M},p\right)$. Then we have $T_{p}\mathcal{M}\equiv\mathrm{Der}\left(\mathcal{M},p\right)$.

\end{enumerate}

\noindent

Let $\mathcal{M}$ be a  differentiable manifold of dimension $n$ with atlas $(\mathcal{U}_{\alpha},\psi_{\alpha})_{\alpha\in I}$ and let $\mathcal{N}$ be a differentiable manifold of dimension $m$ with atlas $(\mathcal{V}_{\beta},\psi_{\beta})_{\beta\in J}$. Let $\Phi$ be a mapping from $\mathcal{M}$ to $\mathcal{N}$. $\Phi$ it is said to be differentiable at a point $p\in\mathcal{U}_{\alpha}$ if the map $\psi_{\beta}\circ\Phi\circ\phi^{-1}_{\alpha}: \mathbb{R}^{n}\to\mathbb{R}^{m}$, where $\Phi (p) = q\in\mathcal{V}_{\beta}$, is differentiable at $\phi^{-1}_{\alpha}(p)$. $\Phi$ is said to be differentiable if it is differentiable at every point $p$ in $\mathcal{M}$. $\Phi$ is said to be a diffeomorphism if it is a differentiable biyective map with differentiable inverse. $\Phi$ is said to be a local diffeomorphism if for every $p\in\mathcal{M}$ there exists an open set $\mathcal{U}\in\mathcal{M}$ containing $p$ such that $\Phi$ restricted to $\mathcal{U}$ is a diffeomorphism.

The differential $d_{p}\Phi$ of $\Phi$ at a point $p\in\mathcal{M}$ is a linear map $d_{p}\Phi : T_{p}\mathcal{M}\to T_{q}\mathcal{N}$ which can be naturally defined, for each of the equivalent definitions of the tangent space, as follows

\begin{enumerate}

\item $d_{p}\Phi: \left[\phi_{\alpha},u\right] \mapsto \left[\psi_{\alpha},d_{\phi_{\alpha}(p)}\left(\psi_{\alpha}\circ\Phi\circ\phi^{-1}_{\alpha}\right)(u)\right]$

\item $d_{p}\Phi: \left[c\right] \mapsto \left[\Phi\circ c\right]$

\item $d_{p}\Phi: D \mapsto \Phi_{\ast} D\, ,$ where $\Phi_{\ast} D\left( h\right) = D \left(h\circ\Phi\right)$ for every $h\in C^{\infty}\left(\mathcal{N},q\right)$.

\end{enumerate}

We are ready to introduce the tangent bundle of $\mathcal{M}$, which is a special instance of vector bundle, which will be defined in section \ref{sec:fibrebundles}. Let us consider the set

\begin{equation}
T\mathcal{M} = \left\{(p,v): p\in\mathcal{M}\, , v\in T_{p}\mathcal{M}\right\} = \bigcup_{p\in\mathcal{M}}\left\{p\right\}\times T_{p}\mathcal{M}\, .
\end{equation}

\noindent
We define the projection $\pi : T\mathcal{M}\to\mathcal{M}$ as $\pi (p,v)=p$ for every $(p,v)\in T\mathcal{M}$. $T\mathcal{M}$ can be equipped with the topology whose open sets are given by $\pi^{-1}\left(\mathcal{U}_{\alpha}\right)$, where $\mathcal{U}_{\alpha}\subset\mathcal{M}$ is an open set of the atlas of $\mathcal{M}$. In every $\pi^{-1}\left(\mathcal{U}_{\alpha}\right)$ we define coordinates $\tilde{\phi}_{\alpha}$ as follows

\begin{equation}
\tilde{\phi}_{\alpha}(p,v) = \left(\phi_{\alpha}(p), d_{p}\phi_{\alpha}(v)\right)\in\mathbb{R}^{2n}\, ,
\end{equation}

\noindent
for every $(p,v)\in \mathcal{U}_{\alpha}\subset T\mathcal{M}$. Therefore, $T\mathcal{M}$ is a $2n$-dimensional manifold which is in particular a vector bundle of rank $n$ and fibre at a point $p\in\mathcal{M}$ given by the vector space $T_{p}\mathcal{M}$. 

A complete vector field $v\in\mathfrak{X}(\mathcal{M})$\footnote{A vector field is said to be complete if the parameter of each integral curve extends to $\left(-\infty,\infty\right)$.} generates a family of diffeomorphisms $\rho_{t}:\mathcal{M}\to\mathcal{M}\, , \, t\in\mathbb{R}$ as follows.

\definition{\label{def:flowvectorfield} Every complete vector field $v\in\mathfrak{X}(\mathcal{M})$ generates a family of diffeomorphisms  $\rho_{t}:\mathcal{M}\to\mathcal{M}\, , \, t\in\mathbb{R}$. For each $p\in\mathcal{M}$, $\rho_{t}(p)\, , \, t\in\mathbb{R}$ is, by definition, the unique integral curve of $v$ passing through $p$ at $t=0$.}

It can be can be checked that the map $\mathbb{R}\to\mathrm{Diff}\left(\mathcal{M}\right)$ given by $t\to\rho_{t}$ is a group homomorphism. The family $\left\{\rho_{t}\, ,\, t\in\mathbb{R}\right\}$ is called the one-parameter group of diffeomorphisms of $\mathcal{M}$. 

At every point $p\in\mathcal{M}$ we denote the dual space of $T_{p}\mathcal{M}$ as $T^{\ast}_{p}\mathcal{M}$. Elements of $T^{\ast}_{p}\mathcal{M}$ are called one-forms at the point $p$. Similarly, the dual bundle of $T\mathcal{M}$ is denoted by $T^{*}\mathcal{M}$. Sections\footnote{See definition (\ref{def:bundlesection}).} of $T\mathcal{M}$ ($T^{\ast}\mathcal{M}$) are vector (one-form) fields in $\mathcal{M}$. They correspond simply to a smooth choice of vector (one-form) in $T_{p}\mathcal{M}$ ($T^{\ast}_{p}\mathcal{M}$) at every point $p\in\mathcal{M}$. The set of all the vector fields (one-form fields) in $\mathcal{M}$ is denoted by $\mathfrak{X}(\mathcal{M})$ ($\Omega^{1}(\mathcal{M})$) or equivalently by $\Gamma (T\mathcal{M})$ ($\Gamma (T^{\ast}\mathcal{M})$). Analogously, an element $\mathfrak{T}_{p}$ of $\left(T_{p}\mathcal{M}\right)^{\otimes s}\otimes \left(T^{\ast}_{p}\mathcal{M}\right)^{\otimes r}$ is a $(r,s)$ tensor\footnote{For the definition of tensor product see \ref{def:tensorproduct}.} and a section $\mathfrak{T}$ of $\Gamma\left(\left(T\mathcal{M}\right)^{\otimes s} \otimes\left(T^{\ast}\mathcal{M}\right)^{\otimes r}\right)$ is a $(r,s)$ tensor field on $\mathcal{M}$.

Of outermost importance in differential geometry are the tensor algebra\footnote{For the abstract definition of algebra see \ref{def:algebra}.} $\left(\mathrm{T}\left(\mathcal{M}\right),\otimes\right)$ and the algebra of differential forms $\left(\Lambda\left(\mathcal{M}\right),\wedge\right)$. Let $\mathrm{T}_{(r,s)}\left(\mathcal{M}\right)$ denote the set of all tensor fields on $\mathcal{M}$ of type $(r,s)$, and let $\Omega^{k}\left(\mathcal{M}\right)$ denote the set of all $k$-form fields on $\mathcal{M}$. Then we have\footnote{For more details and general definitions, see section \ref{sec:hlahomalgebra}.}

\begin{equation}
\mathrm{T}\left(\mathcal{M}\right) = \sum_{r,s = 1}^{\infty} \mathrm{T}_{(r,s)}\left(\mathcal{M}\right)\, ,\qquad \Lambda\left(\mathcal{M}\right) = \sum_{k=1}^{\infty}\Omega^{k}\left(\mathcal{M}\right)\, .
\end{equation}

\noindent
Please notice that the infinite sum in the definition of $\Lambda\left(\mathcal{M}\right)$ is only formal; for finite-dimensional manifolds it will contain only a finite number of terms. 

There are several important operators that can be defined on $\mathrm{T}\left(\mathcal{M}\right)$ and $\Lambda\left(\mathcal{M}\right)$. Here we will consider the interior product $\iota_{v}$, the exterior derivative or de Rham differential $d$ and the Lie derivative $\mathcal{L}_{v}\, , \, v\in\Gamma\left(T\mathcal{M}\right)$. 

\paragraph{The interior product $\iota_{v}$}

The interior product $\iota_{v}\colon \Omega^{i}\left(\mathcal{M}\right) \to \Omega^{(i-1)}\left(\mathcal{M}\right)$ is a $-1$ degree derivation on the exterior algebra of differential forms $\Lambda\left(\mathcal{M}\right)$. It is defined to be the contraction of a differential form with a vector field $v\in\mathfrak{\mathcal{M}}$ as follows

\begin{eqnarray}
\left( \iota_{v}\omega \right)\left(v_1,\ldots,v_{(i-1)}\right)=\omega\left(v,v_1,\ldots,v_{(p-1)}\right)\, , \qquad \forall\,\, v_{1}, \ldots, v_{(p-1)} \in \mathfrak{X}\left(\mathcal{M}\right)\, .
\end{eqnarray}

\noindent
The interior product is the unique derivation of degree $−1$ on the exterior algebra such that on one-forms corresponds to the natural pairing of one-forms and vectors.

\paragraph{The exterior derivative $d$}

The exterior derivative $d$ is defined to be the unique $\mathbb{R}$-linear mapping $d\colon \Omega^{i}\left(\mathcal{M}\right) \to \Omega^{(i+1)}\left(\mathcal{M}\right)$ such that

\begin{itemize}

\item $df$ is the differential of $f$ for every function $f\in C^{\infty}\left(\mathcal{M}\right)$.

\item $ d\circ df$ = 0 for every function $f\in C^{\infty}\left(\mathcal{M}\right)$.

\item $d\left(\alpha \wedge \beta\right) = d\alpha\wedge\beta + (−1)^p \alpha\wedge d\beta)$, where $\alpha$ is  a $p$-form and $\beta$ is any form. 

\end{itemize}

\noindent
Since the second defining property holds in more generality, that is, $d\circ d\alpha = 0$ for any $p$-form $\alpha$, it is usually written as $d^2=d\circ d =0$. 

\paragraph{The Lie derivative $\mathcal{L}_{v}$}

The Lie derivative can be defined acting on tensor fields of any type $(r,s)$, that is, it has a well defined action on $\mathrm{T}\left(\mathcal{M}\right)$. Intuitively, the lie derivative $\mathcal{L}_{v}$ evaluates the change of a tensor field along the flow of the vector field $v$. It is defined point-wise as follows

\begin{equation}
\label{eq:liederivative}
\left(\mathcal{L}_{v} \mathfrak{T}\right)_p=\left.\frac{d}{dt}\right|_{t=0}\left((\varphi_{-t})_{_\ast}\mathfrak{T}_{\varphi_{t}(p)}\right)=\left.\frac{d}{dt}\right|_{t=0}\left((\varphi_{t})^{\ast}\mathfrak{T}\right)_p\, ,
\end{equation}

\noindent
where $\mathfrak{T}$ is a $(r,s)$ tensor field on $\mathcal{M}$ and $p\in\mathcal{M}$. It can be checked that with the definition (\ref{eq:liederivative}) $\mathcal{L}_{v}\mathfrak{T}$ is again a $(r,s)$ tensor field on $\mathcal{M}$. We now give an algebraic definition. The algebraic definition for the Lie derivative of a tensor field follows from the following four axioms

\begin{itemize}
    \item $\mathcal{L}_{v} f = v(f) $ for all $f\in C^{\infty}(\mathcal{M})$.

    \item The Lie derivative $\mathcal{L}_{v}$ obeys the Leibniz rule. That is, for any tensor fields $\mathfrak{S}$ and $\mathfrak{T}$, we have

\begin{equation}
\mathcal{L}_{v} \left( \mathfrak{S}\otimes \mathfrak{T}\right)=\left(\mathcal{L}_{v} \mathfrak{S}\right)\otimes \mathfrak{T} + \mathfrak{S}\otimes \left(\mathcal{L}_{v} \mathfrak{T}\right), .
\end{equation}

    \item The Lie derivative, when applied to forms, obeys the Leibniz rule with respect to contraction

\begin{equation}
       \mathcal{L}_{v} (\mathfrak{T}(Y_1, \ldots, Y_n)) = (\mathcal{L}_{v} \mathfrak{T})(Y_1,\ldots, Y_n) + \mathfrak{T}((\mathcal{L}_{v} Y_1), \ldots, Y_n) + \cdots + T(Y_1, \ldots, (\mathcal{L}_{v} Y_n)) 
\end{equation}

    \item The Lie derivative, when applied to forms, commutes with the de Rham differential $d$, that is

\begin{equation}
\left[\mathcal{L}_{v}, d\right] = 0\, ,
\end{equation}

\noindent
\end{itemize}

\noindent
The Lie derivative $\mathcal{L}_{v}$ can be compactly written as

\begin{equation}
\label{eq:cartansidentity}
\mathcal{L}_{v} = \iota_{v}\circ d + d\circ\iota_{v}\, ,
\end{equation}

\noindent
which is known as the \emph{Cartan formula}.
%%%%%%%%%%%%%%%%%%%%%%%%%%%%%%%%%%%%%%%%%%%%%%%%%%%%%%%%%%%%%%%%%%%%%%%%%%%%%%%%%%%%%%%%%%%%%%%%%%%
%%%%%%%%%%%%%%%%%%%%%%%%%%%%%%%%%%%%%%%%%%%%%%%%%%%%%%%%%%%%%%%%%%%%%%%%%%%%%%%%%%%%%%%%%%%%%%%%%%%

\subsection{Cartan calculus}
\label{sec:cartancalculus}

%%%%%%%%%%%%%%%%%%%%%%%%%%%%%%%%%%%%%%%%%%%%%%%%%%%%%%%%%%%%%%%%%%%%%%%%%%%%%%%%%%%%%%%%%%%%%%%%%%%
%%%%%%%%%%%%%%%%%%%%%%%%%%%%%%%%%%%%%%%%%%%%%%%%%%%%%%%%%%%%%%%%%%%%%%%%%%%%%%%%%%%%%%%%%%%%%%%%%%%

 Let us denote by $\mathfrak{X}\left(\mathcal{M}\right)$ the $C^{\infty}\left(\mathcal{M}\right)$-module of vector fields on $\mathcal{M}$. Then 

\begin{equation}
\mathfrak{X}^{\bullet}\left(\mathcal{M}\right)=\bigoplus^{\dim \mathcal{M}}_{k=0}\Lambda^{k} \mathfrak{X}\left(\mathcal{M}\right)\, ,
\end{equation}

\noindent
is a graded commutative algebra, the so-called graded commutative algebra of multivector fields, where the corresponding algebra product is given by the wedge product, denoted by $\wedge$. $\mathfrak{X}^{\bullet}\left(\mathcal{M}\right)$ can be equipped with  a degree minus one Lie bracket $[\cdot,\cdot] : \mathfrak{X}^{\bullet}\left(\mathcal{M}\right) \times \mathfrak{X}^{\bullet}\left(\mathcal{M}\right) \to \mathfrak{X}^{\bullet}\left(\mathcal{M}\right)$ that satisfies the (graded) Leibniz rule with respect to the algebra product, that is, the wedge product. $[\cdot,\cdot]$ is given by

\begin{eqnarray}
\label{eq:Schouten}
\left [ u_{1} \wedge \cdots \wedge u_{m}, v_{1} \wedge \cdots \wedge
  v_{n} \right] 
=  \sum_{i=1}^{m} \sum_{j=1}^{n} (-1)^{i+j} \left[ u_{i},v_{j}\right]
\wedge u_{1} \wedge \cdots \wedge \hat{u}_{i} \wedge  \cdots \wedge
u_{m}\wedge v_{1} \wedge \cdots \wedge \hat{v}_{j} \wedge \cdots \wedge v_{n}\nonumber\, ,
\end{eqnarray} 

\noindent
where $u_{1}\wedge\cdots\wedge u_{m}\, , \, v_{1}\wedge\cdots\wedge v_{n}\, \in \mathfrak{X}^{\bullet}\left(\mathcal{M}\right)$ and $\left[ u_{i},v_{j}\right]$ is the standard Lie bracket of vector fields. This is the so-called \emph{Schouten bracket}, and it makes $\left( \mathfrak{X}^{\bullet}\left(\mathcal{M}\right), \wedge, [\cdot,\cdot]\right)$ into a particular instance of Gerstenhaber algebra\footnote{See definition \ref{def:gealgebra}.}.

We can define also the interior product of any decomposable multivector field, say $v_{1} \wedge \cdots \wedge v_{n}$, with any $\beta \in \Omega^{\bullet}(\mathcal{M})$ is given by

\begin{equation} 
\label{eq:interior}
\iota(v_{1} \wedge \cdots \wedge v_{n}) \beta = \iota_{v_{n}} \cdots
\iota_{v_{1}} \beta,
\end{equation}

\noindent
where $\iota_{v_{i}} \beta$ is the stands for the usual interior product of vector fields and differential forms. The formula for the interior product of any multivector can be obtained by extending using $C^{\infty}\left(\mathcal{M}\right)$ linearity. 

The Lie derivative $\mathcal{L}_{v}$ of any differential form $\beta$ along any given multivector field $v \in\mathfrak{X}^{\bullet}\left(\mathcal{M}\right)$ can be written in terms of the graded commutator of $d$ and $\iota_{v}$ as follows

\begin{equation} 
\label{eq:Lie}
\mathcal{L}_{v} \beta =  d \iota_{v} \beta - (-1)^{|v|} \iota_{v} d\beta\, ,
\end{equation}

\noindent
where $\iota_{v}$ must be understood as a degree $-|v|$ operator. We will need one more identity. Let $u,v \in
\mathfrak{X}^{\bullet}\left(\mathcal{M}\right)$. Then it can be proven that

\begin{equation} 
\label{eq:commutator}
\iota_{[u,v]} \beta = (-1)^{(|u|-1)|v|} \mathcal{L}_{u} \iota_{v}  \beta - \iota_{v}\mathcal{L}_{u} \beta\, .
\end{equation}

\noindent
The graded commutative algebra of multivector fields $\mathfrak{X}^{\bullet}\left(\mathcal{M}\right)$ together with the Schouten bracket is therefore a particular instance of a Gerstenhaber algebra that can be constructed in any differential manifold $\mathcal{M}$.

%%%%%%%%%%%%%%%%%%%%%%%%%%%%%%%%%%%%%%%%%%%%%%%%%%%%%%%%%%%%%%%%%%%%%%%%%%%%%%%%%%%%%%%%%%%%%%%%%%%
%%%%%%%%%%%%%%%%%%%%%%%%%%%%%%%%%%%%%%%%%%%%%%%%%%%%%%%%%%%%%%%%%%%%%%%%%%%%%%%%%%%%%%%%%%%%%%%%%%%
%%%%%%%%%%%%%%%%%%%%%%%%%%%%%%%%%%%%%%%%%%%%%%%%%%%%%%%%%%%%%%%%%%%%%%%%%%%%%%%%%%%%%%%%%%%%%%%%%%%
%%%%%%%%%%%%%%%%%%%%%%%%%%%%%%%%%%%%%%%%%%%%%%%%%%%%%%%%%%%%%%%%%%%%%%%%%%%%%%%%%%%%%%%%%%%%%%%%%%%

\section{Fibre bundles}
\label{sec:fibrebundles}

%%%%%%%%%%%%%%%%%%%%%%%%%%%%%%%%%%%%%%%%%%%%%%%%%%%%%%%%%%%%%%%%%%%%%%%%%%%%%%%%%%%%%%%%%%%%%%%%%%%
%%%%%%%%%%%%%%%%%%%%%%%%%%%%%%%%%%%%%%%%%%%%%%%%%%%%%%%%%%%%%%%%%%%%%%%%%%%%%%%%%%%%%%%%%%%%%%%%%%%
%%%%%%%%%%%%%%%%%%%%%%%%%%%%%%%%%%%%%%%%%%%%%%%%%%%%%%%%%%%%%%%%%%%%%%%%%%%%%%%%%%%%%%%%%%%%%%%%%%%
%%%%%%%%%%%%%%%%%%%%%%%%%%%%%%%%%%%%%%%%%%%%%%%%%%%%%%%%%%%%%%%%%%%%%%%%%%%%%%%%%%%%%%%%%%%%%%%%%%%

\definition{\label{def:bundle} Let $\mathcal{F}$, $\mathcal{M}$ and $\mathcal{E}$ be differentiable manifolds and let $\pi: \mathcal{E}\to\mathcal{M}$ a differentiable map. The quadruple $\left(\mathcal{E},\pi,\mathcal{M},\mathcal{F}\right)$ is a locally trivial differentiable fibre bundle if for every $p\in\mathcal{M}$ there is an open set $\mathcal{U}$ containing $p$ and a diffeomorphism $\phi: \pi^{-1}\left(\mathcal{U}\right)\to\mathcal{U}\times\mathcal{F}$ such that the following diagram commutes

\begin{center}
\begin{tikzpicture}
\label{diag:bundle}
  \matrix (m) [matrix of math nodes,row sep=8em,column sep=9em,minimum width=2em]
  {
\pi^{-1}\left( \mathcal{U}\right) & \mathcal{U}\times F \\
& \mathcal{U} \\};
  \path[-stealth]
    (m-1-1) edge node [above] {$\phi$} (m-1-2)
    (m-1-1) edge node [below] {$\pi$} (m-2-2)
    (m-1-2) edge node [right] {$\mathrm{pr}_{1}$} (m-2-2);
\end{tikzpicture}
\end{center}

}

\noindent
where $\mathrm{pr}_{1}$ is the proyection on the first factor. $\mathcal{E}$ is the total space, $\mathcal{M}$ is the base space, $\mathcal{F}$ is the typical fibre and $\pi$ is the bundle projection. For each $p\in\mathcal{M}$, the set $\mathcal{E}_{p} \equiv \pi^{-1}\left(p\right)$ is the fibre over $p$, which is diffeomorphic to $\mathcal{F}$. The maps $\phi:\pi^{-1}\left(\mathcal{U}\right)\to\mathcal{U}\times\mathcal{F}$ are called the local trivializations of the bundle. Such a local trivialization must be of the form $\phi = \left(\pi_{\pi^{-1}\left(\mathcal{U}\right)},\Phi\right)$ where 

\begin{equation}
\Phi:\pi^{-1}\left(\mathcal{U}\right)\to\mathcal{F}\, ,
\end{equation}

\noindent
is a differentiable map such that 

\begin{equation}
\Phi_{\left|\,\mathcal{E}_{p}\right.}:\mathcal{E}_{p}\to\mathcal{F}\, ,
\end{equation}

\noindent
is a diffeomorphism. The pair $\left(\mathcal{U}, \phi\right)$, where $\phi$ is a local trivialization over the open set $\mathcal{U}\subset\mathcal{M}$ is called a bundle chart. A family $\left(\mathcal{U}_{\alpha}, \phi_{\alpha}\right)_{\alpha\in I}$ such that $\left(\mathcal{U}_{\alpha}\right)_{\alpha\in I}$ is a cover of $\mathcal{M}$ is a bundle atlas. Given two different bundle charts $\left(\mathcal{U}_{\alpha},\phi_{\alpha}\right)$ and $\left(\mathcal{U}_{\beta},\phi_{\beta}\right)$ such that $\mathcal{U}_{\alpha}\cap\mathcal{U}_{\beta}\neq\emptyset$ we have the overlap map

\begin{equation}
\phi_{\alpha}\circ\phi^{-1}_{\beta}: \mathcal{U}_{\alpha}\cap\mathcal{U}_{\beta}\times\mathcal{F}\to \mathcal{U}_{\alpha}\cap\mathcal{U}_{\beta}\times \mathcal{F}\, ,
\end{equation}

\noindent
which can be written as follows

\begin{equation}
\phi_{\alpha}\circ\phi^{-1}_{\beta}(p,q) = \left(p,\Phi_{\alpha\beta}(p)(q)\right)\, , \qquad p\in\mathcal{M}\, , \qquad q\in\mathcal{F}\, ,
\end{equation}

\noindent
where $\Phi_{\alpha\beta}:\mathcal{U}_{\alpha}\cap\mathcal{U}_{\beta}\to \mathrm{Diff}\left(\mathcal{F}\right)$ is given by

\begin{equation}
p\mapsto \Phi_{\alpha\beta}(p) = \Phi_{\alpha\left|\,\mathcal{E}_{p}\right.} \circ \Phi^{-1}_{\beta\left|\,\mathcal{E}_{p}\right.}\, .
\end{equation}

\noindent
The functions $\Phi_{\alpha\beta}$ are called the transition maps, and satisfy

\begin{itemize}

\item $\Phi_{\alpha\alpha}(p) = \mathrm{Id}_{\mathrm{Diff}(\mathcal{F})}\, ,\qquad p\in\mathcal{U}_{\alpha}\, ,$

\item $\Phi_{\alpha\beta}(p) = \Phi_{\beta\alpha}(p)^{-1}\, ,\qquad p\in\mathcal{U}_{\alpha}\cap\mathcal{U}_{\beta}\, ,$

\item $\Phi_{\alpha\beta}(p)\circ \Phi_{\beta\gamma}(p) \circ \Phi_{\gamma\alpha} (p) = \mathrm{Id}_{\mathrm{Diff}(\mathcal{F})}\, ,\qquad p\in\mathcal{U}_{\alpha}\cap\mathcal{U}_{\beta}\cap\mathcal{U}_{\gamma}\, ,$

\end{itemize}

\noindent
for all $\alpha\, , \beta\, ,\gamma \in I$. The characterization just given of the transition maps as a map to the diffeomorphisms of $\mathcal{F}$ can be usually restricted to a map onto a lie group $G$ acting on $\mathcal{F}$ by a particular action $\Psi: G\times\mathcal{F}\to\mathcal{F}$. The reader is invited to consult \cite{JMLee,DHbundle} for more details.

\definition{\label{def:bundlemorphism} Let $\xi_{1}=\left(\mathcal{E}_{1},\pi_{1},\mathcal{M}_{1},\mathcal{F}_{1}\right)$ and $\xi_{2}=\left(\mathcal{E}_2,\pi_2,\mathcal{M}_2,\mathcal{F}_2\right)$ differentiable fibre bundles. A morphism from $\xi_{1}$ to $\xi_{2}$ is a couple of maps $F:\mathcal{E}_{1}\to\mathcal{E}_2$ and $f:\mathcal{M}_{1}\to\mathcal{M}_2$ such that the following diagram commutes

\begin{center}
\begin{tikzpicture}
\label{diag:bundlemorphism}
  \matrix (m) [matrix of math nodes,row sep=8em,column sep=9em,minimum width=2em]
  {
\mathcal{E}_{1} & \mathcal{E}_{2} \\
\mathcal{M}_{1} & \mathcal{M}_{2} \\};
  \path[-stealth]
    (m-1-1) edge node [above] {$F$} (m-1-2)
    (m-1-1) edge node [left] {$\pi_{1}$} (m-2-1)
    (m-2-1) edge node [above] {$f$} (m-2-2)
    (m-1-2) edge node [right] {$\pi_{2}$} (m-2-2);
\end{tikzpicture}
\end{center}

}

\noindent
If $F$ and $f$ are diffeomorphisms, then $\left(F,f\right):\xi_{1}\to\xi_{2}$ is a bundle isomorphism.

\definition{\label{def:bundlesection} A differentiable global section of a fibre bundle $\xi=\left(\mathcal{E},\pi,\mathcal{M},\mathcal{F}\right)$ is a differentiable map $\sigma:\mathcal{M}\to \mathcal{E}$ such that $\pi\circ\sigma=\mathrm{Id}_{\mathcal{M}}$. A differentiable local section over an open set $\mathcal{U}$ is a differentiable map $\sigma:\mathcal{M}\to \mathcal{U}$ such that $\pi\circ\sigma=\mathrm{Id}_{\mathcal{U}}$.

}

\noindent
The set of differentiable sections of $\xi$ is denoted by $\Gamma\left(\xi\right)$ or $\Gamma\left(\mathcal{E}\right)$. Notice that a fibre bundle may not have any global section.

\definition{\label{def:vectorbundle} Let $V$ be a finite dimensional vector space over the complex or real numbers. A smooth vector bundle with typical fibre $V$ is a fibre bundle $\left(\mathcal{E},\pi,\mathcal{M},V\right)$ such that

\begin{itemize}

\item for each $p\in\mathcal{M}$ we have that $\mathcal{E}_{p}=\pi^{-1}(p)$ is a vector space isomorphic to $V$.

\item for every $p\in\mathcal{M}$ there exist a bundle chart $\left(\mathcal{U}_{\alpha},\phi_{\alpha}\right)$ containing $p$ such that 

\begin{equation}
\Phi_{\left|\, \mathcal{E}_{p}\right.}:\mathcal{E}_{p}\to V\, ,
\end{equation}

\noindent
is a vector space isomorphism, where $\phi = \left(\pi_{\pi^{-1}\left(\mathcal{U}\right)},\Phi\right)$.

\end{itemize}

}

\noindent
The typical example of vector bundle is the tangent bundle $T\mathcal{M}$ over a manifold $\mathcal{M}$. The notion of bundle morphism, given in definition (\ref{def:bundlemorphism}) specializes to vector bundles by requiring $F_{\left|\,\pi^{-1}_{1}(p)\right.}: \pi^{-1}_{1}(p) \to \pi^{-1}_{2}\left(f(p)\right)$ to be linear.

Given two vector bundles $\pi_{1}: \mathcal{E}_{1}\to \mathcal{M}$ and $\pi_{2}: \mathcal{E}_{2}\to \mathcal{M}$, we can define the \emph{Whitney sum bundle} $\pi_{1}\oplus\pi_{2}: \mathcal{E}_{1}\oplus\mathcal{E}_{2}\to\mathcal{M}$ such that the fibre at a point $p\in\mathcal{M}$ is given by $\left(\mathcal{E}_{1}\oplus\mathcal{E}_{2}\right)_{p} = \mathcal{E}_{1\, p}\oplus\mathcal{E}_{2\, p}$.

The pull-back of a vector bundle $\pi:\mathcal{E}\to\mathcal{M}$ by a smooth map $f:\mathcal{N}\to\mathcal{M}$, where $\mathcal{N}$ is a differentiable manifold, is the vector bundle $\left(f^{\ast}\mathcal{E}\right)$ over $\mathcal{N}$ defined as follows

\begin{equation}
f^{\ast}\mathcal{E} = \left\{ (q,e)\in \mathcal{N}\times\mathcal{E}\, \left|\right.\, f(q) = \pi(e) \right\}\subset \mathcal{N}\times\mathcal{E}\, ,
\end{equation}

\noindent
and equipped with the subspace topology and the projection map $\mathrm{pr}_{1} : f^{\ast}\mathcal{E} \to \mathcal{N}$ given by the projection onto the first factor

\begin{equation}
\mathrm{pr}_{1}(q,e) = q\, .
\end{equation}

\noindent
Notice that the following diagram commutes

\begin{center}
\begin{tikzpicture}
  \matrix (m) [matrix of math nodes,row sep=8em,column sep=9em,minimum width=2em]
  {
f^{\ast}\mathcal{E} & \mathcal{E} \\
\mathcal{N} & \mathcal{M} \\};
  \path[-stealth]
    (m-1-1) edge node [above] {$\mathrm{pr}_{2}$} (m-1-2)
    (m-1-1) edge node [left] {$\mathrm{pr}_{1}$} (m-2-1)
    (m-2-1) edge node [above] {$f$} (m-2-2)
    (m-1-2) edge node [right] {$\pi$} (m-2-2);
\end{tikzpicture}
\end{center}

\noindent
where $\mathrm{pr}_{2}$ is the projection on the second factor. If $\left(\mathcal{U}, \phi\right)$ is a local trivialization of $\mathcal{E}$, then $\left(f^{-1}\left(\mathcal{U}\right), \psi\right)$ is a local trivialization of $f^{\ast}\mathcal{E}$ where

\begin{equation}
\psi(q,e) = \left(q, \mathrm{pr}_2\left(\phi(e)\right)\right)\, , \qquad\forall\,\, (q,e)\, \in f^{\ast}\mathcal{E}\, . 
\end{equation}

\noindent
Therefore, the fibre at a point $q\in\mathcal{N}$ is given by

\begin{equation}
\left(f^{\ast}\mathcal{E}\right)_{q} = \mathcal{E}_{f(q)}\, .
\end{equation}

\noindent 
A section $\sigma\in\Gamma\left(\mathcal{E}\right)$ induces a section $f^{\ast}\sigma\in\Gamma\left(f^{\ast}\mathcal{E}\right)$ defined by $f^{\ast}\sigma=\sigma\circ f$.

\begin{ep}
\label{ep:tangentpullback}
As an example of pull-back of a vector bundle we are going to consider the pull-back of the tangent bundle $T\mathcal{M}$ of a differentiable manifold $\mathcal{M}$. Let $\mathcal{N}$ be a differentiable manifold and let

\begin{equation}
f : \mathcal{N}\to\mathcal{M}\, 
\end{equation}

\noindent
be a map. The pull-back bundle is defined as follows

\begin{equation}
f^{\ast}T\mathcal{M} = \left\{ (q,e)\in \mathcal{N}\times T\mathcal{M}\, \left|\right.\, f(q) = \pi(e) \right\}\subset \mathcal{N}\times T\mathcal{M}\, .
\end{equation}

\noindent
Notice that the following diagram commutes

\begin{center}
\begin{tikzpicture}
\label{diag:bundlemorphismpullback}
  \matrix (m) [matrix of math nodes,row sep=8em,column sep=9em,minimum width=2em]
  {
f^{\ast} T\mathcal{M} & T\mathcal{M} \\
\mathcal{N} & \mathcal{M} \\};
  \path[-stealth]
    (m-1-1) edge node [above] {$\mathrm{pr}_{2}$} (m-1-2)
    (m-1-1) edge node [left] {$\mathrm{pr}_{1}$} (m-2-1)
    (m-2-1) edge node [above] {$f$} (m-2-2)
    (m-1-2) edge node [right] {$\pi$} (m-2-2);
\end{tikzpicture}
\end{center}

\noindent
Notice that in general $f^{\ast} T\mathcal{M}$ is not equal to $T\mathcal{N}$. Only when $f$ is a diffeomorphism we have $f^{\ast} T\mathcal{M} \simeq T\mathcal{N}$.

\end{ep}

%%%%%%%%%%%%%%%%%%%%%%%%%%%%%%%%%%%%%%%%%%%%%%%%%%%%%%%%%%%%%%%%%%%%%%%%%%%%%%%%%%%%%%%%%%%%%%%%%%%
%%%%%%%%%%%%%%%%%%%%%%%%%%%%%%%%%%%%%%%%%%%%%%%%%%%%%%%%%%%%%%%%%%%%%%%%%%%%%%%%%%%%%%%%%%%%%%%%%%%
%%%%%%%%%%%%%%%%%%%%%%%%%%%%%%%%%%%%%%%%%%%%%%%%%%%%%%%%%%%%%%%%%%%%%%%%%%%%%%%%%%%%%%%%%%%%%%%%%%%
%%%%%%%%%%%%%%%%%%%%%%%%%%%%%%%%%%%%%%%%%%%%%%%%%%%%%%%%%%%%%%%%%%%%%%%%%%%%%%%%%%%%%%%%%%%%%%%%%%%

\section{Lie groups}
\label{sec:liegroups}

%%%%%%%%%%%%%%%%%%%%%%%%%%%%%%%%%%%%%%%%%%%%%%%%%%%%%%%%%%%%%%%%%%%%%%%%%%%%%%%%%%%%%%%%%%%%%%%%%%%
%%%%%%%%%%%%%%%%%%%%%%%%%%%%%%%%%%%%%%%%%%%%%%%%%%%%%%%%%%%%%%%%%%%%%%%%%%%%%%%%%%%%%%%%%%%%%%%%%%%
%%%%%%%%%%%%%%%%%%%%%%%%%%%%%%%%%%%%%%%%%%%%%%%%%%%%%%%%%%%%%%%%%%%%%%%%%%%%%%%%%%%%%%%%%%%%%%%%%%%
%%%%%%%%%%%%%%%%%%%%%%%%%%%%%%%%%%%%%%%%%%%%%%%%%%%%%%%%%%%%%%%%%%%%%%%%%%%%%%%%%%%%%%%%%%%%%%%%%%%

\definition{\label{def:Liegroup} A smooth manifold $G$ is called a Lie group if it is also an abstract group such that the multiplication map $\cdot$ and the inverse map are $C^{\infty}$ maps.}

\definition{\label{def:lrtraslations} For a Lie group $G$ and an element $g\in G$, the maps $L_{g}: G\to G$ and $R_{g}: G\to G$ are defined by

\begin{equation}
L_{g}(h)=g\cdot h\, ,\qquad R_{g}(h) = h\cdot g\, , \qquad \forall h\in G\, .
\end{equation}

}

\definition{\label{def:leftinvectors} A vector field $X\in\mathfrak{X}\left(G\right)$ is called left-invariant if and only if $L_{g\ast}X=X$ for all $g\in G$. The set of left-invariant vectors is denoted by $\mathfrak{X}^{\mathrm{L}}\left(G\right)$. Similar considerations apply to right-invariant vectors.}

\noindent
$\mathfrak{X}^{\mathrm{L}}\left(G\right)$ is closed under the Lie bracket operation $\left[\cdot, \cdot\right] : G\times G\to G$. The set of left invariant vectors $\mathfrak{X}^{\mathrm{L}}\left(G\right)$ equipped with the binary operation $\left[\cdot, \cdot\right]$ is the lie algebra $\mathfrak{g}$ of $G$. Let us define now the adjoint representation of a Lie group. Given an element $g\in G$, we define the map $C_{g}$ as follows

\begin{eqnarray}
C_{g}:G &\to & G\nonumber\\
h &\mapsto & g h g^{-1}\, .
\end{eqnarray}

\noindent
$C_{g}:G\to G$ is a Lie group automorphism called the conjugation map. The corresponding tangent map is given by

\begin{eqnarray}
\left(dC_{g}\right)|_{h} : T_{h}G &\to & T_{ghg^{-1}}G \, .
\end{eqnarray}

\noindent
Therefore, for each $g\in G$, $\left(dC_{g}\right)|_{e} : T_{e}G \to T_{e}G$ is an automorphism of the tangent space of $G$ at the identity. In other words, $\left(dC_{g}\right)|_{e}$ is an automorphism of the Lie algebra $\mathfrak{g}$ of $G$, which respects the Lie bracket on $\mathfrak{g}$ and thus is a Lie-algebra endomorphism. We will call it the adjoint map $\mathrm{Ad}_{g}\equiv\left(dC_{g}\right)|_{e} : \mathfrak{g}\to\mathfrak{g}$. We define also the map $C: g\to C_{g}$, which is a group homomorphism from $G$ to $\mathrm{Aut}\left(G\right)$.

The map $\mathrm{Ad}: g\to \mathrm{Ad}_{g}$ is a Lie group homomorphism $G\to\mathrm{End}(\mathfrak{g})$ called the adjoint representation of $G$. Here $\mathrm{End}(\mathfrak{g})$ denotes the group of linear Lie-algebra endomorphisms of $\mathfrak{g}$. Therefore, Ad assigns to every element $g\in G$ an element of $\mathrm{Gl}\left(\mathfrak{g}\right)$ and therefore is a representation of $G$ on the vector space $\mathfrak{g}$. The adjoint representation ad of $\mathfrak{g}$ can be obtained from the adjoint representation of $G$ as follows

\begin{equation}
\mathrm{ad} =  d\left(\mathrm{Ad}\right)\, .
\end{equation}

\noindent
One can show that $\mathrm{ad}\left(v\right) w = \left[ v, w \right]\, ,\,\, \forall\,\, v, w \in\mathfrak{g}$. We proceed now to define the left action of a Lie group $G$ on a differentiable manifold $\mathcal{M}$. The right action is defined similarly.

\definition{\label{def:groupaction} A left action of a Lie group $G$ on $\mathcal{M}$ is a group homomorphism 

\begin{eqnarray}
\psi:\,\, G &\to &\mathrm{Diff}\left(\mathcal{M}\right)\nonumber\\
       g &\mapsto &\psi_{g} \, .
\end{eqnarray}

}

\definition{\label{def:groupactionevaluation} The evaluation map associated with an action $\psi:\, G \to \mathrm{Diff}\left(\mathcal{M}\right)$ is

\begin{eqnarray}
l_{\psi}:\,\, \mathcal{M}\times G &\to &\mathcal{M}\nonumber\\
       (p,g) &\mapsto &\psi_{g}(p) \, .
\end{eqnarray}

}

\noindent
The map $l: G\times G\to G$ given by  $l(g,h) = L_{g}(h)$ is an example of left action of $G$ onto itself. Given a left action $l: G\times\mathcal{M}\to\mathcal{M}$ and a fixed point $p\in\mathcal{M}$, the isotropy group of $p$ is defined to be

\begin{equation}
G_{p} = \left\{ g\in G : gp = p \right\}\, ,
\end{equation}

\noindent
which is a closed Lie subgroup of $G$.

\definition{\label{def:properaction} Let $l: G\times\mathcal{M}\to\mathcal{M}$ be a left group action. If the map $P: G\times\mathcal{M}\to\mathcal{M}\times\mathcal{M}$ given by $(g,p)\to \left(l(g,p),p\right)$ is proper, the action is said to be proper.}

\noindent
\begin{ep}
\label{ep:Raction} If $v$ is a complete vector field on $\mathcal{M}$, then

\begin{eqnarray}
\rho:\,\,\mathbb{R}&\to &\mathrm{Diff}\left(\mathcal{M}\right)\nonumber\\
                 t &\mapsto & \rho_{t}\, ,
\end{eqnarray}

\noindent
is a smooth action of $\mathbb{R}$ on $\mathcal{M}$.
\end{ep}

\noindent
Let $\psi :\, G\to\mathrm{Diff}\left(\mathcal{M}\right)$ be an action.

\definition{\label{def:orbitspace} The orbit of $G$ through $p\in\mathcal{M}$ is $O_{p} = \left\{\psi_{g}(p)\,\, |\,\, g\in G\right\}$.

} 

\definition{\label{def:actions} An action $\psi :\, G\to\mathrm{Diff}\left(\mathcal{M}\right)$ is said to be

\begin{itemize}

\item {\bf transitive} if $O_{p} = \mathcal{M}\, ,\,\,\, p\in\mathcal{M}$.

\item {\bf free} if $G_{p}$ is trivial $\forall\,\, p\in\mathcal{M}$.

\item {\bf locally free} if $G_{p}$ is discrete $\forall\,\, p\in\mathcal{M}$.

\end{itemize}

}

\noindent
Let $\sim$ be the orbit equivalence relation on $\mathcal{M}$ defined by

\begin{equation}
p \sim q \,\Leftrightarrow\, p,q\in O_{p}\, .
\end{equation}

\noindent
The space of orbits $\mathcal{M}/\sim = \mathcal{M}/G$ is called the orbit space. Let

\begin{eqnarray}
\pi:\, \mathcal{M} &\to &\mathcal{M}/G\nonumber\\
                p &\mapsto & G_{p}\, .
\end{eqnarray}

\noindent
be the point orbit projection. Then $\mathcal{M}/G$ can be equipped with the weakest topology for which $\pi$ is continuous, namely, $\mathcal{U}\subseteq\mathcal{M}/G$ is open if and only if $\pi^{-1}\left(\mathcal{U}\right)$ is open in $\mathcal{M}$.

%%%%%%%%%%%%%%%%%%%%%%%%%%%%%%%%%%%%%%%%%%%%%%%%%%%%%%%%%%%%%%%%%%%%%%%%%%%%%%%%%%%%%%%%%%%%%%%%%%%
%%%%%%%%%%%%%%%%%%%%%%%%%%%%%%%%%%%%%%%%%%%%%%%%%%%%%%%%%%%%%%%%%%%%%%%%%%%%%%%%%%%%%%%%%%%%%%%%%%%
%%%%%%%%%%%%%%%%%%%%%%%%%%%%%%%%%%%%%%%%%%%%%%%%%%%%%%%%%%%%%%%%%%%%%%%%%%%%%%%%%%%%%%%%%%%%%%%%%%%
%%%%%%%%%%%%%%%%%%%%%%%%%%%%%%%%%%%%%%%%%%%%%%%%%%%%%%%%%%%%%%%%%%%%%%%%%%%%%%%%%%%%%%%%%%%%%%%%%%%

\section{Symplectic Geometry}
\label{sec:symplectic}

%%%%%%%%%%%%%%%%%%%%%%%%%%%%%%%%%%%%%%%%%%%%%%%%%%%%%%%%%%%%%%%%%%%%%%%%%%%%%%%%%%%%%%%%%%%%%%%%%%%
%%%%%%%%%%%%%%%%%%%%%%%%%%%%%%%%%%%%%%%%%%%%%%%%%%%%%%%%%%%%%%%%%%%%%%%%%%%%%%%%%%%%%%%%%%%%%%%%%%%
%%%%%%%%%%%%%%%%%%%%%%%%%%%%%%%%%%%%%%%%%%%%%%%%%%%%%%%%%%%%%%%%%%%%%%%%%%%%%%%%%%%%%%%%%%%%%%%%%%%
%%%%%%%%%%%%%%%%%%%%%%%%%%%%%%%%%%%%%%%%%%%%%%%%%%%%%%%%%%%%%%%%%%%%%%%%%%%%%%%%%%%%%%%%%%%%%%%%%%%

\definition{\label{def:symplecticvector} Let $V$ be a vector space. The pair $(V, \omega)$ is a symplectic vector space if $\omega\in \Lambda^{2} V^{\ast}$ is non-degenerate, that is, if the kernel

\begin{equation}
\mathrm{ker}\, \omega \equiv \left\{ v\in V\,\, | \,\,\omega(v,w) = 0\, ,\,\,\forall w\in V\right\}\, ,
\end{equation}

\noindent
is trivial.
}

\noindent
Two symplectic vector spaces $\left(V_{1},\, \omega_{1}\right)$ and $\left(V_{1},\, \omega_{1}\right)$ are called \emph{symplectomorphic} if there is an isomorphism $F:V_{1}\to V_{2}$ that relates the symplectic structures, that is, $F^{\ast}\omega_{2} = \omega_{1}$. Given a symplectic vector space $(V,\,\omega)$, the group of automorphisms that preserve $\omega$ is denoted by $\mathrm{Sp}\left( V\right)$, which is a Lie subgroup of $\mathrm{Gl}\left( V\right)$, the group of automorphisms of $V$.

\begin{ep}
\label{ep:symplecticvector1} The simplest and arguably most important example of symplectic vector space consists of $\mathbb{R}^{2n}$, for some $n\in\mathbb{N}$, with basis $\left\{e_{1},\dots , e_{n}, f_{1},\dots , f_{n}\right\}$ equipped with the bilinear form $\omega$ given by

\begin{equation}
\omega\left( e_{i}, e_{j}\right) = 0\, , \quad  \omega\left( f_{i}, f_{j}\right) = 0\, , \quad \omega\left( e_{i}, f_{j}\right) = -\omega\left( f_{j}, e_{i}\right) = \delta_{ij}\, .
\end{equation}

\noindent
$\omega$ is a symplectic structure on $\mathbb{R}^{2n}$. It can be easily seen that every symplectic vector space of dimension $2n$ is symplectomorphic to $\left(\mathbb{R}^{2n}, \omega\right)$.

\end{ep}

\noindent
Another standard examples of symplectic vector space are the following

\begin{ep}
\label{ep:symplecticvector3} Let $E$ be a complex vector space of dimension $n$, equipped with a complex, positive definite inner product $h: E\times E\to\mathbb{C}$. Then $E$, taken as a real vector space, equipped with the bilinear form $\omega = \Im{\rm m}\left(h\right)$, is a symplectic vector space. The condition $h\left(v_{1},v_{2}\right) =\overline{h\left(v_{2},v_{1}\right)}$ translates into the antisymmetry of $\omega =\Im{\rm m}\left(h\right)$ as a real form on $V$.
\end{ep}

\definition{\label{def:smanifold}  A symplectic manifold $(\mathcal{M},\omega)$ is a real manifold equipped with a smooth, point-wise non-degenerate , global section $\omega$ of $\Lambda^2T^{*}\mathcal{M} $. Therefore, at every point $p\in\mathcal{M}$, $T_{p}\mathcal{M}$ is a symplectic vector space equipped with the symplectic form $\omega |_{p}$. }

\noindent
Given two symplectic manifolds $\left(M_{1},\,\omega_{1}\right)$ and $\left(M_{2},\,\omega_{2}\right)$, a symplectomorphism is a diffeomorphism $F: \mathcal{M}_{1}\to\mathcal{M}_{2}$ such that $F^{\ast}\omega_{2} = \omega_{1}$. The group of symplectomorphisms of a symplectic manifold $\left(M,\omega\right)$ onto itself is denoted by $\mathrm{Symp}\left(\mathcal{M},\, \omega\right)$. The non-degeneracy of $\omega$ has very important consequences, for instance

\begin{itemize}

\item Every symplectic manifold is orientable, with volume form given by the Liouville form 

\begin{equation}
\Lambda_{L} = \frac{\omega^{n}}{n!}\neq 0\, ,
\end{equation}

\noindent
where $2n$ is the dimension of $\mathcal{M}$.

\item There exists an isomorphism between the vector fields and one-forms on $\mathcal{M}$, given by

\begin{eqnarray}
\label{eq:ismvectorforms}
\tilde{\omega} : \,\mathfrak{X}_{\mathrm{Ham}}\left(\mathcal{M}\right) &\to & \Omega\left(\mathcal{M}\right)\nonumber\\
v &\mapsto & \iota_{v}\omega\, .
\end{eqnarray}

\end{itemize}

\noindent
We denote by $\mathfrak{X}_{\mathrm{Sym}}\left(\mathcal{M}\right)$ the set of all vector fields in $\mathcal{M}$ that preserves $\omega$, that is

\begin{equation}
\label{eq:symvectorfield}
\mathcal{L}_{v} \omega = d\iota_{v}\omega + \iota_{v}d\omega = d\iota_{v}\omega = 0\, ,
\end{equation}

\noindent
where we have used (\ref{eq:cartansidentity}). An element of $\mathfrak{X}_{\mathrm{Sym}}\left(\mathcal{M}\right)$ is called a symplectic vector field, and generates symplectomorphisms through the corresponding flow \ref{def:flowvectorfield}.  As already mentioned, due to the non-degeneracy of $\omega$, for every one-form $\xi\in\Omega\left(\mathcal{M}\right)$ there exists a unique vector field $v\in\mathfrak{X}\left(\mathcal{M}\right)$ such that

\begin{equation}
\label{eq:hamvectorxi}
\iota_{v}\omega = \xi\, .
\end{equation}

\noindent
In particular, for any function $f\in C^{\infty}\left(\mathcal{M}\right)$ there exists a unique vector field $v_{f}\in\mathfrak{X}\left(\mathcal{M}\right)$ such that

\begin{equation}
\label{eq:hamvector}
\iota_{v_{f}}\omega = -df\, .
\end{equation}

\noindent
$\iota_{v_{f}}$ is the so-called Hamiltonian vector field associated to $f$. The space of vector fields satisfying equation (\ref{eq:hamvector}) for some function $f\in C^{\infty}\left(\mathcal{M}\right)$ is denoted by $\mathfrak{X}_{\mathrm{Ham}}\left(\mathcal{M}\right)$, the space of hamiltonian vector fields.

\prop{\label{prop:hamissymplectic} Every Hamiltonian vector field is a symplectic vector field. That is

\begin{equation}
\mathfrak{X}_{\mathrm{Ham}}\left(\mathcal{M}\right)\subseteq \mathfrak{X}_{\mathrm{Sym}}\left(\mathcal{M}\right)\, .
\end{equation}

}

\proof{Given a Hamiltonian vector field $v_{f}$ respect to some function $f\in C^{\infty}\left(\mathcal{M}\right)$, we have, using Cartan's identity (\ref{eq:cartansidentity}), $\mathcal{L}_{v_{f}}\omega = d^2 f = 0$. \qed}

\noindent
Therefore, from Eq. (\ref{eq:symvectorfield}) and (\ref{eq:hamvector}) respectively, we see that $\iota_{v}\omega$ is closed for a symplectic vector field and exact for a hamiltonian vector field. Restricting $\tilde{\omega}$ in Eq. (\ref{eq:ismvectorforms}) to the set of hamiltonian vector fields $\mathfrak{X}_{\mathrm{Ham}}\left(\mathcal{M}\right)$ we obtain a new isomorphism 

\begin{equation}
\omega |_{\mathfrak{X}_{\mathrm{Ham}}} :\, \mathfrak{X}_{\mathrm{Ham}}\left(\mathcal{M}\right)\to B^{1}\left(\mathcal{M}\right)\, ,
\end{equation}

\noindent
where $B^{1}\left(\mathcal{M}\right) = \Omega^{1}\left(\mathcal{M}\right)\cap \mathrm{Im}\, d$ is the space of exact one-forms. Similarly, restricting $\tilde{\omega}$ to the set of symplectic vector fields $\mathfrak{X}_{\mathrm{Sym}}\left(\mathcal{M}\right)$ we obtain another isomorphism

\begin{equation}
\omega |_{\mathfrak{X}_{\mathrm{Sym}}} :\, \mathfrak{X}_{\mathrm{Ham}}\left(\mathcal{M}\right)\to Z^{1}\left(\mathcal{M}\right)\, ,
\end{equation}

\noindent
where $Z^{1}\left(\mathcal{M}\right) = \Omega^{1}\left(\mathcal{M}\right)\cap \mathrm{Ker}\, d$ is now the space of closed one-forms. Therefore, the quotient of symplectic and hamiltonian vector fields is just the first de Rham cohomology group

\begin{equation}
H^{1}\left(\mathcal{M}\right) = \frac{\mathfrak{X}_{\mathrm{Sym}}\left(\mathcal{M}\right)}{\mathfrak{X}_{\mathrm{Ham}}\left(\mathcal{M}\right)}\, .
\end{equation}

\noindent
Therefore, the following exact sequence of vector spaces holds

\begin{equation}
\label{eq:sechamsym}
0\to \mathfrak{X}_{\mathrm{Ham}}\left(\mathcal{M}\right) \to \mathfrak{X}_{\mathrm{Sym}}\left(\mathcal{M}\right)\to H^{1}\left(\mathcal{M}\right)\to 0\, .
\end{equation}

\noindent
Hence, if $H^{1}\left(\mathcal{M}\right) = 0$, every symplectic vector field is Hamiltonian. 

\prop{\label{prop:commham} Given two symplectic vector fields $v_{f_{1}}, v_{f_{2}}\in\mathfrak{X}_{\mathrm{Sym}}\left(\mathcal{M}\right)$ we have that $[v_{f_{1}},v_{f_{2}}]$ is Hamiltonian, with Hamiltonian function $\omega (v_{f_{1}},v_{f_{2}})$. }

\proof{Let $v_{f_{1}}, v_{f_{2}}\in\mathfrak{X}_{\mathrm{Sym}}\left(\mathcal{M}\right)$. Then we have

\begin{equation}
d\omega\left(v_{f_{1}}, v_{f_{2}}\right) = di_{v_{f_{2}}} i_{v_{f_{1}}}\omega = \mathcal{L}_{v_{f_{2}}} i_{v_{f_{1}}}\omega - i_{v_{f_{2}}}di_{v_{f_{1}}}\omega = i_{\mathcal{L}_{v_{f_{2}}}v_{f_{1}}}\omega = - i_{\left[ v_{f_{1}}, v_{f_{2}}\right]}\omega\, .
\end{equation}

\qed
}

\noindent
Therefore, $\left[ \mathfrak{X}_{\mathrm{Sym}}\left(\mathcal{M}\right), \mathfrak{X}_{\mathrm{Sym}}\left(\mathcal{M}\right)\right] \subseteq \mathfrak{X}_{\mathrm{Ham}}\left(\mathcal{M}\right)$, and in particular $\mathfrak{X}_{\mathrm{Ham}}\left(\mathcal{M}\right)$ is an ideal in the Lie algebra $\mathfrak{X}_{\mathrm{Sym}}\left(\mathcal{M}\right)$, and the quotient Lie algebra is abelian. Hence, \ref{eq:sechamsym} is an exact sequence of Lie algebras, where $H^{1}\left(\mathcal{M}\right)$ carries the trivial Lie algebra structure.

Consider now the following surjective map

\begin{eqnarray}
h:\, C^{\infty}\left(\mathcal{M}\right) &\to &  \mathfrak{X}_{\mathrm{Ham}}\left(\mathcal{M}\right)\nonumber\\
f &\mapsto & v_{f}\, .
\end{eqnarray}

\noindent
The kernel of $h$ is the space $Z^{0}\left(\mathcal{M}\right) = H^{0}\left(\mathcal{M}\right)$ of locally constant functions. Therefore, we can write the following exact sequence of vector spaces 

\begin{equation}
\label{eq:sechamsymII}
0\to Z^{0}\left(\mathcal{M}\right) \to C^{\infty}\left(\mathcal{M}\right) \to \mathfrak{X}_{\mathrm{Ham}}\left(\mathcal{M}\right)\to 0\, .
\end{equation}

\noindent
It is possible to define a Lie algebra structure on $C^{\infty}\left(\mathcal{M}\right)$ such that (\ref{eq:sechamsymII}) is an exact sequence of Lie algebras. 

\definition{\label{def:poissonbracket} Let $\left(\mathcal{M},\omega\right)$ be a symplectic manifold. The Poisson bracket of two funcions $f, g \in C^{\infty}\left(\mathcal{M}\right)$ is defined as

\begin{equation}
\left\{ f, g\right\} = \omega( v_{f}, v_{g})\, .
\end{equation}

}

\noindent
The Poisson bracket is anti-symmetric. Using Cartan's identity (\ref{eq:cartansidentity}), the Poisson bracket can be rewritten as follows

\begin{equation}
\left\{ f, g\right\} = \mathcal{L}_{v_{f}} g = -\mathcal{L}_{v_{g}} f\, .
\end{equation}

\noindent
Therefore, if $\left\{ f, g\right\} = 0$, then $f$ is constant along solution curves of $v_{g}$ and vice-versa.

\prop{\label{prop:poissonalgebra} The Poisson bracket defines a Lie algebra structure into $C^{\infty}\left(\mathcal{M}, \mathbb{R}\right)$. The map

\begin{eqnarray}
 C^{\infty}\left(\mathcal{M}\right) &\to &  \mathfrak{X}\left(\mathcal{M}\right)\nonumber\\
f &\mapsto & v_{f}\, ,
\end{eqnarray}

\noindent
is a Lie algebra isomorphism, that is 

\begin{equation}
\label{eq:liealgiso}
v_{\left\{f,g\right\}} = [ v_{f} , v_{g} ]\, .
\end{equation}
}

\proof{We have to prove that the Poisson bracket satisfies the Jacobi identity. This follows from 

\begin{equation}
\left\{f,\left\{ g, h\right\}\right\} = \mathcal{L}_{v_{f}}\left\{ g, h\right\} = \omega\left(\left[v_{f}, v_{g}\right], v_{h}\right) + \omega\left( v_{g}, \left[v_{f}, v_{h}\right]\right) = \omega\left(v_{\left\{ f, g\right\}}, v_{h}\right) + \omega\left( v_{g}, v_{_{\left\{ f, g\right\}}}\right) = \left\{h,\left\{ f, g\right\}\right\} + \left\{g,\left\{ f, h\right\}\right\}\, .
\end{equation}

\noindent
Equation (\ref{eq:liealgiso}) is a particular instance of proposition \ref{prop:commham}. \qed

}

\prop{\label{prop:poissonalgebraII} The algebra $\left( C^{\infty}\left(\mathcal{M}, \mathbb{R}\right), \left\{ \cdot, \cdot\right\}\right)$ is a Poisson algebra\footnote{See definition \ref{def:poissonalgebra}.}.

}

\proof{We have to proof that the Poisson bracket satisfies equation (\ref{eq:poissonalgebra}). Indeed we have

\begin{equation}
\left\{ fg, h\right\} = -\mathcal{L}_{v_{h}} (fg) =  -\mathcal{L}_{v_{h}} (f) g -  f\mathcal{L}_{v_{h}} (g) = -\left\{ f, h\right\} g - f \left\{ g, h\right\}\, .
\end{equation}

\qed
}

%%%%%%%%%%%%%%%%%%%%%%%%%%%%%%%%%%%%%%%%%%%%%%%%%%%%%%%%%%%%%%%%%%%%%%%%%%%%%%%%%%%%%%%%%%%%%%%%%%%
%%%%%%%%%%%%%%%%%%%%%%%%%%%%%%%%%%%%%%%%%%%%%%%%%%%%%%%%%%%%%%%%%%%%%%%%%%%%%%%%%%%%%%%%%%%%%%%%%%%
%%%%%%%%%%%%%%%%%%%%%%%%%%%%%%%%%%%%%%%%%%%%%%%%%%%%%%%%%%%%%%%%%%%%%%%%%%%%%%%%%%%%%%%%%%%%%%%%%%%
%%%%%%%%%%%%%%%%%%%%%%%%%%%%%%%%%%%%%%%%%%%%%%%%%%%%%%%%%%%%%%%%%%%%%%%%%%%%%%%%%%%%%%%%%%%%%%%%%%%

\subsection{Moment maps}
\label{sec:momentmaps}

%%%%%%%%%%%%%%%%%%%%%%%%%%%%%%%%%%%%%%%%%%%%%%%%%%%%%%%%%%%%%%%%%%%%%%%%%%%%%%%%%%%%%%%%%%%%%%%%%%%
%%%%%%%%%%%%%%%%%%%%%%%%%%%%%%%%%%%%%%%%%%%%%%%%%%%%%%%%%%%%%%%%%%%%%%%%%%%%%%%%%%%%%%%%%%%%%%%%%%%
%%%%%%%%%%%%%%%%%%%%%%%%%%%%%%%%%%%%%%%%%%%%%%%%%%%%%%%%%%%%%%%%%%%%%%%%%%%%%%%%%%%%%%%%%%%%%%%%%%%
%%%%%%%%%%%%%%%%%%%%%%%%%%%%%%%%%%%%%%%%%%%%%%%%%%%%%%%%%%%%%%%%%%%%%%%%%%%%%%%%%%%%%%%%%%%%%%%%%%%

Example \ref{ep:Raction} showed that the group of diffeomorphisms $\left\{\rho_{t}:\, \mathcal{M}\to\mathcal{M}\, , \, \, t\in\mathbb{R}\right\}$ generated by a vector field $v\in\mathfrak{X}\left(\mathcal{M}\right)$ can be understood as an action of $\mathbb{R}$ on $\mathcal{M}$ given by

\begin{eqnarray}
\psi :\, \mathbb{R} & \to &\mathrm{Diff}\left(\mathcal{M}\right)\nonumber\\
              t     & \to & \rho_{t}\, .
\end{eqnarray}

\noindent
The action $\psi: \mathbb{R} \to \mathrm{Diff}\left(\mathcal{M}\right)$ is said to be symplectic if it acts on a Symplectic manifold $\left(\mathcal{M},\omega\right)$ preserving the symplectic form $\omega$

\begin{equation}
\rho^{\ast}_{t}\omega = \omega\, .
\end{equation}

\noindent
This is equivalent to

\begin{equation}
\mathcal{L}_{v}\omega = d\left(\iota_{v}\omega\right) = 0\, ,
\end{equation}

\noindent
where $v\in\mathfrak{X}\left(\mathcal{M}\right)$ generates the one-parameter group of diffeomorphisms $\left\{\rho_{t}:\, \mathcal{M}\to\mathcal{M}\, , \, \, t\in\mathbb{R}\right\}$. Therefore, if the action is symplectic, then $\left(\iota_{v}\omega\right)$ is a closed one-form. If, in addition, if $\left(\iota_{v}\omega\right)$ is exact, then the action is said to be Hamiltonian. In that case, there exist a function $f\in C^{\infty}\left(\mathcal{M}\right)$ such that

\begin{equation}
\mathcal{L}_{v}\omega = -df\, .
\end{equation}

\noindent
$f$ is of course the Hamiltonian function associated to $v$, and therefore we can write $v = v_{f}$. The moment map construction generalizes the concept of the Hamiltonian action of $\mathbb{R}$ on a symplectic manifold to a general Lie group $G$ acting on a symplectic manifold. 

\definition{\label{def:momentmaps} Let $\left(\mathcal{M},\omega\right)$ be a symplectic manifold and $G$ a connected Lie group with Lie algebra $\mathfrak{g}$. Let us denote by $\mathfrak{g}^{\ast}$ the dual vector space of $\mathfrak{g}$ and assume that there exists a symplectic action of $G$ on $\mathcal{M}$

\begin{equation}
\psi :\, G\to\mathrm{Symp}\left(\mathcal{M},\omega\right)\, .
\end{equation}

\noindent
Then the action is said to be Hamiltonian if there exists a map

\begin{equation}
\mu :\, \mathcal{M}\to \mathfrak{g}^{\ast}\, ,
\end{equation}

\noindent
satisfying

\begin{enumerate}

\item For each $v\in\mathfrak{g}$, let

\begin{itemize}

\item $\mu^{v} :\, \mathcal{M}\to \mathbb{R}\, ,$ given by $\mu^{v}(p)\equiv <\mu (p), v >$, where $<\cdot, \cdot >$ is the natural pairing of $\mathfrak{g}$ and $\mathfrak{g}^{\ast}$.

\item $\tilde{v}$ be the vector field generated by the one-parameter subgroup $\left\{e^{t v}\, \,  |\, \, t\in\mathbb{R}\right\}\subseteq G$.

\end{itemize}

Then 

\begin{equation}
d\mu^{v} = -\iota_{\tilde{v}}\omega\, ,
\end{equation}

\noindent
that is, $\mu^{v}$ is a Hamiltonian function for the vector field $\tilde{v}$.

\item $\mu$ is equivariant with respect to the given action $\psi$ and the coadjoint action $\mathrm{Ad}^{\ast}$ of $G$ on $\mathfrak{g}^{\ast}$

\begin{equation}
\mu\circ\psi_{g} = \mathrm{Ad}^{\ast}_{g}\circ\mu\, ,\qquad g\in G\, .
\end{equation}

\end{enumerate}
 
}

\noindent

$\left(\mathcal{M},\omega , G, \mu\right)$ is then called a Hamiltonian $G$-space with momenp map $\mu$. For connected Lie groups, Hamiltonian actions can be equivalently defined in terms of the so-called comoment map

\begin{equation}
\mu^{\ast} :\, \mathfrak{g}\to C^{\infty}\left(\mathcal{M}\right)\, ,
\end{equation}

\begin{enumerate}

\item $\mu^{\ast}\left( v\right) \equiv \mu^{v}$ is a Hamiltonian function for the vector field $\tilde{v}$.

\item $\mu^{\ast}$ is a Lie algebra homomorphism

\begin{equation}
\mu^{\ast} \left[ v, w\right] = \left\{\mu^{\ast}\left( v\right), \mu^{\ast}\left( w\right)\right\}\, ,
\end{equation}

\noindent
where $\left\{\cdot, \cdot\right\}$ is the Poisson bracket on $C^{\infty}\left(\mathcal{M}\right)$.

\end{enumerate}

\noindent

\begin{ep}
\label{ep:RS1hamiltonian} {\bf Case $G = S^{1}$.} Here $\mathfrak{g} = \mathbb{R}$ and $\mathfrak{g}^{\ast} = \mathbb{R}$. A moment map $\mu :\, \mathcal{M}\to\mathbb{R}$ satisfies

\begin{enumerate}

\item For the generator $v=1$ of $\mathfrak{g}$ we have $\mu^{v} (p) = \mu (p)\cdot 1$, that is, $\mu^{v} = \mu$ and $\tilde{v}$ is the standard vector field $\mathcal{M}$ generated by $S^{1}$. Then $d\mu =- \iota_{\tilde{v}}\omega$.

\item $\mu$ is invariant $\mathcal{L}_{v}\mu = 0$.

\end{enumerate}

\end{ep}

\cleardoublepage

%%%%%%%%%%%%%%%%%%%%%%%%%%%%%%%%%%%%%%%%%%%%%%%%%%%%%%%%%%%%%%%%%%%%%%%
%%% CHAPTER 3: L infinity algebras
%%%%%%%%%%%%%%%%%%%%%%%%%%%%%%%%%%%%%%%%%%%%%%%%%%%%%%%%%%%%%%%%%%%%%%%
%\renewcommand{\chaptername}{Part I\\ Chapter}

%\renewcommand{\leftmark}{$L_{\infty}$-algebras}
\chapter{$L_{\infty}$-algebras}
\label{chapter:linfty}

In this section we review $L_{\infty}$-algebras and explicitly describe general $L_{\infty}$-morphisms. An $L_{\infty}$-algebra structure on graded vector space $L$ can be defined to be a collection of skew-symmetric maps $\{l_{k} \maps L^{\tensor k} \to L \}_{k=1}^{\infty}$ with $\deg l_{k} =k-2$ which satisfy a rather complicated generalization of the Jacobi identity. We will therefore start with a more elegant description, given in terms of coalgebras, and prove its equivalence to the previous characterization.

%%%%%%%%%%%%%%%%%%%%%%%%%%%%%%%%%%%%%%%%%%%%%%%%%%%%%%%%%%%%%%%%%%%%%%%%%%%%%%%%%%%%%%%%%%%%%%%%%%%
%%%%%%%%%%%%%%%%%%%%%%%%%%%%%%%%%%%%%%%%%%%%%%%%%%%%%%%%%%%%%%%%%%%%%%%%%%%%%%%%%%%%%%%%%%%%%%%%%%%
%%%%%%%%%%%%%%%%%%%%%%%%%%%%%%%%%%%%%%%%%%%%%%%%%%%%%%%%%%%%%%%%%%%%%%%%%%%%%%%%%%%%%%%%%%%%%%%%%%%
%%%%%%%%%%%%%%%%%%%%%%%%%%%%%%%%%%%%%%%%%%%%%%%%%%%%%%%%%%%%%%%%%%%%%%%%%%%%%%%%%%%%%%%%%%%%%%%%%%%

\section{Basic definitions}

%%%%%%%%%%%%%%%%%%%%%%%%%%%%%%%%%%%%%%%%%%%%%%%%%%%%%%%%%%%%%%%%%%%%%%%%%%%%%%%%%%%%%%%%%%%%%%%%%%%
%%%%%%%%%%%%%%%%%%%%%%%%%%%%%%%%%%%%%%%%%%%%%%%%%%%%%%%%%%%%%%%%%%%%%%%%%%%%%%%%%%%%%%%%%%%%%%%%%%%
%%%%%%%%%%%%%%%%%%%%%%%%%%%%%%%%%%%%%%%%%%%%%%%%%%%%%%%%%%%%%%%%%%%%%%%%%%%%%%%%%%%%%%%%%%%%%%%%%%%
%%%%%%%%%%%%%%%%%%%%%%%%%%%%%%%%%%%%%%%%%%%%%%%%%%%%%%%%%%%%%%%%%%%%%%%%%%%%%%%%%%%%%%%%%%%%%%%%%%%

\begin{definition}
An $L_{\infty}[1]$-structure on a graded vector space $M$ is a choice of degree one codifferential $Q$ on the coalgebra 

\begin{equation}
C(M) = \bar{S}\left( M\right)\, .
\end{equation}
 
\end{definition}

\begin{thm} 
\label{thm:coalgebraLinfty}
An $L_{\infty}[1]$-structure on a graded vector space $M$, that is, a choice of degree one codiferential $Q$ on $\bar{S}(M)$, uniquely determines a family of degree one linear maps
\begin{equation}
\left(m_{k}: \bar{S}^{k}(M)\to M\right)_{k\in\mathbb{N}^{+}}\, ,
\end{equation}

\noindent
such that

\begin{equation}
\label{eq:L1condition}
\sum_{r+s=k}\sum_{\sigma\in \mathrm{Sh}(r,s)} \epsilon(\sigma) m_{(s+1)}\left(m_{r}\left(x_{\sigma(1)}\otimes\dots\otimes x_{\sigma(r)}\right)\otimes x_{\sigma(r+1)}\otimes\dots\otimes x_{\sigma(k)}\right) = 0\, .
\end{equation}

\noindent
where $\epsilon(\sigma) = \epsilon(\sigma;x_{1},\dots,x_{r})$ and $x_{1}, \dots, x_{k}\in M$. Conversely, any such family $(m_{k})_{k\in\mathbb{N}^{+}}$ of degree one linear maps uniquely determines a degree one codifferential $Q$ on $\bar{S}(M)$.
\end{thm}

\proof{Given a codifferential $Q$ on $\bar{S}(M)$, consider the restrictions

\begin{equation}
\label{eq:restrict}
Q_{k}=Q \vert_ {\S^{k}(M)} \maps \S^{k}(M) \to \S(M), \quad 1 \leq k < \infty\, ,
\end{equation}

\noindent
so that $Q = \sum_{k}^{\infty}Q_{k}$, and also the projections

\begin{equation} 
\label{eq:Qproj}
Q^{k}_{m} = \mathrm{pr}_{\bar{S}^{k}(M)} \circ Q_{m} \maps \bar{S}^{m}(M) \to
\bar{S}^{k}(M)\, .
\end{equation}

\noindent
It follows from proposition \ref{prop:codifferentialrestriction} that $Q$ can be uniquely determined by the collection of maps

\begin{equation} 
\label{eq:struct_maps}
Q^{1}_{k} = \mathrm{pr}_{M} \circ Q_{k} \maps \bar{S}^{k}(M) \to M, \quad
k \geq 1\, .
\end{equation}

\noindent
The complete coderivation $Q$ can be written as

\begin{eqnarray}
\label{eq:coder_eq}
Q_{m}\left( x_{\mathrm{1}} \odot \cdots \odot x_{m}\right) = 
Q^{1}_{m}\left( x_{1} \odot \cdots \odot  x_{m}\right) +\\
 \sum^{m-1}_{i=\mathrm{1}} \sum_{\sigma \in \Sh(i,m-i)}
\epsilon\left(\sigma\right) Q^{1}_{i}\left( x_{\sigma(\mathrm{1})} \odot \cdots \odot
 x_{\sigma\left( i\right)} \right)\odot  x_{\sigma\left( i+1\right)} \odot \cdots \odot  x_{\sigma\left( m\right)}\, ,
\end{eqnarray} 

\noindent
for any $x_{i} \in M$. Defining now the maps $\left(m_{k}\right)_{k\in\mathbb{N}^{+}}$ as follows

\begin{equation}
\left(m_{k} = Q^{1}_{k}: \bar{S}^{k}(M)\to M\right)_{k\in\mathbb{N}^{+}}\, ,
\end{equation}

\noindent
the condition $Q \circ Q =0$ is equivalent to the generalized Jacobi identity \eqref{eq:L1condition} for the collection $\left(m_{k}\right)_{k\in\mathbb{N}^{+}}$. In particular, it implies that $l_{1}$ is degree one differential on $L$. On the other hand, let us assume that $\bar{S}(M)$ is equipped with set of degree one maps 

\begin{equation}
\left(m_{k}: \bar{S}^{k}(M)\to M\right)_{k\in\mathbb{N}^{+}}\, ,
\end{equation}

\noindent
obeying equation (\ref{eq:L1condition}). Taking the $\left(m_{k}\right)_{k\in\mathbb{N}^{+}}$ as the $\left( Q^{1}_{k}\right)_{k\in\mathbb{N}^{+}}$ components in Lemma 2.4 in \cite{Lada:1994mn} we conclude that there exists a unique codifferential $Q$ in $\bar{S}(M)$ such that its $\left( Q^{1}_{k}\right)_{k\in\mathbb{N}^{+}}$ restrictions are given by the $\left(m_{k}\right)_{k\in\mathbb{N}^{+}}$ maps.

\qed}

\noindent
Hence, an $L_{\infty}[1]$-structure on a graded vector space $M$ can be equivalently defined in terms of a degree one coderivation $Q$ on $C(M)$ or in terms of a family of morphisms $\left(m_{k}: S^{k}(M)\to M\right)_{k\in\mathbb{N}^{+}}$ obeying (\ref{eq:L1condition}). An $L_{\infty}$-structure is related to an $L_{\infty}[1]$-structure by a degree shift in $M$. In particular, an $L_{\infty}$-structure on a graded vector space $L$ is nothing but an $L_{\infty}[1]$-structure on the graded vector space $M=s^{-1}L$.

\definition{\label{def:Linfinity} An $L_{\infty}$-structure on a graded vector space $L$ is an $L_{\infty}[1]$-structure on the graded vector space $s^{-1}L$.}

\noindent
An equivalent, more practical, definition is the following

\begin{definition}[\cite{Lada:1994mn}] 
\label{def:Linfinity2} 

An $L_{\infty}$-algebra is a graded vector space $L$ together with a collection 

\begin{equation}
\left\{l_{k} \maps L^{\tensor k} \to L| 1 \leq k < \infty \right\}
\end{equation}

\noindent
of graded skew-symmetric linear maps with  $|l_{k}|=2-k$ such that the following identity is satisfied for $1 \leq m < \infty$

\begin{eqnarray} 
\label{eq:gen_jacobi}
   \sum_{\substack{i+j = m+\mathrm{1}\, ,\\ \sigma \in \mathrm{Sh}\left(i,m-i\right)}}
  \left(-1\right)^{\sigma}\epsilon(\sigma)\left(-1\right)^{i\left( j-1\right)} l_{j}
   \left( l_{i}(x_{\sigma(\mathrm{1})}, \dots, x_{\sigma\left( i\right)}), x_{\sigma(i+1)}\, ,
   \ldots, x_{\sigma(m)}\right)=0\, .
\end{eqnarray}

\end{definition}

\proof{Let us prove the equivalence of both definitions. If we consider an $L_{\infty}[1]$ structure on $s^{-1}L$ then condition (\ref{eq:L1condition}) can be written as

\begin{equation}
\label{eq:L1condition2}
\sum_{r+s=k}\sum_{\sigma\in \mathrm{Sh}(r,s)} \epsilon(\sigma) m_{(s+1)}\left(m_{r}\left(s^{-1}x_{\sigma(1)}\otimes\dots\otimes s^{-1}x_{\sigma(r)}\right)\otimes s^{-1}x_{\sigma(r+1)}\otimes\dots\otimes s^{-1}x_{\sigma(k)}\right) = 0\, .
\end{equation}

\noindent
where $\epsilon(\sigma) = \epsilon\left(\sigma; s^{-1}x_{1},\dots, s^{-1}x_{k}\right)$ and $x_{1}, \dots, x_{k}\in L$. Using equation (\ref{eq:sk}) it can be proven that

\begin{equation}
\label{eq:mdec}
m_{r}\left(s^{-1}x_{\sigma(1)}\otimes\dots\otimes s^{-1}x_{\sigma(r)}\right) = (-1)^{\sum_{i=1}^{r}(r-i) |x_{\sigma (i)}|} m_{r}\circ s^{-r}\left(x_{\sigma(1)}\otimes\dots\otimes x_{\sigma(r)}\right)\, .
\end{equation}

\noindent
Using now equation (\ref{eq:mdec}) together with the following commutative diagram

\begin{center}
\begin{tikzpicture}
\label{diag:coalgebradiagramII}
  \matrix (m) [matrix of math nodes,row sep=8em,column sep=9em,minimum width=2em]
  {
\Lambda^{k}\left( L\right) & L \\
S^{k}\left(s^{-1} L\right) & s^{-1} L \\};
  \path[-stealth]
    (m-1-1) edge node [left] {$s^{-k}$} (m-2-1)
    (m-1-1) edge node [above] {$l_{k}$} (m-1-2)
    (m-2-1) edge node [above] {$m_{k}$} (m-2-2)
    (m-1-2) edge node [right] {$s^{-1}$} (m-2-2);
\end{tikzpicture}
\end{center}

\noindent
we obtain

\begin{equation}
m_{r}\circ s^{-r}\left(x_{\sigma(1)}\otimes\dots\otimes x_{\sigma(r)}\right) =  s^{-1}\circ l_{r}\left(x_{\sigma(1)}\otimes\dots\otimes x_{\sigma(r)}\right)\, .
\end{equation}

\noindent
Therefore

\begin{eqnarray}
\label{eq:L1conditionml}
m_{(s+1)}\left(m_{r}\left(s^{-1}x_{\sigma(1)}\otimes\dots\otimes s^{-1}x_{\sigma(r)}\right)\otimes s^{-1}x_{\sigma(r+1)}\otimes\dots\otimes s^{-1}x_{\sigma(k)}\right) =\nonumber\\(-1)^{\sum_{i=1}^{r}(r-i) |x_{\sigma (i)}|+\sum_{i=1}^{k-r+1}(k-r+1-i) |\tilde{x}_{i}|} s^{-1}\circ l_{(s+1)}\left(l_{r}\left(x_{\sigma(1)}\otimes\dots\otimes x_{\sigma(r)}\right)\otimes x_{\sigma(r+1)}\otimes\dots\otimes x_{\sigma(k)}\right) \, ,
\end{eqnarray}

\noindent
where

\begin{equation}
\tilde{x}_{1} = l_{r}\left(x_{\sigma(1)},\hdots , x_{\sigma (r)}\right)\, ,\quad \tilde{x}_{i} = x_{\sigma (r-1+i)}\, ,\quad 2\leq i \leq k-r+1\, .
\end{equation}

\noindent
Since

\begin{equation}
|l_{r}\left(x_{\sigma(1)},\hdots , x_{\sigma (r)}\right)| = 2-r+\sum_{i=1}^{r} |x_{\sigma(i)}|\, , 
\end{equation}

\noindent
we have

\begin{equation}
\label{eq:signs}
\sum_{i=1}^{k-r+1}|\tilde{x}_{i}| = 2-r+\sum^{k}_{i=1}|x_{\sigma (i)}|\, ,\quad (-1)^{\sum_{i=1}^{r}(r-i) |x_{\sigma (i)}|+\sum_{i=1}^{k-r+1}(k-r+1-i) |\tilde{x}_{i}|} =(-1)^{(k-r)r+\sum_{i=1}^{k}(k-i) |x_{\sigma(i)}|}\, .
\end{equation}

\noindent
Finally, using equation (\ref{eq:epsilondecalage}) together with equation (\ref{eq:signs}) in equation (\ref{eq:L1condition}), we obtain equation (\ref{eq:gen_jacobi}).

\qed}

\noindent
Therefore, any $L_{\infty}$-algebra $(L,l_{k})$ corresponds to a certain kind of graded co-algebra $C(s^{-1}L)$ equipped with a co-derivation $Q$ which satisfies the identity

\begin{equation}
Q \circ Q =0\, .
\end{equation} 

\noindent
As we have seen, this identity is the origin of equation \eqref{eq:gen_jacobi}. It is easy to see that for small values of $m$ that equation \eqref{eq:gen_jacobi} is a \emph{generalized Jacobi identity} for the multi-brackets $\{l_{k}\}$. For $k=1$, it implies that the degree one linear map
$l_{1}$ satisfies

\begin{equation}
l_{1} \circ l_{1}=0
\end{equation}

\noindent
and hence every $L_{\infty}$-algebra $(L,l_{k})$ has an underlying cochain complex $(L,d=l_{1})$. For $k=2$ we have that $[\cdot,\cdot] = l_{2}$ is a degree zero linear map that satisfies

\begin{equation}
d \left[ x_{1}, x_{2}\right] = \left[ dx_{1}, x_{2}\right] + (-1)^{|x_{1}|}\left[x_{1}, dx_{2}\right]\, .
\end{equation}

\noindent
Hence $l_{2}$ can be interpreted as a bracket, which is skew symmetric

\begin{equation}
\left[ x_{1}, x_{2}\right] = -(-1)^{|x_{1}||x_{2}|}\left[ x_{2}, x_{1}\right]\, ,
\end{equation}

\noindent
but does not satisfy the usual Jacobi identity.

\begin{definition} \label{def:Lninfinity} A Lie $n$-algebra is a $L_{\infty}$-algebra $\left( L,\left\{l_{k} \right\}\right )$ such that the corresponding graded vector space $L$ is concentrated in degrees $0, -\mathrm{1}, \dots, \mathrm{1}-n$.
\end{definition}

\noindent
Notice that if $\left( L,\left\{l_{k} \right\}\right )$ is a Lie $n$-algebra, simply by degree counting then $l_{k} =0 $ for $k > n+1$. Therefore, a Lie $1$-algebra is nothing but a Lie algebra.

%%%%%%%%%%%%%%%%%%%%%%%%%%%%%%%%%%%%%%%%%%%%%%%%%%%%%%%%%%%%%%%%%%%%%%%%%%%%%%%%%%%%%%%%%%%%%%%%%%%
%%%%%%%%%%%%%%%%%%%%%%%%%%%%%%%%%%%%%%%%%%%%%%%%%%%%%%%%%%%%%%%%%%%%%%%%%%%%%%%%%%%%%%%%%%%%%%%%%%%

\section{$L_{\infty}$-morphisms} 
\label{sec:linfinitymorphisms}

%%%%%%%%%%%%%%%%%%%%%%%%%%%%%%%%%%%%%%%%%%%%%%%%%%%%%%%%%%%%%%%%%%%%%%%%%%%%%%%%%%%%%%%%%%%%%%%%%%%
%%%%%%%%%%%%%%%%%%%%%%%%%%%%%%%%%%%%%%%%%%%%%%%%%%%%%%%%%%%%%%%%%%%%%%%%%%%%%%%%%%%%%%%%%%%%%%%%%%%

The notion of $L_{\infty}$-morphism will be essential in this work. We now give \cite{Lada:1994mn} a naive definition of what could be an $L_{\infty}-morphism$.

\begin{definition}
\label{def:strict_morph_def_1}
If $(L^{1},l^{1}_{k})$ and $(L^{2},l^{2}_{k})$ be $L_{\infty}$-algebras then a degree zero linear map $f \maps L^{1} \to L^{2}$ is a \emph{strict $L_{\infty}$-morphism} if and only if the following holds

\begin{equation} 
\label{eq:strict_def_eq}
l^{2}_{k} \circ f^{\tensor k} = f \circ l^{1}_{k} \quad \forall k \geq 1\, .
\end{equation}
\end{definition}

\noindent
The definition above however does not reflect the higher structure which resides within the theory in a natural way. Actually, there is a better definition, see Remark 5.3 of \cite{Lada:1994mn}, which uses the previously mentioned relationship between $L_{\infty}$-algebras and differential graded coalgebras. This turns out to give to the collection of morphisms between two $L_{\infty}$-algebras the
structure of a  simplicial set, see reference \cite{1998math.....12034H}, which therefore permits to consider homotopies among morphisms, homotopies among homotopies \emph{et cetera}. As we did when we defined $L_{\infty}$-algebras, we define first a morphism of $L_{\infty}[1]$-algebras. The corresponding definition for $L_{\infty}$-algebras can then be obtained by a degree shift in $L$.

\begin{definition} 
\label{def:L1infty-morph_basic_def}

An $L_{\infty}[1]$-morphism between $L_{\infty}[1]$-algebras $(M^{1},m^{1}_{k})$ and $\left(M^{2},m^{2}_{k}\right)$ is a morphism $F[1] \maps \left( C(M^{1}),Q^{1} \right) \to \left( C(M^{2}),Q^{2} \right)$ between the corresponding underlying differential graded coalgebras. $F[1]$ is thus a morphism between the graded coalgebras $C(M^{1})$ and $C(M^{2})$ such that

\begin{equation} 
\label{eq:preserve_codiff[1]}
F[1] \circ Q^{1} = Q^{2} \circ F[1]\, .
\end{equation} 
\end{definition}

\noindent
As it turns out, an $L_{\infty}[1]$-morphism $F[1]$ between $(M^{1},m^{1}_{k})$ and $(M^{2},m^{2}_{k})$ corresponds to an infinite collection of symmetric, degree zero, `structure maps'

\begin{equation}
F[1] = \left(f_{k}[1] \maps S^{k}\left(M^{1}\right) \to M^2 \quad 1 \leq k < \infty\right)\, ,
\end{equation}

\noindent
and such that a given compatibility relation with the multi-brackets must be satisfied. More precisely, the following proposition holds.

\prop{\label{prop:morphism} Let $(M^{1},m^{1}_{k})$ and $(M^{2},m^{2}_{k})$ be $L_{\infty}[1]$-algebras. A morphism from $(M^{1},m^{1}_{k})$ to $(M^{2},m^{2}_{k})$ is a family of morphism

\begin{equation}
F[1] = \left(f_{k}[1] \maps S^{k}\left(M^{1}\right) \to M^2 \quad 1 \leq k < \infty\right)\, ,
\end{equation}

\noindent
such that

\begin{eqnarray}
\label{eq:morphism1condition}
\sum_{r+s=k}\sum_{\sigma\in\mathrm{Sh}(r,s)} \epsilon\left(\sigma ; x_{1}, \hdots , x_{k}\right) f_{s+1}[1]\left(m_{r}\left(x_{\sigma (1)},\hdots, x_{\sigma (r)}\right),x_{\sigma (r+1)},\hdots, x_{\sigma (k)}\right) = \nonumber\\ \sum_{l=1}^{k}\sum_{j_{1}+\cdots +j_{l} = k}\sum_{\tau\in \Sigma_{k}} \frac{\epsilon\left(\tau ; x_{1}, \hdots , x_{k}\right)}{l!j_{1}!\dots j_{l}!} n_{l}\left(f_{j_{1}}[1]\left(x_{\sigma (\tilde{k}_{1}+1)},\hdots, x_{\sigma (\tilde{k}_{1}+j_{1})}\right),\hdots, f_{j_{l}}[1]\left(x_{\sigma (\tilde{k}_{l}+1)},\hdots, x_{\sigma (\tilde{k}_{l}+j_{l})}\right)\right) \, ,
\end{eqnarray}

\noindent
where $\tilde{k}_{1} = 0$ and $\tilde{k}_{s} = \sum_{i=1}^{s-1} j_{i}\, ,\,\, 1<s\leq l$.

}

\proof{See chapter 2 of \cite{coisotropic}.}

\noindent
The corresponding notion of morphism of $L_{\infty}$-algebras goes as follows

\begin{definition} 
\label{def:Linfty-morph_basic_def}

An $L_{\infty}$-morphism $F$ between $L_{\infty}$-algebras $(L^{1},l^{1}_{k})$ and $\left(L^{2},l^{2}_{k}\right)$ is a morphism $F \maps \left( C(s^{-1}L^{1}),Q^{1} \right) \to \left( C(s^{-1}L^{2}),Q^{2} \right)$ between their corresponding differential graded coalgebras. $F$ is hence a morphism between the graded coalgebras $C(s^{-1}L^{1})$ and $C(s^{-1}L^{2})$ such that

\begin{equation} 
\label{eq:preserve_codiff}
F \circ Q^{1} = Q^{2} \circ F\, .
\end{equation} 
\end{definition}

\noindent
As in the $L_{\infty}[1]$ case, the notion of $L_{\infty}$-morphism corresponds to an infinite family of maps, which in this case are skew-symmetric structure maps

\begin{equation}
F = \left(f_{k} \maps \Lambda^{k}L^{1} \to L^2 \quad 1 \leq k < \infty\right)\, ,
\end{equation}

\noindent
where $|f_{k}|=1-k$. Here again the family of maps have to satisfy a somewhat complicated compatibility relation involving the multi-brackets. Such compatibility relation can be obtained from equation (\ref{eq:morphism1condition}) by means of the following commutative diagram

\begin{center}
\begin{tikzpicture}
\label{diag:coalgebradiagram}
  \matrix (m) [matrix of math nodes,row sep=8em,column sep=9em,minimum width=2em]
  {
\Lambda^{k}\left( L^{1}\right) & L^{2} \\
S^{k}\left(s^{-1} L^{1}\right) & s^{-1} L^{2} \\};
  \path[-stealth]
    (m-1-1) edge node [left] {$s^{-k}$} (m-2-1)
    (m-1-1) edge node [above] {$f_{k}$} (m-1-2)
    (m-2-1) edge node [above] {$f[1]_{k}$} (m-2-2)
    (m-1-2) edge node [right] {$s^{-1}$} (m-2-2);
\end{tikzpicture}
\end{center}

\noindent
by performing the corresponding degree shift in $L^{1}$ and $L^{2}$. %The precise result is the following
%
%\begin{prop}\label{prop:morphism2} Let $(L^{1},l^{1}_{k})$ and $(L^{2},l^{2}_{k})$ be $L_{\infty}$-algebras. A morphism from $(L^{1},l^{1}_{k})$ to $(L^{2},l^{2}_{k})$ is a family of skew-symmetric maps

%\begin{equation}
%F = \left(f_{k} \maps \left(L^{1}\right)^{\tensor k} \to L^2 \quad 1 \leq k < \infty\right)\, ,
%\end{equation}

%\noindent
%of degree $|f_{k}|=1-k$ such that

%\begin{eqnarray}
%\label{eq:morphism1conditionL}
%\sum_{r+s=k}\sum_{\sigma\in\mathrm{Sh}(r,s)} \epsilon\left(\sigma ; x_{1}, \hdots , x_{k}\right) f_{s+1}[1]\left(m_{r}\left(x_{\sigma (1)},\hdots, x_{\sigma (r)}\right),x_{\sigma (r+1)},\hdots, x_{\sigma (k)}\right) = \nonumber\\ \sum_{l=1}^{k}\sum_{j_{1}+\cdots +j_{l} = k}\sum_{\tau\in \Sigma_{k}} \frac{\epsilon\left(\tau ; x_{1}, \hdots , x_{k}\right)}{l!j_{1}!\dots j_{l}!} n_{l}\left(f_{j_{1}}[1]\left(x_{\sigma (\tilde{k}_{1}+1)},\hdots, x_{\sigma (\tilde{k}_{1}+j_{1})}\right),\hdots, f_{j_{l}}[1]\left(x_{\sigma (\tilde{k}_{l}+1)},\hdots, x_{\sigma (\tilde{k}_{l}+j_{l})}\right)\right) \, ,
%\end{eqnarray}

%\end{prop}

\noindent
We see that, particular, the degree zero map $f_{1}$ is a morphism between the corresponding complexes $(L^{1},l^{1}_{1})$ and $(L^{2},l^{2}_{1})$

\begin{equation}
f_{1}\circ l^{1}_{1} = l^{2}_{1}\circ f_{1}\, .
\end{equation}

\noindent
As it happens for $L_{\infty}[1]$-morphisms, the compatibility relation between the family of maps $\left( f_{i}\right)_{i\geq 1}$ and the multibrackets precisely corresponds in the language of coalgebras  to equation (\ref{eq:preserve_codiff}). It can be easily seen that strict morphisms as defined in \ref{def:strict_morph_def_1} correspond to the case given by $f_{i}=0$ for $i \geq 2$. $L_{\infty}$-morphisms can be composed in the standard sense, and therefore it is possible to consider the category of $L_{\infty}$-algebras without explicit use the higher structure present in $L_{\infty}$-morphisms.

We define now the notion of $L_{\infty}$-\emph{quasi-isomorphism}. This definition naturally reflexts the homotopical structure that exists between morphisms.

\begin{definition} 
\label{def:Linfty_qiso_def}
Let $(f_{k}) \maps (L^{1},l^{1}_{k}) \to (L^{2},l^{2}_{k})$ be an $L_{\infty}$-algebra morphisim. Then we say that $\left( f_{k}\right)_{k\geq 1}$ is an $L_{\infty}$-quasi-isomorphism if and only if the corresponding morphism of complexes

\begin{equation}
f_{1} \maps (L^{1},l^{1}_{1}) \to (L^{2},l^{2}_{1}) 
\end{equation}

\noindent
induces an isomorphism on the cohomology of the underlying complexes

\begin{equation}
H^{\bullet} \left(f_{1}\right) : H^{\bullet}\left(L^{1}\right) \xrightarrow{\cong} H^{\bullet}\left(L^{2}\right)\, .
\end{equation}

\end{definition}

%%%%%%%%%%%%%%%%%%%%%%%%%%%%%%%%%%%%%%%%%%%%%%%%%%%%%%%%%%%%%%%%%%%%%%%%%%%%%%%%%%%%%%%%%%%%%%%%%%%
%%%%%%%%%%%%%%%%%%%%%%%%%%%%%%%%%%%%%%%%%%%%%%%%%%%%%%%%%%%%%%%%%%%%%%%%%%%%%%%%%%%%%%%%%%%%%%%%%%%

\subsection{Morphisms from Lie algebras to $L_{\infty}$-algebras}\label{secP}

%%%%%%%%%%%%%%%%%%%%%%%%%%%%%%%%%%%%%%%%%%%%%%%%%%%%%%%%%%%%%%%%%%%%%%%%%%%%%%%%%%%%%%%%%%%%%%%%%%%
%%%%%%%%%%%%%%%%%%%%%%%%%%%%%%%%%%%%%%%%%%%%%%%%%%%%%%%%%%%%%%%%%%%%%%%%%%%%%%%%%%%%%%%%%%%%%%%%%%%

Since it will be useful in chapter \ref{chapter:multisymplectic},  we will consider $L_{\infty}$-algebra morphisms whose sources are simply Lie algebras $(\mathfrak{g},[\cdot,\cdot])$. In that case the conditions that the components of the morphism must satisfy are extremely simplified and the resulting expression can be indeed used for practical purposes. We will assume also that the image $(L,l_{k})$ of the $L_{\infty}$-algebra morphism is a Lie-$n$ algebra such that 

\begin{equation}
\label{eq:property}
\forall\,\, i\geq \mathrm{2} \qquad l_{i}\left(x_{1},\hdots,x_{i}\right) = 0 \quad\mathrm{whenever}\quad \sum_{k=1}^{i}|x_{k}|<0\, .
\end{equation}

\noindent
This is indeed the relevant case for Lie-$n$ algebras arising from $n$-plectic manifolds, as we will see in chapter \ref{chapter:multisymplectic}. The relevant proposition is then the following

\begin{prop} 
\label{prop:Lie_alg_P_cor}
Let $\left(\mathfrak{g},[\cdot,\cdot]\right)$ be a Lie algebra and let $\left(L,l_{k}\right)$ is a Lie $n$-algebra that 
satisfies the property \eqref{eq:property}. Then a collection of $n$ anti-symmetric maps

\begin{equation}
f_{m} : \g^{\otimes m} \to L, \quad |f_{m}| = 1-m, \quad 1
\leq m \leq n\, ,
\end{equation}

\noindent
can be taken to be the components of an $L_{\infty}$-morphism $\left( f_{k}\right)_{k\geq 1} :\mathfrak{g}\to L$ if and only if $\forall x_{i} \in \g$

\begin{eqnarray}
\label{eq:acor_eq1}
\sum_{\mathrm{1} \leq i < j \leq m}
(-1)^{i+j+1}f_{m-1}\left(\left[ x_{i},x_{j}\right],x_{1},\ldots,\widehat{x_{i}},\ldots,\widehat{x_{j}},\ldots,x_{m}\right)\\
=l_{1} f_{m}\left(x_{1},\ldots,x_{m}\right) + l_{m}\left(f_{1}(x_{1}),\ldots,f_{1}(x_{m})\right)\, .
\end{eqnarray}

\noindent
for $2 \leq m \leq n$ and

\begin{eqnarray} 
\label {eq:acor_eq2}
\sum_{\mathrm{1} \leq i < j \leq n+\mathrm{1}}
(-1)^{i+j+1}f_{n}\left(\left[ x_{i},x_{j}\right],x_{\mathrm{1}},\ldots,\widehat{x_{i}},\ldots,\widehat{x_{j}},\ldots,x_{n+\mathrm{1}}\right)
=l_{n+1}\left(f_{1}(x_{1}),\ldots,f_{1}(x_{n+\mathrm{1}})\right)\, .
\end{eqnarray}

\end{prop}

\begin{proof}
See appendix A.5 of \cite{2013arXiv1304.2051F}
\end{proof}

\cleardoublepage

%%%%%%%%%%%%%%%%%%%%%%%%%%%%%%%%%%%%%%%%%%%%%%%%%%%%%%%%%%%%%%%%%%%%%%%
%%% CHAPTER 4: Multisymplectic geometry and L-infinity algebras
%%%%%%%%%%%%%%%%%%%%%%%%%%%%%%%%%%%%%%%%%%%%%%%%%%%%%%%%%%%%%%%%%%%%%%%
%\renewcommand{\chaptername}{Part I\\ Chapter}

%\renewcommand{\leftmark}{Multisymplectic Geometry and $L_{\infty}$ algebras}
\chapter{Multisymplectic Geometry}
\label{chapter:multisymplectic}

%%%%%%%%%%%%%%%%%%%%%%%%%%%%%%%%%%%%%%%%%%%%%%%%%%%%%%%%%%%%%%%%%%%%%%%%%%%%%%%%%%%%%%%%%%%%%%%%%%%
%%%%%%%%%%%%%%%%%%%%%%%%%%%%%%%%%%%%%%%%%%%%%%%%%%%%%%%%%%%%%%%%%%%%%%%%%%%%%%%%%%%%%%%%%%%%%%%%%%%
%%%%%%%%%%%%%%%%%%%%%%%%%%%%%%%%%%%%%%%%%%%%%%%%%%%%%%%%%%%%%%%%%%%%%%%%%%%%%%%%%%%%%%%%%%%%%%%%%%%
%%%%%%%%%%%%%%%%%%%%%%%%%%%%%%%%%%%%%%%%%%%%%%%%%%%%%%%%%%%%%%%%%%%%%%%%%%%%%%%%%%%%%%%%%%%%%%%%%%%

\section{$n$-plectic manifolds}

%%%%%%%%%%%%%%%%%%%%%%%%%%%%%%%%%%%%%%%%%%%%%%%%%%%%%%%%%%%%%%%%%%%%%%%%%%%%%%%%%%%%%%%%%%%%%%%%%%%
%%%%%%%%%%%%%%%%%%%%%%%%%%%%%%%%%%%%%%%%%%%%%%%%%%%%%%%%%%%%%%%%%%%%%%%%%%%%%%%%%%%%%%%%%%%%%%%%%%%
%%%%%%%%%%%%%%%%%%%%%%%%%%%%%%%%%%%%%%%%%%%%%%%%%%%%%%%%%%%%%%%%%%%%%%%%%%%%%%%%%%%%%%%%%%%%%%%%%%%
%%%%%%%%%%%%%%%%%%%%%%%%%%%%%%%%%%%%%%%%%%%%%%%%%%%%%%%%%%%%%%%%%%%%%%%%%%%%%%%%%%%%%%%%%%%%%%%%%%%

We will closely follow \cite{2013arXiv1304.2051F}. For more details about Multisymplectic Geometry the interested reader can consult \cite{1998math......5040E,Ibort,CILeon,JAZ:4974756,2011arXiv1106.4068R}.

\begin{definition}
\label{def:n-plectic_def}
A differentiable manifold $\mathcal{M}$ is said to be \emph{$n$-plectic} or multisymplectic if it is equipped with an $(n+1)$-form $\omega\in\Omega^{n+1}\left(\mathcal{M}\right)$ such that it is
both closed

\begin{equation}
d\omega=0\, ,
\end{equation}

\noindent
and non-degenerate

\begin{equation}
\forall\,\, p \in \mathcal{M} \qquad  \forall\,\, u \in T_{p}\mathcal{M},\,\,\qquad \iota_{u} \omega =0 \Rightarrow u =0\, .
\end{equation}

\noindent
If $\omega$ is an $n$-plectic form on $\mathcal{M}$, then we call the pair $(\mathcal{M},\omega)$ 
an $n$-plectic manifold. More generally, if $\omega$ is closed, but not necessarily
non-degenerate, then we call $(\mathcal{M},\omega)$ a pre-$n$-plectic manifold. We will only deal with $n$-plectic manifolds, although most of our results can be straightforwardly extended to the pre-$n$-plectic case.
\end{definition}

\noindent
Please notice that a 1-plectic manifold is simply a symplectic manifold.

\begin{definition}
\label{def:multidiff} Given two $n$-plectic manifolds $\left(\mathcal{M}_{1},\omega_{1}\right)$ and $\left(\mathcal{M}_{2},\omega_{2}\right)$, a diffeomorphism $F\maps \mathcal{M}_{1}\to\mathcal{M}_{2}$ is said to be a multisymplectic diffeomorphism if and only if $F^{\ast}\omega_{2} = \omega_{1}$.
\end{definition}

\begin{definition} 
\label{def:hamiltonian}
Given an $n$-plectic manifold $\left(\mathcal{M},\omega\right)$, an $(n-1)$-form $\beta\in\Omega^{n-1}\left(\mathcal{M}\right)$ is said to be \emph{Hamiltonian} if and only if there exists a vector field $u_{\beta} \in \mathfrak{X}\left(\mathcal{M}\right)$ such that

\begin{equation}
d\beta= -\iota_{u_{\beta}} \omega\, .
\end{equation}

\noindent
We say then that $u_{\beta}$ is a \emph{Hamiltonian vector field} corresponding to $\beta$. We respectively denote by $\Omega^{n-1}_{\mathrm{Ham}}\left(\mathcal{M}\right)$ and $\mathfrak{X}_{\mathrm{Ham}}\left(\mathcal{M}\right)$, the set of Hamiltonian $(n-1)$-forms and the set of Hamiltonian vector fields on an $n$-plectic manifold, which are vector spaces. Note that, due to the non-degeneracy of $\omega$, for every Hamiltonian form there is a unique Hamiltonian vector field associated.

\end{definition}

\begin{definition}
\label{def:loc_ham}
A vector field $u$ on a $n$-plectic manifold $\left(\mathcal{M},\omega\right)$ is a \emph{local Hamiltonian vector field} if and only if

\begin{equation}
\label{eq:localhamm}
\mathcal{L}_{u}\omega =0\, ,
\end{equation}

\noindent
We denote by $\mathfrak{X}_{\mathrm{Ham}}\left(\mathcal{M}\right)$ the vector space of local Hamiltonian vector fields.
\end{definition}

\noindent
Please notice that equation (\ref{eq:localhamm}) is equivalent to

\begin{equation}
\label{eq:localhamm2}
di_{v}\omega =0\, ,
\end{equation}

\noindent
and therefore, for Hamiltonian vector fields $i_{u}\omega$ is an exact $n$-form while for locally Hamiltonian vector fields $i_{u}\omega$ is a closed $n$-form, which can be always locally written in terms of an exact form and therefore the name \emph{local Hamiltonian vector field}. If $H^{1}_{\mathrm{dR}}\left(\mathcal{M}\right) = 0$ both definitions of course coincide.

\begin{definition}
\label{def:bracket_def}
Let $\left(\mathcal{M},\omega\right)$ be a $n$-plectic manifold. Given $\alpha,\beta\in
\Omega^{n-1}_{\mathrm{Ham}}\left(\mathcal{M}\right)$, we define then the \emph{bracket} $\{ \alpha , \beta\}$ to be the $(n-1)$-form given by

\begin{equation}
\{ \alpha , \beta\} = \iota_{u_{\beta}}\iota_{u_{\alpha}}\omega\, ,
\end{equation}

\noindent
where $u_{\alpha}$ and $u_{\beta}$ respectively stand for the Hamiltonian vector fields for $\alpha$ and $\beta$.
\end{definition}

\begin{prop}
\label{prop:brac_prop}
Let $\left(\mathcal{M},\omega\right)$ be an $n$-plectic manifold and let $u_{1},u_{2} \in \mathfrak{X}_{\mathrm{Ham}}\left(\mathcal{M}\right)$ be local Hamiltonian vector fields. Then $\left[ u_{1}, u_{2}\right]$ is a global Hamiltonian vector field with

\begin{equation}
d\iota_{u_{1} \wedge u_{2}} \omega = -\iota_{[u_{1},u_{2}]} \omega\, ,
\end{equation}

\noindent
and thus  $\mathfrak{X}_{\mathrm{Ham}}\left(\mathcal{M}\right)$ and $\Xham\left(\mathcal{M}\right)$ are Lie subalgebras of $\mathfrak{X}\left(\mathcal{M}\right)$.
\end{prop}

\begin{proof}
Let $u_{1},u_{2}$ be locally Hamiltonian vector fields. Then by equation (\ref{eq:commutator}),

\begin{equation}
\mathcal{L}_{u_{1}} \iota_{u_{2}} \omega = \iota_{\left[ u_{1},u_{2}\right]} \omega\, .
\end{equation}

\noindent
Using now equation (\ref{eq:Lie}),

\begin{equation}
\mathcal{L}_{u_{1}} \iota_{u_{2}} \omega = \iota_{u_{1}} d\iota_{u_{2}} \omega
+ d \iota_{u_{1}} \iota_{u_{2}} \omega\, .
\end{equation}

\noindent
However $\iota_{u_{1}} d\iota_{u_{2}} \omega=0$, since $d\iota_{u_{2}}=\mathcal{L}_{u_{2}} - \iota_{u_{2}}d$. \qed 
\end{proof}

Therefore, proposition\ref{prop:brac_prop} implies for instance that if $u_{\alpha}$ and $u_{\beta}$ are respectively Hamiltonian vector fields for $\alpha $ and $\beta$, then $\left[ u_\alpha, u_\beta\right]$ is a Hamiltonian vector field for $\left\{ \alpha,\beta\right\}$. Notice that the bracket defined in \ref{def:bracket_def} is skew-symmetric but it fails to satisfy the Jacoby identity. In particular we have\footnote{See proposition 3.5 in reference \cite{2011arXiv1106.4068R}.} 

\begin{equation}
\label{eq:failjacobi}
\left\{\alpha_{1},\left\{\alpha_{2},\alpha_{3}\right\}\right\}-\left\{\left\{\alpha_{1},\alpha_{2}\right\},\alpha_{3}\right\}-\left\{\alpha_{2},\left\{\alpha_{1},\alpha_{3}\right\}\right\} = -d\iota_{v_{\alpha_{1}}\wedge v_{\alpha_{2}}\wedge v_{\alpha_{3}}}\omega\, .
\end{equation}

\noindent
Therefore, the space $\ham{n-1}$ of Hamiltonian forms equipped with the bracket $\left\{\cdot,\cdot\right\}$ is not a Lie algebra unless $n=1$, which is the well-know symplectic case. Hence, we cannot straightforwardly extend the Poisson structure present in the set of functions on a symplectic manifold to the set of Hamiltonian forms on a multisymplectic manifold. However, equation (\ref{eq:failjacobi}) shows that $\left\{\cdot,\cdot\right\}$ fails to satisfy the Jacobi identity by an exact form, which suggests the existence of an underlying $n$-Lie algebra structure of which $\ham{n-1}$ would be part of. Roughly speaking, if we identify the interior product of $k$ Hamiltonian vector fields with $\omega$ as $l_{k}$ acting on the corresponding $k$ Hamiltonian forms, then equation \ref{eq:failjacobi} is the condition that $l_{2}$ and $l_{3}$ have to obey if they are part of an Lie-$n$ algebra.

We present now a theorem that gives a natural $L_{\infty}$-structure, in particular a $n$-Lie algebra structure, on differential forms, extending the bracket $\left\{\cdot,\cdot\right\}$ on $\Omega^{n-1}_{\mathrm{Ham}}\left(\mathcal{M}\right)$. See theorem 5.2 in \cite{2012LMaPh.100...29R} and theorem 6.7 in \cite{2010arXiv1003.1004Z}. A detailed exposition can be found in \cite{2011arXiv1106.4068R}.

\begin{thm} 
\label{thm:ham-infty}
Let $\left(\mathcal{M},\omega\right)$ be an $n$-plectic manifold . Then there exists a Lie $n$-algebra
$L_{\infty}(\mathcal{M},\omega)=(L,\left\{l_{k} \right\})$ with underlying graded vector space 

\begin{equation}
L^{i} =
\begin{cases}
\Omega^{n-1}_{\mathrm{Ham}}\left(\mathcal{M}\right) & k = 0,\\
\Omega^{n-1+k}\left(\mathcal{M}\right) & 1-n \leq k < 0\, ,
\end{cases}
\end{equation}

\noindent
and maps  $\left\{ l_{k} : L^{\otimes k} \to L| 1 \leq k < \infty \right\}$ defined as follows

\begin{equation}
l_{1}\left( \beta\right)=d\beta\, ,
\end{equation}

\noindent
if $|\beta | < 0$ and

\begin{equation}
l_{k}\left( \beta_{1},\hdots, \beta_{k}\right) =
\begin{cases}
\xi(k) \iota\left( u_{\beta_{1}}\wedge\cdots\wedge u_{\beta_{k}}\right) \omega  & \text{if  $|\beta_{1}\otimes\cdots\otimes\beta_{k}| = 0$},\\
0 & \text{if $|\beta_{1}\otimes\cdots\otimes\beta_{k}| < 0$}\, ,
\end{cases}
\end{equation}

\noindent
for $k>1$, where $u_{\beta_{i}}$ is the Hamiltonian vector field associated to $\beta_{i} \in \Omega^{n-1}_{\mathrm{Ham}}\left(\mathcal{M}\right)$ and $\xi(k)=-(-1)^{\frac{k(k+1)}{2}}$.
\end{thm}

\noindent
We can extend definition \ref{def:bracket_def} and define a $k$-ary bracket in $L\left(\mathcal{M,\omega}\right)$ as follows

\begin{equation}
\label{eq:karybracket}
\left\{x_{1},\hdots,x_{k}\right\} = l_{k}\left( x_{1},\hdots,x_{k}\right)\, ,\qquad x_{1},\hdots,x_{k}\in L\left(\mathcal{M,\omega}\right)\, .
\end{equation}

\noindent
Please notice that in the $n=1$ case, the underlying complex is simply the vector space of Hamiltonian functions $C^{\infty}\left(\mathcal{M}\right)$. The only non-zero bracket is therefore $l_{2}=\left\{\cdot,\cdot\right\}$, which is simply a Lie bracket. We thus recover the Lie algebra which underlies the usual Poisson algebra that can be constructed for symplectic manifold. As explained in section \ref{sec:symplectic}, in that case there is a well-defined surjective Lie algebra morphism

\begin{equation}
\pi : C^{\infty}\left(\mathcal{M}\right) \epi \mathfrak{X}_{\mathrm{Ham}}\left(\mathcal{M}\right)\, ,
\end{equation}

\noindent
which send a function to its (unique) Hamiltonian vector field. If $\mathcal{M}$ is connected, it can be shown that $\pi$ fits in the following short exact sequence

\begin{equation} 
\label{eq:KS_extension}
0 \to \mathbb{R} \to C^{\infty}\left(\mathcal{M}\right) \xrightarrow{\pi} \mathfrak{X}_{\mathrm{Ham}}\left(\mathcal{M}\right) \to 0\, .
\end{equation}

\noindent
Equation (\ref{eq:KS_extension}) is the so-called Kostant-Souriau central extension, see references \cite{Kostant:1970,Souriau:1967}. The Kostant-Souriau central extension characterizes,up to isomorphism, the Lie algebra of $C^{\infty}\left(\mathcal{M}\right)$ as follows: it is the central extension, which can be shown to be unique, given by the symplectic form, evaluated at $p \in \mathcal{M}$. The higher analog of the central extension (\ref{eq:KS_extension}) is given by the cochain map

\begin{equation}
\label{eq:pi_map}
\pi : L\left(\mathcal{M},\omega\right) \epi \mathfrak{X}_{\mathrm{Ham}}\left(\mathcal{M}\right)\, , 
\end{equation}

\noindent
which is trivial en all degrees but zero. In degree zero it assigns to every Hamiltonian form $\alpha$ its unique Hamiltonian vector field. The map $\pi$ fits hence in the following short exact sequence

\begin{equation} 
\label{eq:KS_extensionmulti}
0 \to \tilde{\Omega} \to L\left(\mathcal{M}\right) \xto{\pi} \Xham\left(\mathcal{M}\right) \to 0\, .
\end{equation}

\noindent
Here $\tilde{\Omega}$ stands for the cocomplex 

\begin{equation}
\tilde{\Omega} = C^{\infty}\left(\mathcal{M}\right)\to \Omega^{1}\left(\mathcal{M}\right)\to\cdots\to\Omega^{n-1}_{\mathrm{cl}}\left(\mathcal{M}\right)\, ,
\end{equation}

\noindent
where $\Omega^{n-1}_{\mathrm{cl}}\left(\mathcal{M}\right)$ is the set of closed $(n-1)$-forms in $\mathcal{M}$ and the coboundary operator is the de Rahm differential.

%%%%%%%%%%%%%%%%%%%%%%%%%%%%%%%%%%%%%%%%%%%%%%%%%%%%%%%%%%%%%%%%%%%%%%%%%%%%%%%%%%%%%%%%%%%%%%%%%%%
%%%%%%%%%%%%%%%%%%%%%%%%%%%%%%%%%%%%%%%%%%%%%%%%%%%%%%%%%%%%%%%%%%%%%%%%%%%%%%%%%%%%%%%%%%%%%%%%%%%
%%%%%%%%%%%%%%%%%%%%%%%%%%%%%%%%%%%%%%%%%%%%%%%%%%%%%%%%%%%%%%%%%%%%%%%%%%%%%%%%%%%%%%%%%%%%%%%%%%%
%%%%%%%%%%%%%%%%%%%%%%%%%%%%%%%%%%%%%%%%%%%%%%%%%%%%%%%%%%%%%%%%%%%%%%%%%%%%%%%%%%%%%%%%%%%%%%%%%%%

\subsection{Homotopy moment maps}
\label{sec:homotopymoment}

%%%%%%%%%%%%%%%%%%%%%%%%%%%%%%%%%%%%%%%%%%%%%%%%%%%%%%%%%%%%%%%%%%%%%%%%%%%%%%%%%%%%%%%%%%%%%%%%%%%
%%%%%%%%%%%%%%%%%%%%%%%%%%%%%%%%%%%%%%%%%%%%%%%%%%%%%%%%%%%%%%%%%%%%%%%%%%%%%%%%%%%%%%%%%%%%%%%%%%%
%%%%%%%%%%%%%%%%%%%%%%%%%%%%%%%%%%%%%%%%%%%%%%%%%%%%%%%%%%%%%%%%%%%%%%%%%%%%%%%%%%%%%%%%%%%%%%%%%%%
%%%%%%%%%%%%%%%%%%%%%%%%%%%%%%%%%%%%%%%%%%%%%%%%%%%%%%%%%%%%%%%%%%%%%%%%%%%%%%%%%%%%%%%%%%%%%%%%%%%

As usual $G$ denotes a Lie group and $\mathfrak{g}$ its Lie algebra. We will consider only actions $\psi$ from the left. That is, $G$ acts on $\Omega^{\bullet}\left(\mathcal{M}\right)$ from the left through the inverse pullback

\begin{equation}
g \cdot \omega \mapsto \psi^{\ast}_{g^{-1}} \omega\, ,
\end{equation}

\noindent
where $\psi_{g}$ is the diffeomorphism that corresponds to $g$. The corresponding infinitesimal action of the Lie algebra $\mathfrak{g}$ is denoted by the map

\begin{equation} 
\label{eq:lie_alg_action}
u_{-} : \mathfrak{g} \to \mathfrak{X}\left(\mathcal{M}\right), \quad y \mapsto u_{y}\, ,
\end{equation}

\noindent
where

\begin{equation}
u_{y}|_{p} = \frac{d}{dt} \exp(-ty) \cdot p|_{t=0} \qquad
\forall p \in \mathcal{M}\, .
\end{equation}

\noindent
We name $u_{-}$ as the \emph{fundamental vector field} which is associated to the $G$ action.

In the context of symplectic geometry, we can express a moment map $\mathcal{M} \to \mathfrak{g}^{*}$ also as
a \emph{comoment map}, namely a Lie algebra morphism $\mathfrak{g} \to C^{\infty}\left(\mathcal{M}\right)$; see section \ref{sec:symplectic} for more details. We will introduce in this section, closely following the seminal paper \cite{2013arXiv1304.2051F}, the natural analog in multisymplectic geometry of the comoment map in symplectic geometry. It is the so-called \emph{homotopy moment map}.

Let us just point out that a moment map $\mu$ is an equivariant  $\mathfrak{g}^{*}$-valued smooth function on a symplectic manifold $\left(\mathcal{M},\omega\right)$ which equipped with an action, namely $\mu\in \left(\mathfrak{g}^{*}\otimes \mathcal{C}^{\infty} \left(\mathcal{M}\right)\right)^{G}$. $\mu$ has to satisfy a particular condition respecto to $\omega$. This notion can be also expressed as a comoment map, which is nothing but a Lie algebra map $\mu : \mathfrak{g}\to(\mathcal{C}^{\infty}\left(\mathcal{M}\right),\{,\})$ from the Lie algebra of $G$, to the Poisson algebra that can be associated to the symplectic form.

\begin{definition} 
\label{def:main_def}
Let $G$ be a Lie group with Lie algebra $\mathfrak{g}$. Let $\left(\mathcal{M},\omega\right)$ be an $n$-plectic manifold which is equipped with a $G$-action preserving $\omega$, and such that the  $\mathfrak{g}$-action $x \mapsto u_{x}$ is through Hamiltonian vector fields. A \emph{homotopy moment map} is the following lift 

\begin{center}
\begin{tikzpicture}
\label{diag:homotopymoment}
  \matrix (m) [matrix of math nodes,row sep=8em,column sep=9em,minimum width=2em]
  {
 & L\left(\mathcal{M},\omega\right) \\
\mathfrak{g} & \mathfrak{X}_{\mathrm{Ham}}\left(\mathcal{M}\right) \\};
  \path[-stealth]
    (m-1-2) edge node [right] {$\pi$} (m-2-2)
    (m-2-1) edge node [above] {$v_{-}$} (m-2-2)
    (m-2-1) edge node [right] {} (m-1-2);
\end{tikzpicture}
\end{center}

\noindent
of the Lie algebra morphism $u_{-}$ (\ref{eq:lie_alg_action}) by means of the
$L_{\infty}$-morphism $\pi$ \ref{eq:pi_map} in the category of $L_{\infty}$-algebras. This lift corresponds to an $L_{\infty}$-morphism 

\begin{equation}
\left( f_{k}\right)_{k\geq 1} :  \mathfrak{g} \to L_{\infty}\left(\mathcal{M},\omega\right)\, ,
\end{equation}

\noindent
that satisfies

\begin{equation}
\label{eq:condicionham}
-\iota_{u_{y}} \omega=d\left(f_{1}(y)\right)\qquad \text{for all $y\in \mathfrak{g}$}\, .
\end{equation}
\end{definition}

\noindent
Please notice that the condition $-\iota_{u_{y}} \omega=d(f_{1}(y))$ implies that $u_{y}$ is the unique Hamiltonian vector field for $f_{1}(y) \in \Omega^{n-1}_{\mathrm{Ham}}\left(\mathcal{M}\right)$. In addition, using proposition \ref{prop:Lie_alg_P_cor}, we can rewrite the conditions on the components $\left(f_{k}\right)_{k\geq 1} : \g^{\otimes k} \to L\left(\mathcal{M},\omega\right)$ of the $L_{\infty}$-morphism as follows

\begin{eqnarray} 
\label{eq:main_eq_1}
\sum_{1 \leq i < j \leq k}
(-1)^{i+j+1}f_{k-1}\left(\left[y_{i},y_{j}\right],y_{1},\ldots,\widehat{y_{i}},\ldots,\widehat{y_{j}},\ldots,y_{k}\right)\\
=df_{k}\left((y_{\mathrm{1}},\hdots,y_{k}\right) + \xi(k)\iota(u_{1}\wedge \cdots \wedge u_{k})\omega\, ,
\end{eqnarray}

\noindent
for $2 \leq k  \leq n$ plus

\begin{eqnarray} 
\label{eq:main_eq_2}
\sum_{1 \leq i < j \leq n+1}
(-1)^{i+j+\mathrm{1}}f_{n}\left(\left[ y_{i},y_{j}\right],y_{1},\ldots,\widehat{y_{i}},\hdots,\widehat{y_{j}},\hdots,y_{n+1}\right)
=\xi(n+\mathrm{1})\iota\left(u_{1}\wedge \cdots \wedge u_{{n+1}}\right)\omega\, .
\end{eqnarray}

\noindent
Here $u_{i}$ stands for the vector field associated to $y_{i}$ via the $\mathfrak{g}$-action. Please notice that the theorem \ref{thm:ham-infty} implies in particular that $L_{\infty}\left(\mathcal{M},\omega\right)$ satisfies \ref{eq:property}.

Please notice also that proposition \ref{prop:brac_prop} implies in particular that $u_{\left[ x , y \right]}= \left[ u_{x},u_{y} \right]$ is a Hamiltonian vector field for 

\begin{equation}
\left\{ f_{1}(x), f_{1}(y)\right\}=l_{2}\left( f_{\mathrm{1}}(x), f_{1}(y)\right)\, .
\end{equation}

\noindent
In general, the map $f_{\mathrm{1}} : \mathfrak{g} \to \Omega^{n-1}_{\mathrm{Ham}}\left(\mathcal{M}\right)$ will not preserve the bracket on $\mathfrak{g}$, \emph{i.e.}, we will have

\begin{equation}
f_{\mathrm{1}}\left(\left[ x, y \right]\right) \neq \left\{ f_{1}(x) , f_{1}(y)\right\}\, .
\end{equation}

\noindent
This is a nice property that should be expected, since the Lie bracket of $\mathfrak{g}$ satisfies the Jacobi identity but $\left\{\cdot\cdot\right\}$ does not.

\begin{definition}
\label{def:hamiltonianaction} Let $\left(\mathcal{M},\omega\right)$ be an $n$-plectic manifold. The action of a Lie group $G$ on $\left(\mathcal{M},\omega\right)$ is said to be Hamiltonian if an homotopy moment map for such action exists.
\end{definition}

%%%%%%%%%%%%%%%%%%%%%%%%%%%%%%%%%%%%%%%%%%%%%%%%%%%%%%%%%%%%%%%%%%%%%%%%%%%%%%%%%%%%%%%%%%%%%%%%%%%
%%%%%%%%%%%%%%%%%%%%%%%%%%%%%%%%%%%%%%%%%%%%%%%%%%%%%%%%%%%%%%%%%%%%%%%%%%%%%%%%%%%%%%%%%%%%%%%%%%%
%%%%%%%%%%%%%%%%%%%%%%%%%%%%%%%%%%%%%%%%%%%%%%%%%%%%%%%%%%%%%%%%%%%%%%%%%%%%%%%%%%%%%%%%%%%%%%%%%%%
%%%%%%%%%%%%%%%%%%%%%%%%%%%%%%%%%%%%%%%%%%%%%%%%%%%%%%%%%%%%%%%%%%%%%%%%%%%%%%%%%%%%%%%%%%%%%%%%%%%

\section{Multisymplectic diffeomorphisms and $n$-algebra morphisms}
\label{sec:multidiff}

%%%%%%%%%%%%%%%%%%%%%%%%%%%%%%%%%%%%%%%%%%%%%%%%%%%%%%%%%%%%%%%%%%%%%%%%%%%%%%%%%%%%%%%%%%%%%%%%%%%
%%%%%%%%%%%%%%%%%%%%%%%%%%%%%%%%%%%%%%%%%%%%%%%%%%%%%%%%%%%%%%%%%%%%%%%%%%%%%%%%%%%%%%%%%%%%%%%%%%%
%%%%%%%%%%%%%%%%%%%%%%%%%%%%%%%%%%%%%%%%%%%%%%%%%%%%%%%%%%%%%%%%%%%%%%%%%%%%%%%%%%%%%%%%%%%%%%%%%%%
%%%%%%%%%%%%%%%%%%%%%%%%%%%%%%%%%%%%%%%%%%%%%%%%%%%%%%%%%%%%%%%%%%%%%%%%%%%%%%%%%%%%%%%%%%%%%%%%%%%

In this section we are going to study the relation between strict morphisms of $L_{\infty}$-algebras and multisymplectic diffeomorphisms of the corresponding multisimplectic manifolds. That is, we want to know under which conditions, if any, we are able to conclude that a strict morphism of Lie $n$-algebras $\left\{ L\left(M_{a},\omega_{a}\right),l^{a}_{k}\right\}\, , a=1,2\, ,~k=1,\cdots,n+1\, ,$ \emph{induces} a multisymplectic diffeomorphism between the corresponding $n$-plectic manifolds $\left(M_{a},\omega_{a}\right)$. We know that in the symplectic case the answer is positive: two symplectic manifolds with corresponding isomorphic Poisson algebras are symplectomorphic. We will see that in the $n$-plectic case the answer is also positive, at least for a special class of $n$-plectic manifolds, those which are \emph{locally homogeneous} with respect to the multisymplectic form. Hence, at least in those cases,  the $L_{\infty}$-algebra constructed on a multisymplectic manifold is powerful enough to contain important information about the manifold itself and its differential structure.

Let $\phi : L\left(M_{2},\omega_{2}\right) \to L\left(M_{1},\omega_{1}\right)$ be an strict Lie $n$-algebra morphism and let us write $\phi = \left(\phi_{n-1},\dots,\phi_{0}\right)$, where $\phi_{i} :  L_{i}\left(M_{2},\omega_{2}\right) \to L_{i}\left(M_{1},\omega_{1}\right)\, , ~ i = 1-n,\dots,0$\footnote{For the precise definition of strict morphism of $L_{\infty}$-algebras see \ref{sec:linfinitymorphisms}.}.

In order to relate strict Lie $n$-algebra morphisms and multisymplectic diffeomorphisms, we need first the existence of $\phi$ to imply the existence of a diffeomorphism $F:M_{1}\to M_{2}$, which then must be checked to be a multisymplectic diffeomorphism. This is easily achieved by making use of the following lemma

\lemma{\label{lem:diff} Let $\mathcal{M}_{a}\, , a=1,2\, ,$ be differentiable manifolds and $\psi : \left( C^{\infty}\left(\mathcal{M}_{2}\right), \cdot\right)\to \left( C^{\infty}\left(\mathcal{M}_{1}\right), \cdot\right)$ an algebra morphism, where $\cdot$ denotes the usual multiplication of functions. Then $\psi = F^{\ast}$, where $F : \mathcal{M}_{1}\to\mathcal{M}_{2}$ is a smooth map. In addition, if $\psi$ is an algebra isomorphism then $F$ is a diffeomorphism.}

\proof{See theorem 4.2.36 in reference \cite{tensoranalisis}.}

\noindent
\newline Hence, assuming that $\phi_{1-n}: C^{\infty}\left(\mathcal{M}_{2}\right)\to  C^{\infty}\left(\mathcal{M}_{1}\right)$ is an algebra isomorphism from $\left\{ C^{\infty}\left(\mathcal{M}_{2}\right), \cdot\right\}$ to $\left\{ C^{\infty}\left(\mathcal{M}_{1}\right), \cdot\right\}$, we can conclude the existence of a diffeomorphism $F$ from $\mathcal{M}_{1}$ such that $\mathcal{M}_{2}$ and 

\begin{equation}
\label{eq:phiF}
\phi_{n-1} = F^{\ast}\, .
\end{equation}

\noindent
Since $\phi$ is an strict Lie $n$-algebra morphism, it preserves the maps $l^{a}_{k}\, , ~k = 1,\dots,n+1$ of the corresponding Lie $n$-algebra\footnote{Notice that $l^{a}_{1}=d_{a}$.}

\begin{equation}
d_{2}\circ\phi = \phi\circ d_{1}
\end{equation}

\begin{equation}
\phi_{2-k}\circ l^{2}_k\left(\alpha_{1},\cdots,\alpha_{k}\right) = l^{1}_k\left(\phi_{0}\circ \alpha_{1},\dots,\phi_{0}\circ \alpha_{k}\right)\, , \qquad \forall~~\alpha_{1}\, ,\dots,\alpha_{k}\in \Omega^{n-1}_{\mathrm{Ham}}\left(\mathcal{M}_{2}\right)\, ,\qquad k = 2,\dots , n+1 \, .
\end{equation}

\noindent
Using now the definition (\ref{eq:karybracket}) of the $L_{\infty}$-algebra maps $l_{k}$, as well as (\ref{eq:phiF}), we obtain

\begin{equation}
F^{\ast} \left\{ \alpha^{2}_{1},\cdots, \alpha^{2}_{n+1}\right\}_{2} = \left\{ \phi_{0}\circ \alpha^{2}_{1},\cdots, \phi_{0}\circ \alpha^{2}_{n+1}\right\}_{1}\, , \qquad \forall~~\alpha^{2}_{1}\, ,\dots,\alpha^{2}_{n+1}\in \Omega^{n-1}_{\mathrm{Ham}}\left(\mathcal{M}_{2}\right)\, .
\end{equation} 

\noindent
Further assuming that $\phi_{i}=F^{\ast}\, , ~ i = 1-n,\dots,0\, ,$ it can be proven that the initial set-up consisting of two $n$-plectic manifolds and a strict Lie-$n$ algebra morphism $\phi$ is equivalent, in a precise sense to be specified in a moment, to considering a unique manifold $\mathcal{M}$ equipped with two $n$-plectic structures $\omega_{1}$ and $\omega_{2}$ such as $l^{1}_{k}=l^{2}_{k}\, , ~ k = 1,\cdots,n+1$. Notice that the condition $\phi_{0}=F^{\ast}$ is non-trivial, since for arbitrary diffeomorphisms we would have

\begin{equation}
F^{\ast}: \Omega^{n-1}_{\mathrm{Ham}}\left(\mathcal{M}_{2}\right)\to \Omega^{n-1}\left(\mathcal{M}_{1}\right)\, ,
\end{equation}

\noindent
and we are requiring

\begin{equation}
F^{\ast}: \Omega^{n-1}_{\mathrm{Ham}}\left(\mathcal{M}_{2}\right)\to \Omega^{n-1}_{\mathrm{Ham}}\left(\mathcal{M}_{1}\right)\, .
\end{equation}

\noindent
We will assume then that the Lie-$n$ algebra morphism is given by $\phi = F^{\ast}$, where $F:\mathcal{M}_{1}\to\mathcal{M}_{2}$ is a diffeomorphism, and conclude then that $F$ must be a multisymplectomorphism by studying the equivalent case of a unique manifold $\mathcal{M}_{1}$ equipped with two multisimplectic structures $\left(\omega_{1}, \omega_{2}\right)$, such that the corresponding Lie-$n$ algebras are equal. Let us first prove the equivalence of both cases.

\prop{\label{prop:equivalence} Let $\left(\mathcal{M}_{a},\omega_{a}\right)\, , ~a=1,2\, ,$ be $n$-plectic manifolds, $\left\{ L\left(\mathcal{M}_{a},\omega_{a}\right),l^{a}_{k}\right\}$ denote the corresponding Lie $n$-algebras and $\phi: \left\{ L\left(\mathcal{M}_{2},\omega_{2}\right),l^{2}_{k}\right\}\to \left\{ L\left(\mathcal{M}_{1},\omega_{1}\right),l^{1}_{k}\right\}$ a strict Lie $n$-algebra morphism such that

\begin{equation}
\phi_{i} = F^{\ast} :  L_{i}\left(M_{2},\omega_{2}\right) \to L_{i}\left(M_{1},\omega_{1}\right)\, , ~ i = 1-n,\dots,0\, ,
\end{equation}

\noindent
where $F:\mathcal{M}_{1}\to\mathcal{M}_{2}$ is a diffeomorphism. Then, $F$ is a multisymplectic diffeomorphism if and only if $ \left\{ L\left(\mathcal{M}_{1},\omega_{1}\right),l^{1}_{k}\right\}= \left\{ L\left(\mathcal{M}_{1},\tilde{\omega}_{1} \equiv F^{\ast}\omega_{2}\right),\tilde{l}^{1}_{k}\right\}$ implies $\omega_{1}=\tilde{\omega}_{1}.$
}

\proof{If $F:\mathcal{M}_{1}\to\mathcal{M}_{2}$ is a multisymplectomorphism, then $F^{\ast}\omega_{2} = \omega_{1}$ and therefore $ \left\{ L\left(\mathcal{M}_{1},\omega_{1}\right),l^{1}_{k}\right\}= \left\{ L\left(\mathcal{M}_{1},\tilde{\omega}_{1}\right),\tilde{l}^{1}_{k}\right\}$ since $\tilde{\omega}_{1}=\omega_{1}$, which in turn implies $\tilde{l}^{1}_{k}=l^{1}_{k}\, ,\,\,k =1,\cdots,n+1$.

On the other hand, if $ \left\{ L\left(\mathcal{M}_{1},\omega_{1}\right),l^{1}_{k}\right\}= \left\{ L\left(\mathcal{M}_{1},\tilde{\omega}_{1} \equiv F^{\ast}\omega_{2}\right),\tilde{l}^{1}_{k}\right\}$ implies $\omega_{1}=\tilde{\omega}_{1}.$, then $\omega_{1} = F^{\ast}\omega_{2}$ and therefore $F$ is a multisymplectomorphism.\qed}

\noindent
In other words, proposition \ref{prop:equivalence} simply states that the following diagram of strict-isomorphisms of $L_{\infty}$-algebras commutes

\begin{center}
\begin{tikzpicture}
\label{diag:commutativemorphism}
  \matrix (m) [matrix of math nodes,row sep=8em,column sep=9em,minimum width=2em]
  {& L\left(\mathcal{M}_{1},F^{\ast}\omega_{2}\right) \\
L\left(\mathcal{M}_{1},\omega_{1}\right) & L\left(\mathcal{M}_{2},\omega_{2}\right) \\};
  \path[-stealth]
    (m-1-2) edge node [above] {Id} (m-2-1)
    (m-2-2) edge node [above] {$\phi$} (m-2-1)
    (m-2-2) edge node [right] {$F^{\ast}$} (m-1-2);
\end{tikzpicture}
\end{center}

\noindent
Therefore, we will consider the equivalent situation of a unique manifold $\mathcal{M}$ equipped with two multisymplectic structures $\omega_{1}$ and $\omega_{2}$. Before proving the main result of this section, namely theorem (\ref{thm:multisymplecticdiff}), it is necessary to introduce the concept of \emph{locally homogeneous manifold} and the lemma (\ref{lem:span}) \cite{1998math......5040E}.

\definition{\label{def:lochomogeneous} Let $\mathcal{M}$ be a differentiable manifold. Consider $p\in\mathcal{M}$ and a compact set $K\in\mathcal{M}$ such that $p\in K^{\circ}$\footnote{$K^{\circ}$ denotes the interior of $K$.}. A local Liouville or local Euler-like vector field at $p$, with respect to $K$, is a vector field $\Delta^{p}$ on $\mathcal{M}$ such that $\mathrm{supp}\Delta^{p} \equiv \overline{\left\{ q\in\mathcal{M} | \Delta^{p}(q)\neq 0\right\}} \subset K$, and there exists a diffeomorphism $\phi: \left(\mathrm{supp}\Delta^{p}\right)^{\circ} \to \mathbb{R}^{n}$ such that $\phi_{\ast}\Delta^{p} = \Delta$, where $\Delta = x^{i}\frac{\partial}{\partial x^{i}}$ is the standard Liouville or dilation vector field in $\mathbb{R}^{n}$.}

\definition{\label{def:homform} A differential form $\omega\in\Gamma\left(\Lambda^{(n+1)}\left(T^{\ast}\mathcal{M}\right)\right)$ is said to be \emph{locally homogeneous} at $p\in\mathcal{M}$ if, for every open set $U$ containing $p$, there exists a local Euler-like vector field $\Delta^{p}$ at $p$ with respect to a compact set $K\subset U$ such that

\begin{equation}
\mathcal{L}_{\Delta^{p}} \omega = f\omega\, ,\qquad f\in C^{\infty}\left(\mathcal{M}\right)\, .
\end{equation}

\noindent
The form is said to be locally homogeneous if it is locally homogeneous for all $p\in\mathcal{M}$.
}

\noindent
Obviously, out of $\mathrm{supp}\Delta^{p}$ , the function $f$ vanishes. A couple $\left(\mathcal{M},\omega\right)$ where $\omega$ is locally homogeneous is called a locally homogeneous manifold. As examples of homogeneous $n$-plectic manifolds we can find symplectic, manifolds and multicotangent bundles and in fact any oriented manifold equipped with its volume form. 

\prop{\label{prop:symplectichomogeneous} Every 1-plectic manifold $\left(\mathcal{M},\omega\right)$ is locally homogeneous with $f=2$. }

\proof{}

The following lemmas\footnote{Lemma (\ref{lem:span}) is actually is a small generalization of lemma 4.5 in \cite{1998math......5040E}.}  will play an important role in the proof of theorem (\ref{thm:multisymplecticdiff})

\lemma{\label{lem:span} Let $\left(\mathcal{M},\omega\right)$ be a locally homogeneous $n$-plectic manifold. Then, the family of hamiltonian vector fields span the tangent bundle of $\mathcal{M}$. That is

\begin{equation}
T_{p}\mathcal{M} = \mathrm{span}\left\{v_{p}\, |\, v\in \Gamma\left(T\mathcal{M}\right)\, ,\,\, i_{v}\omega = d \alpha_{v}\, ,\, \alpha_{v}\in \Omega^{(n-1)}_{\mathrm{Ham}}(\mathcal{M})\right\}\, .
\end{equation}
}

\proof{Let $\left(\mathcal{M},\omega\right)$ be a locally homogeneous $n$-plectic manifold, and let $v_{p}\in T_{p}\mathcal{M}$ be any vector at $p\in\mathcal{M}$. Let $U$ be a contractible open neighborhood of $p$, which can be shrink in order to be contained in a coordinate chart $\mathcal{U}_{\alpha}$, with coordinates $\phi_{\alpha}$. In lemma 4.5 of reference \cite{1998math......5040E}, it was proven the existence of a vector field $v_{U}$ on $U$, such that       $\mathrm{supp}\, v_{U}\subset U$ is compact, $v_{U}|_{p} = v_{p}$ and

\begin{equation}
d\iota_{v_{U}}\omega = 0 \, ,
\end{equation}

\noindent
that is, $\iota_{v_{U}}\omega $ is closed. $v_{U}$ can be extended trivially to all $\mathcal{M}$ by defining a vector field $v\in\Gamma\left(T\mathcal{M}\right)$ as follows

\begin{eqnarray}
v|_{p} &=& v_{U}|_{p}\, , \, \, p\in \mathrm{supp}\, v_{U}\, ; \qquad v|_{p} = 0\, , \, \, p\notin \mathrm{supp}\, v_{U}\, ,
\end{eqnarray}

\noindent
We will prove now that $v$ is a Hamiltonian vector field on $\mathcal{M}$. $\iota_{v_{U}}\omega$ is a closed $n$-form with compact support in $U$, that is $\left[\iota_{v_{U}}\right]\omega\in H^{n}_{C}\left(U\right)$. In addition, we can choose $U\subset \mathcal{U}_{\alpha}$ such that 

\begin{equation}
U\simeq \mathbb{R}^{d}\, , \qquad d = \mathrm{dim}\, \mathcal{M}\, . 
\end{equation}

\noindent
Therefore $H^{n}_{C}\left(U\right)\simeq H^{n}_{C}\left(\mathbb{R}^{d}\right)$ where $H^{n}_{C}\left(\mathbb{R}^{d}\right)$ denotes the $n$-th compactly supported cohomology group of $\mathbb{R}^{d}$. Noticing that $n<d$ we conclude that

\begin{equation}
H^{n}_{C}\left(U\right)\simeq H^{n}_{C}\left(\mathbb{R}^{d}\right)\simeq \left\{ 0\right\}\, . 
\end{equation}

\noindent
Hence, there exist a $(n-1)$-form $\alpha_{U}\in\Omega^{n-1}\left(U\right)$ with compact support contained in $U$ such that 

\begin{equation}
\iota_{v_{U}}\omega = d\alpha_{U}\, .
\end{equation}

\noindent
Now, extending $\alpha_{U}$ trivially to a $(n-1)$-form $\alpha\in \Omega^{n-1}\left(\mathcal{M}\right)$ as follows

\begin{equation}
\alpha |_{p} = \alpha_{U}|_{p}\, , \, \, p\in \mathrm{supp}\, \alpha_{U}\, ; \qquad \alpha |_{p} = 0\, , \, \, p\notin \mathrm{supp}\, \alpha_{U}\, ,
\end{equation}

\noindent
we see that

\begin{equation}
\iota_{v}\omega = d\alpha\, .
\end{equation}

\noindent
Therefore, $v$ is a Hamiltonian vector field in $\mathcal{M}$ that can be build as to give any vector $v|_{p} = v_{p}$ at $p\in\mathcal{M}$. We conclude then that $T_{p}\mathcal{M}$ is generated by Hamiltonian vector fields on $\mathcal{M}$ evaluated at $p\in\mathcal{M}$. \qed
 
}

\noindent

\lemma{\label{lem:equalhamv} Let $\mathcal{M}$ be a multisymplectic manifold equipped with two $n$-plectic structures $\omega_{1}$ and $\omega_{2}$ such that the corresponding Lie-$n$ algebras are equal, namely $L\left(\mathcal{M},\omega_{2}\right) = L\left(\mathcal{M},\omega_{1}\right)$. Let us assume that at least one of the $n$-plectic structures, say $\omega_{1}$, is locally homogeneous. Then

\begin{equation}
v^{1}_{\alpha} = v^{2}_{\alpha}\, ,\qquad\forall\, \, \alpha\in\Omega^{(n-1)}_{\mathrm{Ham}}\left(\mathcal{M}\right)\, ,
\end{equation}

\noindent
where $v^{a}_{\alpha}\, , \, a=1,2$ is the Hamiltonian vector field of $\alpha$ respect to $\omega_{a}$, that is

\begin{equation}
\label{eq:hamvectorfielddef}
d\alpha =\iota_{v^{a}_{\alpha}}\omega_{a}\, ,\qquad a= 1, 2\, .
\end{equation}

}

\proof{Let $\alpha\in\Omega^{(n-1)}_{\mathrm{Ham}}\left(\mathcal{M}\right)$ and let $v^{a}_{\alpha}$ the Hamiltonian vector of $\alpha$ respect to $\omega_{a}$. Since by assumption $L\left(\mathcal{M},\omega_{2}\right) = L\left(\mathcal{M},\omega_{1}\right)$, we can write

\begin{equation}
l^{1}_{2} \left(\alpha, \beta\right) = l^{2}_{2} \left(\alpha, \beta\right)\, , \qquad\forall\,\, \alpha, \beta \in \Omega^{(n-1)}_{\mathrm{Ham}}\left(\mathcal{M}\right)\, ,
\end{equation}

\noindent
which, in turn, implies

\begin{equation}
\label{eq:hamvectorigual}
\omega_{1} \left(v^{1}_{\alpha}, v^{1}_{\beta},\dots\right) = \omega_{2} \left(v^{2}_{\alpha}, v^{2}_{\beta},\dots\right)\, , \qquad\forall\,\, \alpha, \beta \in \Omega^{(n-1)}_{\mathrm{Ham}}\left(\mathcal{M}\right)\, ,
\end{equation}

\noindent
where $v^{a}_{\alpha}$ is the Hamiltonian vector field of $\alpha$ associated to $\omega_{a}$. Using now that

\begin{equation}
d\alpha = -\iota_{v^{1}_{\alpha}}\omega_{1} = -\iota_{v^{2}_{\alpha}}\omega_{2}\, ,\qquad \forall \alpha\in\Omega^{(n-1)}_{\mathrm{Ham}}\left(\mathcal{M}\right)\, ,
\end{equation}

\noindent
we can rewrite equation (\ref{eq:hamvectorigual}) as follows

\begin{equation}
\label{eq:hamvectorigualII}
\omega_{1} \left(v^{1}_{\alpha}, v^{1}_{\beta},\dots\right) = \omega_{1} \left(v^{1}_{\alpha}, v^{2}_{\beta},\dots\right)\, , \qquad\forall\,\, \alpha, \beta \in \Omega^{(n-1)}_{\mathrm{Ham}}\left(\mathcal{M}\right)\, ,
\end{equation}

\noindent
The non-degeneracy of $\omega_{1}$ together with lemma \ref{lem:span} finally implies

\begin{equation}
v^{1}_{\beta} = v^{2}_{\beta}\, , \qquad\forall\,\, \beta \in \Omega^{(n-1)}_{\mathrm{Ham}}\left(\mathcal{M}\right)\, .
\end{equation}

\qed}

\noindent
\thm{\label{thm:multisymplecticdiff} Let $\mathcal{M}$ be a differentiable manifold equiped with two multisimplectic structures $\omega_{1}$ and $\omega_{2}$, such that at least one of them is locally homogeneous. Then $L\left(\mathcal{M},\omega_{2}\right) = L\left(\mathcal{M},\omega_{1}\right)$ if and only if $\omega_{1}=\omega_{2}$.}  

\proof{If $\omega_{1}=\omega_{2}$ it is obvious that $L\left(\mathcal{M},\omega_{2}\right) = L\left(\mathcal{M},\omega_{1}\right)$. On the other hand, let us assume that $L\left(\mathcal{M},\omega_{2}\right) = L\left(\mathcal{M},\omega_{1}\right)$. In particular, the underlying complex and the multilinear brackets constructed from $\omega_{1}$ and $\omega_{2}$ must be equal. We can write then

\begin{equation}
l^{1}_{(n+1)}\left(\alpha_{1},\dots ,\alpha_{(n+1)}\right) = l^{2}_{(n+1)}\left(\alpha_{1},\dots ,\alpha_{(n+1)}\right)\, ,\qquad\forall\,\ \alpha_{1},\dots ,\alpha_{(n+1)}\in\Omega^{(n-1)}_{\mathrm{Ham}}\left(\mathcal{M}\right)\, ,
\end{equation}

\noindent
and therefore

\begin{equation}
\label{eq:equalforms}
i_{v_{\alpha_1}\wedge\cdots\wedge v_{\alpha_{(n+1)}}}\omega_{1} = i_{v_{\alpha_1}\wedge\cdots\wedge v_{\alpha_{(n+1)}}}\omega_{2}\, ,\qquad\forall\,\, v_{\alpha_1}, \dots ,v_{\alpha_{(n+1)}}\in\mathfrak{X}_{\mathrm{Ham}}\left(\mathcal{M}\right)\, ,
\end{equation}

\noindent
where we have used lemma \ref{lem:equalhamv} in order to use the same Hamiltonian vector fields for $\omega_{1}$ and $\omega_{2}$. Evaluating now \ref{eq:equalforms} point-wise we obtain, using lemma \ref{lem:span}

\begin{equation}
i_{v_{1}\wedge\cdots\wedge v_{(n+1)}}\omega_{1}|_{p} = i_{v_{1}\wedge\cdots\wedge v_{(n+1)}}\omega_{2}|_{p}\, ,\qquad\forall\,\, v_{1}|_{p}, \dots ,v_{(n+1)}|_{p}\in T_{p}\left(\mathcal{M}\right)\, .
\end{equation}

\noindent
We conclude hence that $\omega_{1}|_{p}=\omega_{2}|_{p}$ for all $p\in\mathcal{M}$ and therefore $\omega_{1}=\omega_{2}$. \qed
 }

\noindent
\newline From theorem \ref{thm:multisymplecticdiff} and propostion \ref{prop:equivalence} we immediately conclude the final result of this section, namely

\thm{\label{thm:finalmultisymplecticdiff} Let $\left(\mathcal{M}_{a},\omega_{a}\right)\, , ~a=1,2\, ,$ be locally homogeneous multisymplectic manifolds, let $\left\{ L\left(M_{a},\omega_{a}\right),l^{a}_{k}\right\}$ denote the corresponding $L_{\infty}$ algebras and let $\phi: \left\{ L\left(\mathcal{M}_{2},\omega_{2}\right),l^{2}_{k}\right\}\to \left\{ L\left(\mathcal{M}_{1},\omega_{1}\right),l^{1}_{k}\right\}$ an strict $L_{\infty}$-isomorphism such that

\begin{equation}
\phi_{i} = F^{\ast} :  L_{i}\left(\mathcal{M}_{2},\omega_{2}\right) \to L_{i}\left(\mathcal{M}_{1},\omega_{1}\right)\, , ~ i = 1-n,\dots,0\, ,
\end{equation}

\noindent
where $F:\mathcal{M}_{1}\to\mathcal{M}_{2}$ is a diffeomorphism. Then, $F$ is also a multisymplectic diffeomorphism, that is, $F^{\ast}\omega_2 = \omega_1$.}  

\proof{Direct consequence of proposition (\ref{prop:equivalence}) and theorem (\ref{thm:multisymplecticdiff}).\qed}

\noindent

%%%%%%%%%%%%%%%%%%%%%%%%%%%%%%%%%%%%%%%%%%%%%%%%%%%%%%%%%%%%%%%%%%%%%%%%%%%%%%%%%%%%%%%%%%%%%%%%%%%
%%%%%%%%%%%%%%%%%%%%%%%%%%%%%%%%%%%%%%%%%%%%%%%%%%%%%%%%%%%%%%%%%%%%%%%%%%%%%%%%%%%%%%%%%%%%%%%%%%%
%%%%%%%%%%%%%%%%%%%%%%%%%%%%%%%%%%%%%%%%%%%%%%%%%%%%%%%%%%%%%%%%%%%%%%%%%%%%%%%%%%%%%%%%%%%%%%%%%%%
%%%%%%%%%%%%%%%%%%%%%%%%%%%%%%%%%%%%%%%%%%%%%%%%%%%%%%%%%%%%%%%%%%%%%%%%%%%%%%%%%%%%%%%%%%%%%%%%%%%

\section{Product manifolds and Lie $n$-algebra morphisms}
\label{sec:productman}

%%%%%%%%%%%%%%%%%%%%%%%%%%%%%%%%%%%%%%%%%%%%%%%%%%%%%%%%%%%%%%%%%%%%%%%%%%%%%%%%%%%%%%%%%%%%%%%%%%%
%%%%%%%%%%%%%%%%%%%%%%%%%%%%%%%%%%%%%%%%%%%%%%%%%%%%%%%%%%%%%%%%%%%%%%%%%%%%%%%%%%%%%%%%%%%%%%%%%%%
%%%%%%%%%%%%%%%%%%%%%%%%%%%%%%%%%%%%%%%%%%%%%%%%%%%%%%%%%%%%%%%%%%%%%%%%%%%%%%%%%%%%%%%%%%%%%%%%%%%
%%%%%%%%%%%%%%%%%%%%%%%%%%%%%%%%%%%%%%%%%%%%%%%%%%%%%%%%%%%%%%%%%%%%%%%%%%%%%%%%%%%%%%%%%%%%%%%%%%%

Consider two multisymplectic manifolds $\left(\mathcal{M}_{a}, \omega_{a}\right),\, a=1,2,\,$ where $\omega_{a}$ is an $n_{a}$-plectic structure defined on $\mathcal{M}_{a}$. The goal of this section is, roughly speaking, to study the relation between the $n_{a}$-Lie algebra $L\left(\mathcal{M}_{a},\omega_{a}\right)$ constructed over $\mathcal{M}_{a}$ and the Lie-$n$ algebra $L\left(\mathcal{M},\omega\right)$ constructed over the product manifold, $\mathcal{M}=\mathcal{M}_{1}\times\mathcal{M}_{2}$, where $\omega = \mathrm{pr}^{\ast}_{1}\omega_{1}\wedge\mathrm{pr}^{\ast}_{2}\omega_{2}$,  $n = n_1 + n_2 +1$ and $\mathrm{pr}_{a}: \mathcal{M}\to\mathcal{M}_{a}$ is the canonical projection. 

Finding such a relation is relevant for at least two reasons. First, it help us to understand how $n$-plectic Lie algebras are related to the corresponding multisymplectic manifolds in a deeper way, since it give us information about how they behave when some operation is performed in the manifold, in this case the cartesian product. Secondly, it is relevant in order to construct an homotopy moment map\footnote{See section \ref{sec:homotopymoment}.} for the product manifold, assuming the homotopy moment maps for $(\mathcal{M}_{1},\omega_{1})$ and $(\mathcal{M}_{2},\omega_{2})$ exist. Indeed, let $G_{a}$ be a Lie group with Lie algebra $\mathfrak{g}_{a}$. Let $(\mathcal{M}_{a},\omega_{a})$ be a $n_{a}$-plectic manifold equipped with a $G_{a}$ action which preserves $\omega_{a}$ and such that the infinitesimal $\mathfrak{g}_{a}$ action is via Hamiltonian fields. Let us assume that the corresponding homotopy moment maps exist and are given by $f_{a}:\mathfrak{g}_{a}\to L\left(\mathcal{M}_{a},\omega_{a}\right)$. In that case, if $\mathrm{H}: L\left(\mathcal{M}_{1},\omega_{1}\right)\oplus L\left(\mathcal{M}_{2},\omega_{2}\right)\to L\left(\mathcal{M},\omega\right)$ is a $L_{\infty}$-morphism, then the composition of $f_{1}\oplus f_{2}$ and $\mathrm{H}$ is a very reasonable homotopy moment map candidate for the product manifold $G_{1}\times G_{2}\circlearrowleft\left(\mathcal{M},\omega\right)$. To obtain $\mathrm{H}$ by brute force seems to be a very involved task to perform in general. It will turn out to be easier to directly construct an $L_{\infty}$-morphism $F$ from $\mathfrak{g}_{1}\oplus\mathfrak{g}_{2}$ to $\left(\mathcal{M}_{1}\times\mathcal{M}_{2},\omega\right)$, making use of $f_{a}$, which can also be used to make an educated guess for $\mathrm{H}$, as it is illustrated by the following diagram

\begin{center}
\begin{tikzpicture}
\label{diag:productmoment}
  \matrix (m) [matrix of math nodes,row sep=8em,column sep=9em,minimum width=2em]
  {
L\left(\mathcal{M}_{1},\omega_{1}\right)\oplus L\left(\mathcal{M}_{2},\omega_{2}\right) & L\left(\mathcal{M},\omega\right) \\
\mathfrak{g}_{1}\oplus\mathfrak{g}_{2}  \\};
  \path[-stealth]
    (m-2-1) edge node [left] {$f_{1}$} node [right] {$f_{2}$} (m-1-1)
    (m-1-1) edge node [above] {$\mathrm{H}$} (m-1-2)
    (m-2-1) edge node [below] {$F$} (m-1-2);
\end{tikzpicture}
\end{center}

\noindent
Since it will be useful in the following, let us remember that the tangent bundle of the product manifold $\mathcal{M}$ can be written as follows\footnote{See section \ref{sec:fibrebundles}.}

\begin{equation}
T\mathcal{M} = \mathrm{pr}^{\ast}_{1} T\mathcal{M}_{1} \oplus \mathrm{pr}^{\ast}_{2} T\mathcal{M}_{2}\, , 
\end{equation}

\noindent
and therefore

\begin{equation}
\Gamma\left( T\mathcal{M}\right) = \Gamma\left( \mathrm{pr}^{\ast}_{1} T\mathcal{M}_{1}\right) \oplus \Gamma\left( \mathrm{pr}^{\ast}_{2} T\mathcal{M}_{2}\right)\, .
\end{equation}

\noindent
In particular, it holds $\mathrm{pr}^{\ast}_{a}\Gamma\left(T\mathcal{M}_{a}\right)\subset \Gamma\left(\mathrm{pr}^{\ast}_{a}T\mathcal{M}_{a}\right)\subset \Gamma\left(T\mathcal{M}\right)$. The first task is now to check that $\omega$ is indeed an $n$-plectic structure on $\mathcal{M}$, provided $\omega_{a}$ is an $n_{a}$-plectic structure on $\mathcal{M}_{a}$. It is straightforward to see that $\omega$ is closed

\begin{equation}
d\omega = \mathrm{pr}^{\ast}_{1}d\omega_{1}\wedge\mathrm{pr}_2\omega_{2} + (-1)^{n_{1}} \mathrm{pr}^{\ast}_{1}\omega_{1}\wedge\mathrm{pr}^{\ast}_{2}d\omega_{2} = 0\, .
\end{equation} 

\noindent
To see that it is non-degenerate, let us assume that there exists a vector field $v\in\Gamma\left(\mathcal{M}\right)$ such that $i_{v}\omega = 0$. There exist then vector fields $v_{a}\in \Gamma\left(\mathcal{M}_{a}\right)$ such that 

\begin{equation}
v = \mathrm{pr}^{\ast}_{1} v_{1} + \mathrm{pr}^{\ast}_{2} v_{2}\, ,
\end{equation} 

\noindent
and therefore

\begin{equation}
\iota_{v}\omega = \iota_{v}\mathrm{pr}^{\ast}_{1} \omega_{1}\wedge\mathrm{pr}^{\ast}_{2} \omega_{2} + (-1)^{n_{1}+1}\mathrm{pr}^{\ast}_{1} \omega_{1}\wedge\iota_{v}\mathrm{pr}^{\ast}_{2} \omega_{2}\, ,
\end{equation} 

\noindent
implies $v_{a} = 0$ and $v=0$ since $\omega_{a}$ is non-degenerate.  Let us consider now $X_{\alpha_{a}} \in \mathfrak{X}_{\mathrm{Ham}}\left(\mathcal{M}_{a}\right) = \Gamma_{\mathrm{Ham}}\left(T\mathcal{M}_{a}\right)$ and construct the vector field

\begin{equation}
X_{\alpha} = \mathrm{pr}^{\ast}_{1} X_{\alpha_{1}} + \mathrm{pr}^{\ast}_{2} X_{\alpha_{2}}\, .
\end{equation}

\noindent
We have then

\begin{equation}
\label{eq:Ixw}
i_{X_{\alpha}}\omega = - d\left[\mathrm{pr}^{\ast}_{1} \alpha_{1}\wedge\mathrm{pr}^{\ast}_{2}\omega + \mathrm{pr}^{\ast}_{1}\omega_{1}\wedge\mathrm{pr}^{\ast}_{2}\alpha_{2}\right] = - d\alpha\, ,
\end{equation}

\noindent
and hence $X_{\alpha}$ is a hamiltonian vector field for $\omega$ with hamiltonian $(n_{1}+n_{2})$-form $\alpha$, which is of course defined up to a closed form. Therefore we have

\begin{equation}
\mathrm{pr}^{\ast}_{1}\Gamma_{\mathrm{Ham}}\left(T\mathcal{M}_{1}\right) + \mathrm{pr}^{\ast}_{2}\Gamma_{\mathrm{Ham}}\left(T\mathcal{M}_{1}\right)\subseteq \Gamma_{\mathrm{Ham}}\left(T\mathcal{M}\right) \, .
\end{equation}

\noindent
The following example shows that in general 

\begin{equation}
\mathrm{pr}^{\ast}_{1}\Gamma_{\mathrm{Ham}}\left(T\mathcal{M}_{1}\right) + \mathrm{pr}^{\ast}_{2}\Gamma_{\mathrm{Ham}}\left(T\mathcal{M}_{1}\right)\neq \Gamma_{\mathrm{Ham}}\left(T\mathcal{M}\right) \, .
\end{equation}

\noindent
\begin{ep}
\label{ep:hamdisthama} Let $\mathcal{M}_{a} = \mathbb{R}^{2}$ with coordinates $(x^{1}_{a}, x^{2}_{a})$ equipped with the volume form $\omega_{a} = f_{a}\, dx^{1}_{a}\wedge dx^{2}_{a}$, where $f_{a}\in C^{\infty}\left(\mathbb{R}^{2}\right)$ is a no-where vanishing, non-constant, differentiable function. Then

\begin{equation}
\mathcal{M} = \mathbb{R}^4\, , \qquad \omega = f_{1} f_{2}\, dx^{1}\wedge\cdots\wedge dx^{4}
\end{equation}

\noindent
where the coordinates of $\mathbb{R}^{4}$ are denoted by $(x^{1},\dots , x^{4})$. Then 

\begin{equation}
X_{\alpha} = -\frac{1}{f_{1} f_{2}} \frac{\partial}{\partial x^{1}}\, , \qquad \alpha = x^{2} dx^{3}\wedge dx^{4}\, ,
\end{equation}

\noindent
is a Hamiltonian vector field which cannot be written as a sum of Hamiltonian vector fields $X_{\alpha_{a}}$ of $(\mathbb{R}^{2},\omega_{a})$.
\end{ep}

\noindent
We are going to define now two applications $h_{\Omega}$ and $h_{\mathfrak{X}}$ as follows

\begin{eqnarray}
h_{\Omega}: \Omega^{n_{1}-1}_{\mathrm{Ham}}\left(\mathcal{M}_{1}\right)\oplus \Omega^{n_{2}-1}_{\mathrm{Ham}}\left(\mathcal{M}_{2}\right) &\to &  \Omega^{n_{1}+n_{2}}_{\mathrm{Ham}}\left(\mathcal{M}\right)\nonumber\\
\alpha_{1}\oplus \alpha_{2} &\mapsto &\mathrm{pr}^{\ast}_{1} \alpha_{1}\wedge\mathrm{pr}^{\ast}_{2}\omega_{2} + \mathrm{pr}^{\ast}_{1}\omega_{1}\wedge\mathrm{pr}^{\ast}_{2}\alpha_{2}\, ,
\end{eqnarray}

\begin{eqnarray}
h_{\mathfrak{X}}: \mathfrak{X}_{\mathrm{Ham}}\left(\mathcal{M}_{1}\right) \oplus  \mathfrak{X}_{\mathrm{Ham}}\left(\mathcal{M}_{2}\right)&\to & \mathfrak{X}_{\mathrm{Ham}}\left(\mathcal{M}\right)\nonumber\\
X_{\alpha_{1}}\oplus X_{\alpha_{2}} &\mapsto &\mathrm{pr}^{\ast}_{1} X_{\alpha_{1}} + \mathrm{pr}^{\ast}_{2} X_{\alpha_{2}}\, ,
\end{eqnarray}

\noindent
which make the following diagram commutative

\begin{center}
\begin{tikzpicture}
\label{diag:commutativemorphismII}
  \matrix (m) [matrix of math nodes,row sep=8em,column sep=9em,minimum width=2em]
  {
\Omega^{n_{1}-1}_{\mathrm{Ham}}\left(\mathcal{M}_{1}\right)\oplus \Omega^{n_{2}-1}_{\mathrm{Ham}}\left(\mathcal{M}_{2}\right) & \Omega^{n_{1}+n_{2}}_{\mathrm{Ham}}\left(\mathcal{M}\right) \\
\mathfrak{X}_{\mathrm{Ham}}\left(\mathcal{M}_{1}\right)\oplus \mathfrak{X}_{\mathrm{Ham}}\left(\mathcal{M}_{2}\right) & \mathfrak{X}_{\mathrm{Ham}}\left(\mathcal{M}\right) \\};
  \path[-stealth]
    (m-1-1) edge node [left] {$j$} (m-2-1)
    (m-1-1) edge node [above] {$h_{\Omega}$} (m-1-2)
    (m-2-1) edge node [above] {$h_{\mathfrak{X}}$} (m-2-2)
    (m-1-2) edge node [right] {$k$} (m-2-2);
\end{tikzpicture}
\end{center}

\noindent
Here $j\left(\alpha_{1}\oplus\alpha_{2}\right) = X_{\alpha_{1}}\oplus X_{\alpha_{2}}$ and $k(\alpha) = X_{\alpha}$. The vector space $\mathfrak{X}_{\mathrm{Ham}}\left(\mathcal{M}_{1}\right)\oplus \mathfrak{X}_{\mathrm{Ham}}\left(\mathcal{M}_{2}\right)$ over the real numbers $\mathbb{R}$ can be endowed with an $\mathbb{R}$-linear Lie bracket 

\begin{equation}
\left[\cdot,\cdot\right]_{0}:\mathfrak{X}_{\mathrm{Ham}}\left(\mathcal{M}_{1}\right)\oplus \mathfrak{X}_{\mathrm{Ham}}\left(\mathcal{M}_{2}\right)\oplus \mathfrak{X}_{\mathrm{Ham}}\left(\mathcal{M}_{1}\right)\oplus \mathfrak{X}_{\mathrm{Ham}}\left(\mathcal{M}_{2}\right) \to \mathfrak{X}_{\mathrm{Ham}}\left(\mathcal{M}_{1}\right)\oplus \mathfrak{X}_{\mathrm{Ham}}\left(\mathcal{M}_{2}\right) \, , 
\end{equation}

\noindent
defined as 

\begin{equation}
\left[X_{\alpha_{1}}\oplus X_{\alpha_{2}}, X_{\beta_{1}}\oplus X_{\beta_{2}} \right]_{0} = \left[X_{\alpha_{1}},X_{\beta_{1}}\right]_{1}\oplus \left[X_{\alpha_{2}},X_{\beta_{2}}\right]_{2}\, ,
\end{equation}

\noindent
where $\left[\cdot,\cdot\right]_{a}:\mathfrak{X}_{\mathrm{Ham}}\left(\mathcal{M}_{a}\right)\oplus \mathfrak{X}_{\mathrm{Ham}}\left(\mathcal{M}_{a}\right)\to \mathfrak{X}_{\mathrm{Ham}}\left(\mathcal{M}_{a}\right)$ is the canonical Lie brackets defined on $\mathfrak{X}\left(\mathcal{M}_{a}\right)$ and evaluated on $\mathfrak{X}_{\mathrm{Ham}}\left(\mathcal{M}_{a}\right)$. It can be easily seen that $h_{\mathrm{X}}$ preserves the Lie bracket, that is

\begin{equation}
h_{\mathfrak{X}}\left(\left[X_{\alpha_{1}}\oplus X_{\alpha_{2}}, X_{\beta_{1}}\oplus X_{\beta_{2}} \right]_{0}\right) = \left[ h_{\mathfrak{X}}(X_{\alpha_{1}}\oplus X_{\beta_{1}}), h_{\mathfrak{X}}(X_{\alpha_{2}}\oplus X_{\beta_{2}}) \right]\, .
\end{equation}

\noindent
Similarly, $\Omega^{n_{1}-1}_{\mathrm{Ham}}\left(\mathcal{M}_{1}\right)\oplus \Omega^{n_{2}-1}_{\mathrm{Ham}}\left(\mathcal{M}_{2}\right)$ can be endowed with a bracket

\begin{equation}
 \left\{\cdot,\cdot\right\}_{0}: \Omega^{n_{1}-1}_{\mathrm{Ham}}\left(\mathcal{M}_{1}\right)\oplus \Omega^{n_{2}-1}_{\mathrm{Ham}}\left(\mathcal{M}_{2}\right)\oplus \Omega^{n_{1}-1}_{\mathrm{Ham}}\left(\mathcal{M}_{1}\right)\oplus \Omega^{n_{2}-1}_{\mathrm{Ham}}\left(\mathcal{M}_{2}\right)\to \Omega^{n_{1}-1}_{\mathrm{Ham}}\left(\mathcal{M}_{1}\right)\oplus \Omega^{n_{2}-1}_{\mathrm{Ham}}\left(\mathcal{M}_{2}\right)
\end{equation}

\noindent
defined as follows

\begin{equation}
\left\{ \alpha_{1}\oplus\alpha_{2}, \beta_{1}\oplus\beta_{2}\right\}_{0} 
 = \left\{\alpha_{1},\beta_{1} \right\}_{1}\oplus \left\{\alpha_{2},\beta_{2} \right\}_{2}\, ,
\end{equation}

\noindent
where $\left\{\cdot,\cdot\right\}_{a}:\Omega^{n_{a}-1}_{\mathrm{Ham}}\left(\mathcal{M}_{a}\right)\oplus \Omega^{n_{a}-1}_{\mathrm{Ham}}\left(\mathcal{M}_{a}\right)\to \Omega^{n_{a}-1}_{\mathrm{Ham}}\left(\mathcal{M}_{a}\right)$ is the canonical Hamiltonian bracket defined on $\Omega^{n_{a}-1}_{\mathrm{Ham}}\left(\mathcal{M}_{a}\right)$. However, in this case, $h_{\Omega}$ does not preserve the bracket $\left\{\cdot,\cdot\right\}_{0}$, namely

\begin{equation}
h_{\Omega}\left(\left\{\alpha_{1}\oplus\alpha_{2},\beta_{1}\oplus\beta_{2}\right\}_{0}\right) = \left\{ h_{\Omega}(\alpha_{1}\oplus\alpha_{2}), h_{\Omega}(\beta_{1}\oplus\beta_{2})\right\} + (-1)^{n_1} d\left[ \mathrm{pr}^{\ast}_{1}\alpha_{1}\wedge \mathrm{pr}^{\ast}_{2} d\beta_{2} - \mathrm{pr}^{\ast}_{1}\beta_{1}\wedge \mathrm{pr}^{\ast}_{2} d\alpha_{2}\right]\, . 
\end{equation}

\noindent
As we have previously stated, the goal of this section is to relate the $n_{a}$-Lie algebras $L\left(\mathcal{M}_{a},\omega_{a}\right)$ constructued over $\left(\mathcal{M}_{a},\omega_{a}\right)$ to the $(n_{1}+n_{2}+1)$-Lie algebra $L\left(\mathcal{M},\omega\right)$ constructed over $\left(\mathcal{M},\omega\right)$. More precisely, we want to construct an Lie-$n$ algebra morphism from $L\left(\mathcal{M}_{1},\omega_{1}\right)\oplus L\left(\mathcal{M}_{2},\omega_{2}\right)$ to $L\left(\mathcal{M},\omega\right)$. To construct such morphism, $h_{\Omega}$ is going to be extremely relevant. The idea is to use $h_{\Omega}$ as the very first component of the $L_{\infty}$-algebra morphism $\mathrm{H}$. This suggested by the fact that $h_{\Omega}$ does not preserve the bracket $ \left\{\cdot,\cdot\right\}_{0}$ by an exact form, something that is characteristic of the corresponding component in a $L_{\infty}$-morphism. Assuming therefore that $H_{1} = h_{\Omega}$, we expect to obtain the form af all the other components of $\mathrm{H}$ by imposing the defining and consistency conditions that $\mathrm{H}$ has to obey, namely (\ref{eq:main_eq_1}) and (\ref{eq:main_eq_2}). However, this is an extremely involved procedure, so we have been able to check it only in the simplest case, where $\left(\mathcal{M}_{a},\omega_{a}\right)$ are both symplectic spaces. In any case, let us stress that we expect the procedure to hold in full generality.

%%%%%%%%%%%%%%%%%%%%%%%%%%%%%%%%%%%%%%%%%%%%%%%%%%%%%%%%%%%%%%%%%%%%%%%%%%%%%%%%%%%%%%%%%%%%%%%%%%%
%%%%%%%%%%%%%%%%%%%%%%%%%%%%%%%%%%%%%%%%%%%%%%%%%%%%%%%%%%%%%%%%%%%%%%%%%%%%%%%%%%%%%%%%%%%%%%%%%%%

\subsection{  $\left(\mathcal{M}_{a},\omega_{a}\right)$ Symplectic manifolds}

%%%%%%%%%%%%%%%%%%%%%%%%%%%%%%%%%%%%%%%%%%%%%%%%%%%%%%%%%%%%%%%%%%%%%%%%%%%%%%%%%%%%%%%%%%%%%%%%%%%
%%%%%%%%%%%%%%%%%%%%%%%%%%%%%%%%%%%%%%%%%%%%%%%%%%%%%%%%%%%%%%%%%%%%%%%%%%%%%%%%%%%%%%%%%%%%%%%%%%%

Since $\left(\mathcal{M}_{a},\omega_{a}\right)$ is a 1-plectic manifold, we have that  $\left(\mathcal{M},\omega\right)$ is a 3-plectic manifold. Consequently, the cochain complex $L$ of the Lie 3-algebra  $L\left(\mathcal{M}_{a},\omega_{a}\right)$ is given by

\begin{equation}
L: C^{\infty}\left(\mathcal{M}\right)\to\Omega^{1}\left(\mathcal{M}\right)\to\Omega^{2}\left(\mathcal{M}\right)\to \Omega^{3}_{\mathrm{Ham}}\left(\mathcal{M}\right)\, ,
\end{equation}

\noindent
where the coboundary operator is the usual de Rham exterior derivative. On the other hand, the cochain complex $L_{a}$ which underlies the 1-Lie algebra $L\left(\mathcal{M}_{a},\omega_{a}\right)$ is simply

\begin{equation}
L_{a}: C^{\infty}\left(\mathcal{M}_{a}\right)\, .
\end{equation}

\noindent
In this simpler situation, $L\left(\mathcal{M}_{1},\omega_{1}\right)\oplus L\left(\mathcal{M}_{2},\omega_{2}\right)$ is just a regular Lie-algebra, so we can apply propostion \ref{prop:Lie_alg_P_cor} in order to obtain the remaining components of the $L_{\infty}$-morphism $\mathrm{H}$ from the Lie algebra $L\left(\mathcal{M}_{1},\omega_{1}\right)\oplus L\left(\mathcal{M}_{2},\omega_{2}\right)$ to the 3-Lie algebra $L\left(\mathcal{M},\omega\right)$. $\mathrm{H}$ consists of three maps

\begin{equation}
H_{k}: \left[C^{\infty}\left(\mathcal{M}_{1}\right)\times C^{\infty}\left(\mathcal{M}_{2}\right)\right]^{\otimes k} \to L\, , \qquad k=1,2,3\, .
\end{equation}

\noindent
Please notice that we are imposing $H_{1} = h_{\Omega}$. The two remaining components will be find by imposing the defining conditions of a $L_{\infty}$ morphism on $\mathrm{H}$, which in proposition \ref{prop:Lie_alg_P_cor} have been adapted and simplified to the case of a Lie algebra as the domain of the morphism. As expected, the procedure is consistent and $\mathrm{H}_{2}$ and $\mathrm{H}_{3}$ can be determined, making $\mathrm{H}$ into an honest $L_{\infty}$-morphism. They are given by\footnote{In equation (\ref{eq:H3}) we have omitted the $\mathrm{pr}^{\ast}_{a}$ in order to make more readable the expression.}

\begin{equation}
\mathrm{H}_{2}\left(f_{1},f_{2},g_{1},g_{2}\right) =\frac{1}{2} \left(\mathrm{pr}^{\ast}_{1} f_{1}\wedge \mathrm{pr}^{\ast}_{2}dg_{2} - \mathrm{pr}^{\ast}_{1}df_{1}\wedge \mathrm{pr}^{\ast}_{2} g_{2} -\mathrm{pr}^{\ast}_{1} g_{1}\wedge \mathrm{pr}^{\ast}_{2}df_{2} + \mathrm{pr}^{\ast}_{1}dg_{1}\wedge \mathrm{pr}^{\ast}_{2} f_{2}\right)\, ,
\end{equation}

\begin{equation}
\label{eq:H3}
\mathrm{H}_{3}\left(f_{1},f_{2},g_{1},g_{2},h_{1},h_{2}\right) = \frac{1}{2}\left(f_{1}\left\{g_{2},h_{2}\right\} + f_{2}\left\{g_{1},h_{1}\right\} - g_{1}\left\{f_{2},h_{2}\right\} - g_{2}\left\{f_{1},h_{1}\right\} + h_{1}\left\{f_{2},g_{2}\right\} + h_{2}\left\{f_{1},g_{1}\right\}\right)\, ,
\end{equation}

\noindent
for all $f_{a}, g_{a}, h_{a} \in C^{\infty}\left(\mathcal{M}_{a}\right)\, ,\,\, a =1,2$. As explained in section \ref{sec:homotopymoment}, in order for H to be an homotopy moment map we have to check that

\begin{equation}
\label{eq:conditionhomotopy}
-\iota_{v_{x}} \omega = d\left(f_{1}(x)\right)\, ,
\end{equation}

\noindent
which holds by equation (\ref{eq:Ixw}). Therefore, we have constructed explicitly a Lie-$n$ algebra morphism from $L\left(\mathcal{M}_{1},\omega_{1}\right)\times L\left(\mathcal{M}_{2},\omega_{2}\right)$ to $L\left(\mathcal{M},\omega\right)$, and we expect the same procedure to hold in the general case. However, it turns out to be too involved to be carried out explicitly and a different approach is needed.

%%%%%%%%%%%%%%%%%%%%%%%%%%%%%%%%%%%%%%%%%%%%%%%%%%%%%%%%%%%%%%%%%%%%%%%%%%%%%%%%%%%%%%%%%%%%%%%%%%%
%%%%%%%%%%%%%%%%%%%%%%%%%%%%%%%%%%%%%%%%%%%%%%%%%%%%%%%%%%%%%%%%%%%%%%%%%%%%%%%%%%%%%%%%%%%%%%%%%%%

\subsection{ Product homotopy moment maps}
\label{sec:producthomotpy}

%%%%%%%%%%%%%%%%%%%%%%%%%%%%%%%%%%%%%%%%%%%%%%%%%%%%%%%%%%%%%%%%%%%%%%%%%%%%%%%%%%%%%%%%%%%%%%%%%%%
%%%%%%%%%%%%%%%%%%%%%%%%%%%%%%%%%%%%%%%%%%%%%%%%%%%%%%%%%%%%%%%%%%%%%%%%%%%%%%%%%%%%%%%%%%%%%%%%%%%

Since finding $\mathrm{H}$ explicitly seems to be a too complicated task to be performed by brute force, in this section we are going to pursue a different path. We are going to build an $L_{\infty}$-morphism from\footnote{In this section we are going to use a different notation, more suitable for the expressions that we will find. $a$ and $b$ are not indices but labels that denote different objects.} $\mathfrak{g}_{a}\oplus\mathfrak{g}_{b}$ to $L\left(\mathcal{M}=\mathcal{M}_{a}\times\mathcal{M}_{b}, \omega = \mathrm{pr}^{\ast}_{a}\omega_{a}\wedge \mathrm{pr}^{\ast}_{b}\omega_{b}\right)$, assuming that there exist homotopy moment maps $f_{C}$ for $G_{C}\circlearrowleft\left(\mathcal{M}_{C},\omega_{C}\right)\, ,\,\, C=a,b\,$. Here $\mathfrak{g}_{C}$ is the Lie algebras of the Lie group $G_{C}$. $\mathrm{F}$ will give us an homotopy moment map for the product manifold in terms of homotopy moment maps of the factors. Besides, $\mathrm{F}$ may give us also the opportunity to make an educated guess for $\mathrm{H}:L\left(\mathcal{M}_{a},\omega_{a}\right)\oplus L\left(\mathcal{M}_{b},\omega_{b}\right) \to L\left(\mathcal{M},\omega\right)$\footnote{See reference \cite{CarlosMarco}.}. The precise result is the following

\thm{\label{thm:morphismproduct} Let $G_{C}$ be a Lie group with Lie algebra $\mathfrak{g}_{C}$, where $C=a,b$. Let $(\mathcal{M}_{C},\omega_{C})$ be a $n_{C}$-plectic manifold equipped with a $G_{C}$ action which preserves $\omega_{C}$ and such that the infinitesimal $\mathfrak{g}_{C}$ action is via Hamiltonian fields. Let us assume that the corresponding homotopy moment maps exist and are given by $f^{C}:\mathfrak{g}_{C}\to L\left(\mathcal{M}_{C},\omega_{C}\right)$. Then the following set of maps  

\begin{eqnarray}
\label{eq:producthomotopymomentmap}
F_{k}:\left(\mathfrak{g}_{a}\oplus \mathfrak{g}_{b}\right)^{\otimes k} &\to & L\left(\mathcal{M},\omega\right)\, ,\qquad k=1,\dots , n_{1}+n_{2}+1\nonumber\\ \left(x^{1}_{a}\oplus x^{1}_{b}, \dots , x^{k}_{a}\oplus x^{k}_{b}\right) &\mapsto & \sum_{l=0}^{k}\sum_{\sigma\in\mathrm{Sh}(l,k-l)} (-1)^{\sigma}c^{a}_{l,k-l}\mathrm{pr}^{\ast}_{a} f^{a}_{l}\left(x^{1}_{a}, \dots , x^{l}_{a}\right)\wedge i_{l+1, \dots ,k}\mathrm{pr}^{\ast}_{b}\omega_{b} \\&+& \sum_{l=0}^{k}\sum_{\sigma\in\mathrm{Sh}(l,k-l)} (-1)^{\sigma} c^{b}_{l,k-l}i_{1, \dots ,l}\mathrm{pr}^{\ast}_{a}\omega_{a} \wedge\mathrm{pr}^{\ast}_{b}f^{b}_{(k-l)}\left(x^{l+1}_{b}, \dots , x^{k}_{b}\right)\nonumber\, ,
\end{eqnarray}

\noindent
of degree $|F_{k}|=1-k$ form a $L_{\infty}$-morphism $\mathrm{F}$ from $\mathfrak{g}_{1}\oplus\mathfrak{g}_{2}$ to $\left(\mathcal{M}_{1}\times\mathcal{M}_{2},\omega\right)$. We assume by definition that $f^{a}_{0} = f^{b}_{0} = 0$. In addition, $\mathrm{F}$, is homotopy moment map for the action of $G_{1}\times G_{2}$ on $\left(\mathcal{M},\omega\right)$. Here we have defined\footnote{Please remember that $\xi(k) = -(-1)^{\frac{k(k+1)}{2}}$.}

\begin{equation}
i_{1, \dots ,l}\,\mathrm{pr}^{\ast}_{C}\omega_{C} = \iota\left(\mathrm{pr}^{\ast}_{C} v_{f^{C}_{1}\left(x^{1}_{C}\right)}\wedge \dots \wedge \mathrm{pr}^{\ast}_{C} v_{f^{C}_{1}\left(x^{l}_{C}\right)}\right)\mathrm{pr}^{\ast}_{C}\omega_{C}\, ,
\end{equation}

%\begin{equation}
%\label{eq:cs1}
%c^{a}_{\alpha_{1},\alpha_{2}} = \frac{(-1)^{(\alpha_{2} - 1 - 2 n_{a})\frac{\alpha_{2}}{2}}}{2} \, ,\qquad c^{b}_{\alpha_{1},\alpha_{2}} = \frac{(-1)^{\frac{k(k+1)+\alpha_{2} (\alpha_{2}+1)}{2}+(1+\alpha_{2})(-1-n_{a}+\alpha_{1})}}{2}\, ,\qquad %\alpha_{1}+\alpha_{2} = k >1 \, ,
%\end{equation}

\begin{eqnarray}
\label{eq:cs1}
c^{a}_{l,k-l} & = &\frac{1}{2}\xi (k) \xi (l)(-1)^{(n_{a}+1-l)(k-l)} \, ,\,\,\, l=1,\hdots,k\, ,\quad k>1  \\c^{b}_{l,k-l} & = & \frac{1}{2}\xi (k) \xi (k-l)(-1)^{(n_{a}+1-l)(k-l-1)}\, ,\,\,\, l = 0,\hdots,k-1\, ,\quad k>1 \, ,
\end{eqnarray}

\noindent
and

\begin{equation}
\label{eq:cs2}
c^{a}_{1,0} = 1\, ,\qquad c^{b}_{0,1} = 1\, .
\end{equation}

 }

\proof{We are going to proceed by proving that the $F_{k}:\left(\mathfrak{g}_{a}\oplus \mathfrak{g}_{b}\right)^{\otimes k} \to  L\left(\mathcal{M},\omega\right)\, ,\,\, k = 1,\dots,n_{1}+n_{2}+1\, ,$ assuming (\ref{eq:cs1}) and (\ref{eq:cs2}), are indeed the components of a $L_{\infty}$-morphism by checking that it obeys the conditions (\ref{eq:main_eq_1}) and (\ref{eq:main_eq_2}). Since $F_{k}$ is multilinear, it is enough to check the conditions on elements of the form $(x_{a}^{1},\cdots,x_{a}^{l},x_{b}^{l+1},\cdots, x_{b}^{k}) \in \left(\mathfrak{g}_{a}\oplus \mathfrak{g}_{b}\right)^{\otimes k}$. For $k=1$ we only need to check equation (\ref{eq:condicionham}), which is the condition of $F$ being an homotopy moment map. We can easily see that it is indeed satisfied. We will consider therefore $k>1$. We can write

\begin{equation}
F_{k}\left(x_{a}^{1},\dots,x_{a}^{l},x_{b}^{l+1},\dots, x_{b}^{k}\right) = c^{a}_{l,k-l}f^{a}_{l}\left(x^{1}_{a}, \dots , x^{l}_{a}\right) \iota_{l+1, \dots ,k} \omega_{b} + c^{b}_{l,k-l}\iota_{1, \dots ,l}\omega_{a} f^{b}_{(k-l)}\left(x^{l+1}_{b}, \dots , x^{k}_{b}\right)\, ,
\end{equation}

\noindent
where we have suppressed the $\pr^{a}$ and the $\wedge$ to ease the presentation. We first check condition (\ref{eq:main_eq_1}). The left hand side is given by

\begin{eqnarray}
\sum_{1\leq i < j \leq k} (-1)^{i+j+1} F_{k-1}\left(\left[ x^{i}, x^{j}\right],x^{1},\dots,\hat{x}^{i},\dots,\hat{x}^{j},\dots,x^{k}\right) =\nonumber \\ \sum_{1\leq i < j \leq l} (-1)^{i+j+1} F_{k-1}\left(\left[ x^{i}_{a}, x^{j}_{a}\right],x^{1}_{a},\dots,\hat{x}^{i}_{a},\dots,\hat{x}^{j}_{a},\dots,x^{l}_{a},x^{l+1}_{b},\dots, x^{k}\right)\\+ \sum_{l+1\leq i < j \leq k} (-1)^{i+j+1} F_{k-1}\left(\left[ x^{i}_{b}, x^{j}_{b}\right],x^{1}_{a},\dots,x^{l}_{a},x^{l+1}_{b},\dots,\hat{x}^{i}_{b},\dots,\hat{x}^{j}_{b},\dots, x^{k}\right) = (A) + (B) + (C) + (D)\nonumber\, ,
\end{eqnarray}

\noindent
where

\begin{equation}
(A) = \sum_{1\leq i < j \leq l} (-1)^{i+j+1} c^{a}_{l-1,k-l} f^{a}_{l-1} \left(\left[ x^{i}_{a}, x^{j}_{a}\right],x^{1}_{a},\dots,\hat{x}^{i}_{a},\dots,\hat{x}^{j}_{a},\dots,x^{l}_{a}\right) \iota_{l+1,\dots,k}\omega_{b}\, ,
\end{equation}

\begin{equation}
(B) = \sum_{1\leq i < j \leq l} (-1)^{i+j+1} c^{b}_{l-1,k-l} \iota\left( [i,j],1,\dots,\hat{i},\dots,\hat{j},\dots , k\right) \omega_{a} f^{b}_{k-l} \left(x^{l+1}_{b},\dots, x^{k}_{b}\right)\, ,
\end{equation}

\begin{equation}
(C) = \sum_{l+1\leq i < j \leq k} (-1)^{l+i+j+1} c^{a}_{l,k-l-1} f^{a}_{l} \left(x^{1}_{a},\dots,x^{l}_{a}\right) \iota \left([i,j], l+1,\dots,\hat{i},\dots , \hat{j}, \dots, k\right)\omega_{b}\, ,
\end{equation}

\begin{equation}
(D) = \sum_{l+1\leq i < j \leq k} (-1)^{l+i+j+1} c^{b}_{l,k-l-1} \iota\left(1,\dots , l\right) \omega_{a} f^{b}_{k-l-1} \left(\left[ x^{i}_{b}, x^{j}_{b}\right],x^{l+1}_{b},\dots,\hat{x}^{i}_{b},\dots,\hat{x}^{j}_{b},\dots,x^{k}_{b}\right)\, .
\end{equation}

\noindent
The right hand side of (\ref{eq:main_eq_1}) is given by

\begin{eqnarray}
& & \xi (k)\iota\left( v_{x^{1}_{a}}\wedge\cdots \wedge v_{x^{l}_{a}}\wedge v_{x^{l+1}_{b}}\wedge\cdots\wedge  v_{x^{k}_{b}}\right) \omega  =  \xi(k) (-1)^{(n_{a}+1-l)(k-l)} \iota_{1,\dots,l}\omega_{a}\iota_{l+1,\dots,k}\omega_{b} \nonumber \\ &=& \frac{1}{2}\xi(k) (-1)^{(n_{a}+1-l)(k-l)} \left(\xi(l) l^{a}(x^1_{a},\dots, x^l_{a})\iota_{l+1,\dots,k}\omega_{b} +\xi(k-l)\iota_{1,\dots,l}\omega_{a} l^{b}(x^{l+1}_{b},\dots, x^{k}_{b}) \right)\\ \nonumber &=& \frac{1}{2} (E_{1} + E_{2})\, ,
\end{eqnarray}

\noindent
where

\begin{equation}
E_{1} = \xi(k) (-1)^{(n_{a}+1-l)(k-l)} \xi(l) l^{a}(x^1_{a},\dots, x^l_{a})\iota_{l+1,\dots,k}\omega_{b}\, ,\quad E_{2}=\xi(k) (-1)^{(n_{a}+1-l)(k-l)}\xi(k-l)\iota_{1,\dots,l}\omega_{a} l^{b}(x^{l+1}_{b},\dots, x^{k}_{b})\nonumber\, ,
\end{equation}

\noindent
together with

\begin{equation}
dF_{k}\left(x_{a}^{1},\dots,x_{a}^{l},x_{b}^{l+1},\dots, x_{b}^{k}\right) = (A^{\prime}) + (B^{\prime}) + (C^{\prime}) + (D^{\prime})\, ,
\end{equation}

\noindent
where

\begin{equation}
(A^{\prime}) = c^{a}_{l,k-l} df^{a}_{l}\left(x^{1}_{a}, \dots , x^{l}_{a}\right) \iota_{l+1, \dots ,k} \omega_{b}\, ,
\end{equation}

\begin{equation}
(B^{\prime}) = (-1)^{n_{a}-l} c^{a}_{l,k-l} f^{a}_{l}\left(x^{1}_{a}, \dots , x^{l}_{a}\right) d\iota_{l+1, \dots ,k} \omega_{b}\, ,
\end{equation}

\begin{equation}
(C^{\prime}) = c^{b}_{l,k-l} d\iota_{1, \dots ,l}\omega_{a} f^{b}_{(k-l)}\left(x^{l+1}_{b}, \dots , x^{k}_{b}\right)\, ,
\end{equation}

\begin{equation}
(D^{\prime}) = (-1)^{n_{a}+1-l} c^{b}_{l,k-l}\iota_{1, \dots ,l}\omega_{a} df^{b}_{(k-l)}\left(x^{l+1}_{b}, \dots , x^{k}_{b}\right)\, .
\end{equation}

\noindent
We can see now that 

\begin{equation}
(A)  = (A^{\prime}) + \frac{1}{2}(E_{1})\, , \qquad (D) = (D^{\prime}) + \frac{1}{2}(E_{2})\, ,
\end{equation}

\noindent
are precisely the condition that $f^{a}$ and $f^{b}$ obey respectively. In addition, we have

\begin{equation}
(B) = (C^{\prime}) \, , \qquad (C) = (B^{\prime}) \, ,
\end{equation}

\noindent
and therefore (\ref{eq:main_eq_1}) is satisfied. Notice that we have used the following identity

\begin{equation}
d\iota\left(v_{1}\wedge\dots\wedge v_{k}\right)\omega = (-1)^{k}\sum_{1\leq i<j\leq k}(-1)^{i+j}\iota\left(\left[ v_{i},v_{j}\right]\wedge v_{1}\wedge\dots\wedge v_{i}\wedge\dots\wedge v_{j}\wedge\dots\wedge v_{k}\right)\omega\, ,
\end{equation}

\noindent
that holds for $v_{1},\hdots , v_{k}\in\mathfrak{X}_{\mathrm{Ham}}\left(\mathcal{M}\right)$ Hamiltonian vector fields of a multisymplectic manifold $\left( \mathcal{M},\omega\right)$. We finally check (\ref{eq:main_eq_2}). Let us rewrite the condition here

\begin{eqnarray}
\label{eq:lasthommoment}
& &\sum_{1\leq i < j \leq n_{a}+n_{b}+2} (-1)^{i+j+1} F_{n_{a}+n_{b}+1}\left(\left[ x^{i}, x^{j}\right],x^{1},\dots,\hat{x}^{i},\dots,\hat{x}^{j},\dots,x^{k}\right) = \nonumber\\ & &\sum_{1\leq i < j \leq l} (-1)^{i+j+1} F_{n_{a}+n_{b}+1}\left(\left[ x^{i}_{a}, x^{j}_{a}\right],x^{1}_{a},\dots,\hat{x}^{i}_{a},\dots,\hat{x}^{j}_{a},\dots, x^{l}_{a},x^{l+1}_{b},\dots, x^{n_{a}+n_{b}+2}_{b}\right) \nonumber\\& &\sum_{l+1\leq i < j \leq n_{a}+n_{b}+2} (-1)^{i+j+1} F_{n_{a}+n_{b}+1}\left(\left[ x^{i}_{b}, x^{j}_{b}\right],x^{1}_{a},\dots,x^{l}_{a},x^{l+1}_{b},\dots,\hat{x}^{i}_{b},\dots,\hat{x}^{j}_{b}\dots, x^{n_{a}+n_{b}+2}_{b}\right) \nonumber\\ & = & l_{n_{a}+n_{b}+2}\left(F_{1}(x^{1}),\dots,F_{1}(x^{n_{a}+n_{b}+2})\right)\, .
\end{eqnarray}

\noindent
Please notice now that equation (\ref{eq:lasthommoment}) vanishes identically except when $l=n_{a}+1$. Setting therefore $l=n_{a}+1$ equation (\ref{eq:lasthommoment}) is satisfied by using (\ref{eq:main_eq_2}) for $f^{C}$.

\qed}

\noindent
We have proven that given that if the action $G_{C}\circlearrowleft \left(\mathcal{M}_{C},\omega_{C}\right)$ is Hamiltonian (that is, there exists an homotopy moment map for the action), then one can construct an homotopy moment map for the action of $G_{a}\times G_{b}$ on $\left(\mathcal{M},\omega\right)$. In other words, we have proven the following corollary

\cor{\label{cor:hamisham} Given two multisymplectic manifolds $\left(\mathcal{M}_{C},\omega_{C}\right)\, , \,\, C=a,b\, ,$ equipped with the Hamiltonian action of a Lie group $G_{C}$, then the action of $G_{a}\times G_{b}$ on the product manifold $(\mathcal{M},\omega)$ is Hamiltonian, with homotopy moment map given by equation (\ref{eq:producthomotopymomentmap}).}

\proof{Since the action is Hamiltonian if there exist an Homotopy moment map, from theorem \ref{thm:morphismproduct} the result is obvious.}

%%%%%%%%%%%%%%%%%%%%%%%%%%%%%%%%%%%%%%%%%%%%%%%%%%%%%%%%%%%%%%%%%%%%%%%%%%%%%%%%%%%%%%%%%%%%%%%%%%%
%%%%%%%%%%%%%%%%%%%%%%%%%%%%%%%%%%%%%%%%%%%%%%%%%%%%%%%%%%%%%%%%%%%%%%%%%%%%%%%%%%%%%%%%%%%%%%%%%%%

\subsection{Applications}

%%%%%%%%%%%%%%%%%%%%%%%%%%%%%%%%%%%%%%%%%%%%%%%%%%%%%%%%%%%%%%%%%%%%%%%%%%%%%%%%%%%%%%%%%%%%%%%%%%%
%%%%%%%%%%%%%%%%%%%%%%%%%%%%%%%%%%%%%%%%%%%%%%%%%%%%%%%%%%%%%%%%%%%%%%%%%%%%%%%%%%%%%%%%%%%%%%%%%%%

In section \ref{sec:producthomotpy} we have shown how to build an homotopy moment map for the product manifold of two multisymplectic manifolds, assuming that an homotopy moment map for the individual manifolds exist. Here we are going to apply such construction to some specific examples of geometrical interest. Let us consider, as usual, two multisymplectic manifolds $(\mathcal{M}_{C},\omega_{C})\, ,\,\, C=a, b$. We assume that there is a Hamiltonian action of a Lie group $G_{C}\circlearrowleft \mathcal{M}_{C}$ with corresponding homotopy moment map $f^{C} : \mathfrak{g}_{C}\to L\left(\mathcal{M},\omega\right)$. By corollary \ref{cor:hamisham} we know that there is also a Hamiltonian action 

\begin{equation}
\label{eq:productaction}
G_{a}\times G_{b}\circlearrowleft \left(\mathcal{M}_{a}\times\mathcal{M}_{b},\pr^{\ast}_{a}\, \omega_{a}\wedge\pr^{\ast}_{b}\, \omega_{b} \right)\, ,
\end{equation}

\noindent
of $G_{a}\times G_{b}$ on $\left(\mathcal{M}_{a}\times\mathcal{M}_{b},\pr^{\ast}_{a}\, \omega_{a}\wedge\pr^{\ast}_{b}\, \omega_{b} \right)$ with homotopy moment map given by (\ref{eq:producthomotopymomentmap}). 

\begin{itemize}

%%%%%%%%%%%%%%%%%%%%%%%%%%%%%%%%%%%%%%%%%%%%%%%%%%%%%%%%%%%%%%%%%%%%%%%%%%%%%%%%%%%%%%%%%%%%%%%%%%%

\item $G_{a} = G_{b} = G$ acting on $\left(\mathcal{M}_{a}\times\mathcal{M}_{b},\pr^{\ast}_{a}\, \omega_{a}\wedge\pr^{\ast}_{b}\, \omega_{b} \right)$. 

In this case, the Hamiltonian action on the product manifold can be written as follows

\begin{equation}
\label{eq:productactionigual}
G\times G\circlearrowleft \left(\mathcal{M}_{a}\times\mathcal{M}_{b},\pr^{\ast}_{a}\, \omega_{a}\wedge\pr^{\ast}_{b}\, \omega_{b}\right)\, ,
\end{equation}

\noindent
and therefore one can restrict it to the diagonal $\Delta G = \left\{ (g, g)\, : \,\, g\in G\right\}$ of $G\times G$

\begin{equation}
\Delta G \circlearrowleft (\mathcal{M}_{a}\times\mathcal{M}_{b},\pr^{\ast}_{a}\, \omega_{a}\wedge\pr^{\ast}_{b}\, \omega_{b} )\, .
\end{equation}

\noindent
Notice that the homotopy moment map for the product action (\ref{eq:productaction}) is a $L_{\infty}$-algebra morphism

\begin{equation}
F: \mathfrak{g}_{a}\oplus \mathfrak{g}_{b} \to L\left(\mathcal{M}_{a}\times\mathcal{M}_{b},\pr^{\ast}_{a}\, \omega_{a}\wedge\pr^{\ast}_{b}\, \omega_{b}\right)\, .
\end{equation}

\noindent
Since there exist a morphism of Lie algebras from $\Delta \mathfrak{g}$ to $\mathfrak{g}\oplus\mathfrak{g}$

\begin{eqnarray}
m: \Delta \mathfrak{g} &\to &\mathfrak{g}\oplus\mathfrak{g}\nonumber \\
\left(x,x\right) & \mapsto  & x\oplus x\, ,
\end{eqnarray}

\noindent
we can use it to construct a $L_{\infty}$-algebra morphism

\begin{equation}
F\circ m : \mathfrak{g}\simeq\Delta \mathfrak{g}\to L\left(\mathcal{M}_{a}\times\mathcal{M}_{b},\pr^{\ast}_{a}\, \omega_{a}\wedge\pr^{\ast}_{b}\, \omega_{b}\right)\, ,
\end{equation}

\noindent
which is an homotopy moment map for the action of $\Delta G \simeq G$ on the product manifold $\left(\mathcal{M}_{a}\times\mathcal{M}_{b},\pr^{\ast}_{a}\, \omega_{a}\wedge\pr^{\ast}_{b}\, \omega_{b}\right)$.

%%%%%%%%%%%%%%%%%%%%%%%%%%%%%%%%%%%%%%%%%%%%%%%%%%%%%%%%%%%%%%%%%%%%%%%%%%%%%%%%%%%%%%%%%%%%%%%%%%%

\item $G_{a} = G_{b} = G\, , \quad \mathcal{M}_{a} =\mathcal{M}_{b} = \mathcal{M}\, , \quad \omega_{a} = \omega_{b} = \omega$.

We begin by proving the following lemma, which will be used in a moment

\lemma{\label{lemma:NenMGinvariant} Let $\left(\mathcal{M},\omega\right)$ be an $n$-plectic manifold equipped with the Hamiltonian action of a Lie group $G$ and the corresponding homotopy moment map $f:\mathfrak{g}\to L\left(\mathcal{M},\omega\right)$, where $\mathfrak{g}$ is the Lie algebra of $G$. Let $\mathcal{N}\overset{i}{\hookrightarrow} \mathcal{M}$ a $G$-invariant submanifold of $\mathcal{M}$. Then, assuming that $i^{\ast}\omega$ is a non-degenerate $(n+1)$-form on $\mathcal{N}$, the action $G\circlearrowleft \left(\mathcal{N},i^{\ast}\omega\right)$ is Hamiltonian with homotopy moment map $i^{\ast} f: g\to L\left(\mathcal{N},i^{\ast}\omega\right)$.}

\proof{We need to prove that $f^{\mathcal{N}} = i^{\ast} f$ is indeed an homotopy moment map for the action of $G$ on $\left(\mathcal{N},i^{\ast}\omega\right)$. That is, we have to prove that the following diagram commutes

\begin{center}
\begin{tikzpicture}
\label{diag:commutativemorphismIII}
  \matrix (m) [matrix of math nodes,row sep=8em,column sep=9em,minimum width=2em]
  {& L\left(\mathcal{N},i^{\ast}\omega\right) \\
\mathfrak{g} & \mathfrak{X}_{\mathrm{Ham}}\left(\mathcal{N},i^{\ast}\omega\right) \\};
  \path[-stealth]
    (m-2-1) edge node [above] {$f^{\mathcal{N}}$} (m-1-2)
    (m-2-1) edge node [above] {$v^{\mathcal{N}}_{-}$} (m-2-2)
    (m-1-2) edge node [right] {$\pi$} (m-2-2);
\end{tikzpicture}
\end{center}

\noindent
such that

\begin{equation}
\label{eq:conditionhomotopyN}
-i_{v^{\mathcal{N}}_{x}} i^{\ast}\omega = d f^{\mathcal{N}}_{1} (x)\, ,\qquad\forall\,\, x\in \mathfrak{g}\, ,
\end{equation}

\noindent
where

\begin{equation}
f^{\mathcal{N}} = i^{\ast} f :\mathfrak{g}\to L\left(\mathcal{N},i^{\ast}\omega\right)\, .
\end{equation}

\noindent
Let us introduce the following $L_{\infty}$-subalgebra of  $L\left(\mathcal{M},\omega\right)$

\begin{equation}
L^{\mathcal{N}}\left(\mathcal{M},\omega\right) = C^{\infty}\left(\mathcal{M}\right)\to\Omega^{1}\left(\mathcal{M}\right)\to\cdots\to
\tilde{\Omega}^{n-1}_{\mathrm{Ham}}\left(\mathcal{M}\right)\, ,
\end{equation}

\noindent
where 

\begin{equation}
\tilde{\Omega}^{n-1}_{\mathrm{Ham}}\left(\mathcal{M}\right) = \left\{\alpha\in\Omega^{n-1}_{\mathrm{Ham}}\left(\mathcal{M}\right)\,\, /\,\, v_{\alpha}|_{\mathcal{N}}\in T\mathcal{N}\right\}\, .
\end{equation}

\noindent
Since $L\left(\mathcal{M},\omega\right)$ and $L^{\mathcal{N}}\left(\mathcal{M},\omega\right)$ are equal in every component but the zero one, in order to see that $L^{\mathcal{N}}\left(\mathcal{M},\omega\right)$ is actually a $L_{\infty}$-subalgebra of $L\left(\mathcal{M},\omega\right)$, we only have to check that the binary bracket $l_{2}$ of $L\left(\mathcal{M},\omega\right)$ restricts to $\tilde{\Omega}^{n-1}_{\mathrm{Ham}}\left(\mathcal{M}\right)$. This is indeed the case since given any two Hamiltonian forms $\alpha$ and $\beta$ we have

\begin{equation}
\left[ v_{\alpha},v_{\beta}\right] = v_{\left\{\alpha,\beta\right\}}\, ,
\end{equation}

\noindent
and of course the Lie bracket of two vector fields tangent to $\mathcal{N}$ is again tangent to $\mathcal{N}$. Please notice now that 

\begin{equation}
\label{eq:fginLN}
f_{k}(x)\in L^{\mathcal{N}}\left(\mathcal{M},\omega\right)\, ,\qquad \forall\,\, x\in\mathfrak{g}^{\otimes k}\qquad k\geq 1\, .
\end{equation}

\noindent
Actually, equation \eqref{eq:fginLN} trivially holds for every $k>1$ since $L\left(\mathcal{M},\omega\right)$ and $L^{\mathcal{N}}\left(\mathcal{M},\omega\right)$ are equal in every component but the zero one. We have to check then only the $k=1$ case, that is, we have to prove that 

\begin{equation}
\label{eq:fginLN1}
f_{1}(x)\in \tilde{\Omega}^{n-1}_{\mathrm{Ham}}\left(\mathcal{M}\right)\, ,\qquad \forall\,\, x\in\mathfrak{g}\, .
\end{equation}

\noindent
By assumption, the action of $G$ on $\left(\mathcal{M},\omega\right)$ is through Hamiltonian vector fields. That is, the infinitesimal action of $G$ on $\left(\mathcal{M},\omega\right)$ is generated by Hamiltonian vector fields respect to $\omega$. This fact is encoded in the map $v_{-}$ of diagram \ref{diag:homotopymoment} which assigns to every infinitesimal element $x$ of $G$, which belongs to $\mathfrak{g}$, the Hamiltonian vector field $v_{x}$ that generates the corresponding action. Since $\mathcal{N}$ is a $G$-invariant submanifold of $\mathcal{M}$ we have that

\begin{equation}
v_{x}|_{\mathcal{N}} \in T\mathcal{N}\, ,
\end{equation}

\noindent
since otherwise there would exist at least one point $q\in\mathcal{N}$ such that its orbit under the action of $G$ is not contained in $\mathcal{N}$. Finally, since $v_{x} = v_{f_{1}(x)}$\footnote{We are slightly abusing the notation by calling by the same letter the Hamiltonian vector field that corresponds to a given Hamiltonian form and the Hamiltonian vector field that corresponds to a given element of the Lie algebra $\mathfrak{g}$.} thanks to equation (\ref{eq:condicionham}), we obtain equation (\ref{eq:fginLN1}). Thus we conclude that

\begin{equation}
\label{eq:fginLN2}
f\maps \mathfrak{g}\to  L^{\mathcal{N}}\left(\mathcal{M},\omega\right)\subseteq L\left(\mathcal{M},\omega\right)\, .
\end{equation}

\noindent
We are ready now to first prove that $f^{\mathcal{N}} = i^{\ast} f$ is an honest $L_{\infty}$-morphism from $\mathfrak{g}$ to $L\left(\mathcal{N},i^{\ast}\omega\right)$. %Since $f^{\mathcal{N}}\maps \mathfrak{g}\to L\left(\mathcal{N},i^{\ast}\omega\right)$  is the composition of $i^{\ast}\maps L^{\mathcal{N}}\left(\mathcal{M},\omega\right)\to L\left(\mathcal{N},i^{\ast}\omega\right)$ and $f$ is an $L_{\infty}$-morphism, if we prove that $i^{\ast}$ is as well an $L_{\infty}$-morphism we can conclude that $f^{\mathcal{N}}$ is also an $L_{\infty}$-morphism. In fact  
This is trivially true for all the components of $f^{\mathcal{N}}$ except for $f^{\mathcal{N}}_{1}$. Therefore, what we have to check is that $i^{\ast}f_{1}(x)\in \Omega^{n-1}_{\mathrm{Ham}}\left(\mathcal{N},i^{\ast}\omega\right)$ for all $x\in\mathfrak{g}$. That is indeed the case since, given $x\in\mathfrak{g}$, we have

\begin{equation}
-\iota_{v_{f_{1}(x)}}\omega = df_{1}(x)\,\,\Rightarrow -i^{\ast}\left(\iota_{v_{f_{1}(x)}}\omega\right) = i^{\ast}df_{1}(x)\,\,\Rightarrow -\iota_{v_{f_{1}(x)}}\left(i^{\ast}\omega\right) = d\left(i^{\ast}f_{1}(x)\right) \, ,
\end{equation}

\noindent
since $i_{\ast}v_{f_{1}(x)}=v_{f_{1}(x)}$ thanks to equation (\ref{eq:fginLN2}). Therefore, $i^{\ast}f_{1}(x)$ is Hamiltonian respect to $i^{\ast}\omega$ with Hamiltonian vector field given by the restriction of $v_{f_{1}(x)}$ to $\mathcal{N}$. Is easy to see know that $f^{\mathcal{N}}$ satisfies equations (\ref{eq:conditionhomotopyN}), (\ref{eq:main_eq_1}) and (\ref{eq:main_eq_2}) and therefore it is indeed an homotopy moment map for the action of $G$ on $\left(\mathcal{N},i^{\ast}\omega\right)$.
 \qed

}

The product action can be written in this case as 

\begin{equation}
G\times G\circlearrowleft \left(\mathcal{M}\times\mathcal{M},\omega\wedge\omega\right)
\end{equation}

\noindent
In particular, we can restrict this action to the diagonal $\Delta G$ of $G\times G$ acting on the product manifold

\begin{equation}
G\simeq\Delta G\circlearrowleft \left(\mathcal{M}\times\mathcal{M},\omega\wedge\omega\right)\, .
\end{equation}

\noindent
The diagonal $\Delta \mathcal{M}\simeq \mathcal{M}$ of $\mathcal{M}\times\mathcal{M}$ is invariant under the action of $\Delta G$. Therefore, using the obvious inclusion

\begin{eqnarray}
i : \Delta \mathcal{M} \hookrightarrow \mathcal{M}\times\mathcal{M}\, ,
\end{eqnarray}

\noindent
we can define now an action of $G$ on $\mathcal{M}$ as follows

\begin{equation}
\label{eq:Gactionomega2}
G\simeq\Delta G\circlearrowleft \left(\Delta\mathcal{M},i^{\ast}\left(\omega\wedge\omega\right)\right)\simeq \left(\mathcal{M},\omega\wedge\omega\right)\, ,
\end{equation}

\noindent
which, by means of lemma \ref{lemma:NenMGinvariant}, is Hamiltonian with homotopy moment map given by

\begin{equation}
\label{eq:homotopyG2}
i^{\ast}F : \mathfrak{g}\to L\left(\mathcal{M},\omega\wedge\omega\right)\, ,
\end{equation}

\noindent
where $F$ is as in theorem \ref{thm:morphismproduct}. Therefore, we have proven that if an action $G\circlearrowleft (\mathcal{M},\omega)$ is Hamiltonian, then the action $G\circlearrowleft (\mathcal{M},\omega^{n})\, ,\quad n\in\mathbb{N}$ is also Hamiltonian\footnote{This is done by sightly extending the proof above, allowing $\omega_{a}$ and $\omega_{b}$ to be different and so that the restriction of their wedge product to $\Delta\mathcal{M}$ is non-degenerate.}.

\end{itemize}

\cleardoublepage

%%%%%%%%%%%%%%%%%%%%%%%%%%%%%%%%%%%%%%%%%%%%%%%%%%%%%%%%%%%%%%%%%%%%%%%
%%% BIBLIOGRAPHY
%%%%%%%%%%%%%%%%%%%%%%%%%%%%%%%%%%%%%%%%%%%%%%%%%%%%%%%%%%%%%%%%%%%%%%%

\renewcommand{\leftmark}{\MakeUppercase{Bibliography}}
\phantomsection
\addcontentsline{toc}{chapter}{References}
\bibliographystyle{ThesisStyle}
\bibliography{C:/Users/cshabazi/Dropbox/Referencias/References}
\label{biblio}
\clearpage

\end{document}